\documentclass[12pt]{article}

\usepackage{graphicx}%
\usepackage{multirow}%
\usepackage{amsmath}
\usepackage{amssymb}
\usepackage{amsfonts}
\usepackage{amsthm}
\usepackage{mathrsfs}
\usepackage[title]{appendix}
\usepackage{textcomp}
\usepackage{manyfoot}
\usepackage{booktabs}
\usepackage{algorithmicx}
\usepackage{algpseudocode}
\usepackage{listings}
\usepackage[ruled,vlined]{algorithm2e}
\usepackage[dvipsnames]{xcolor}
\usepackage{eurosym} 
\usepackage{tikz}  
\usetikzlibrary{decorations.pathmorphing}
\usetikzlibrary{decorations.pathreplacing,calligraphy}
\usepackage{makeidx}
\usepackage{dsfont}
\usepackage{csquotes}
\usepackage{placeins}
\usepackage[en-US]{datetime2}
\usepackage{import}
\usepackage{stmaryrd}
\usepackage[colors,economics]{optsys}
\usepackage{cases}
\usepackage{multirow}
\usepackage{pifont}
\usepackage{subcaption}
\usetikzlibrary{fadings,shapes.arrows,shadows}   
\usepackage{hyperref}
\hypersetup{hypertexnames=false}
\usepackage{booktabs}
\usepackage{siunitx}

\makeatletter
\newcommand{\vast}{\bBigg@{3}}
\newcommand{\Vast}{\bBigg@{4}}
\makeatother

\usepackage{wasysym} 

\usepackage{camila}

\usepackage{a4wide}
\usepackage{fullpage}
\usepackage{authblk}
\graphicspath{{./Plots30NodesHD/}}

\title{Temporal and Spatial Decomposition for Prospective Studies in Energy Systems under Uncertainty}

\author[1,2]{Camila Mart\'{\i}nez Parra}
\author[2]{Michel De Lara}
\author[2]{Jean-Philippe Chancelier}
\author[3]{Pierre Carpentier}
\author[1]{Jean-Marc Janin}

\affil[1]{R\'eseau de Transport d'\'Electricit\'e, France}
\affil[2]{CERMICS, \'Ecole nationale des ponts et chauss\'{e}es, IP~Paris, France}
\affil[3]{UMA, ENSTA Paris, IP~Paris, France}

\begin{document}

\maketitle

\begin{abstract}
  The increasing penetration of renewable energy requires greater use of storage resources to manage system 
  intermittency. As a result, there is growing interest in evaluating the opportunity cost of stored energy,
  or usage values, which can be derived by solving a multistage stochastic optimization problem. 
  Stochasticity arises from net demand (the aggregation of demand and non-dispatchable generation), 
  the availability of dispatchable generation, and inflows when the storage facilities considered are 
  hydroelectric dams.
  We aim to compute these usage values for each market zone of the
  interconnected European electricity system, in the context of prospective
  studies currently conducted by RTE, the French TSO.  The energy system is
  mathematically modelled as a directed graph, where nodes represent market
  zones and arcs represent interconnection links. In large energy systems,
  spatial complexity (thirty nodes in the system, each with at most one
  aggregated storage unit) compounds temporal complexity (a one-year horizon
  modelled with two timescales: weekly subproblems with hourly time steps).
  This work addresses three main sources of complexity: temporal, spatial, and
  stochastic. We tackle the multinode multistage stochastic optimisation problem
  by incorporating a spatio-temporal decomposition scheme. To efficiently
  compute usage values, we apply Dual Approximate Dynamic Programming (DADP),
  which enables tractable decomposition across both time and space.  This
  approach yields nodal usage values that depend solely on the local state of
  each node, independently of the others.
  We conduct numerical studies on a realistic system composed of thirty nodes
  (modelling part of Europe) and show that DADP obtains competitive results when
  comparing with traditional methods like Stochastic Dual Dynamic Programming
  (SDDP).
\end{abstract}


\section{Introduction}\label{sec1}

In~\S\ref{Context_and_motivation}, we present the context of this work.
In~\S\ref{Mathematical_formulations_and_numerical_challenges}, we outline the specificities of
the mathematical formulations and the numerical challenges.
In~\S\ref{Contributions_and_outline_of_the_paper}, we stress the contributions
and we present the organization of the paper.

\subsection{Context and motivation}
\label{Context_and_motivation}

    This article is issued from a PhD thesis at {RTE} (Réseau de Transport d'Électricité), 
    the French electricity transmission system operator, responsible for managing, 
    operating, and developing the national high-voltage grid.
    Building on its operational and infrastructure responsibilities, {RTE} also plays a crucial role in 
    long-term energy planning and analysis. Its well-established commitment to prospective energy 
    studies complements its day-to-day management of the electricity system, and contributes to informed 
    decision-making at both national and European levels. 
    In France, RTE is required to publish long-term electricity outlooks (e.g., the biennial adequacy report~\cite{RTE_BP2023})
     and carries out broader foresight studies (such as \emph{Futurs énergétiques 2050}~\cite{RTE_FE2050})
      that assess demand, generation, interconnections and policy-driven pathways, providing independent
       scenario-based analysis for authorities, market participants and the public.

    Within these prospective studies, particular attention is paid to the role of storage facilities 
    in the evolving electricity system. Storage is increasingly recognized as a key tool to mitigate
     the variability of wind and solar generation by shifting energy across hours, days, and seasons, 
     and by providing fast reserves to ensure system security. Consequently, assessing the deployment 
     and operation of storage—both as a technical asset and as a decision variable in long-term 
     scenarios—has become an important component of {RTE}'s modelling and policy advice, as it directly 
     impacts system adequacy, flexibility, and the integration of renewable production.

    {RTE} uses {Antares}~\cite{Antares} to conduct prospective studies. 
    {Antares} simulates the operation of the power system over a long-term horizon 
    (typically one year with an hourly time step), modelling generation, consumption, 
    storage, transmission lines, and their interactions. 
    At the core of these simulations lies the
    \emph{resource adequacy problem}: a multistage optimization from the perspective of a 
    central planner that allocates dispatchable resources hour-by-hour, using an economic dispatch
    to maximize the social economic welfare. 
    In this setting, 
    production costs are assumed input data, and resources are dispatched in merit order, i.e., 
    the cheapest units are used before the more expensive ones. 
    
    Because electrical demand, 
    renewable generation, and thermal-unit availability are uncertain, 
    the resource adequacy problem is \emph{stochastic}. 
    When the problem accounts for different market zones connected by interconnections,
    it becomes \emph{spatially coupled}, allowing for exchanges between zones.
    This problem is naturally formulated on \emph{two timescales}, weekly and
    hourly, both coordinated for the coming week and respecting the hourly energy balances.
    Storage facilities induce \emph{temporal couplings} on both timescales,
    as the energy stored at a given time depends on previous charging and discharging decisions,
    and constrains future decisions by storage capacity limits.
    
    In this context, the question of when stored energy will be used arises. 
    \emph{Usage values} correspond to the marginal value of stored energy (storage prices):
    a price signal that allows choosing when and how much energy is used, depending on the energy
    system setting.
    %

\sloppy
\subsection{Mathematical formulations and numerical challenges}
\label{Mathematical_formulations_and_numerical_challenges}

    Mathematically, the resource adequacy problem is a large-scale,
    multistage stochastic optimization problem ({MSOP}) with two timescales, 
    spatial couplings due to energy exchanges through interconnections,
    and temporal couplings induced by storage facilities.
    Informally, a {MSOP} models sequential decision making under uncertainty over a finite timespan:
    at each stage, the decision maker ({DM}) observes the available information, 
    takes a control action, incurs a stage cost, and the system state evolves according 
    to possibly random dynamics. Decisions must be nonanticipative 
   (adapted to the information available at the time of the decision) and satisfy stagewise feasibility constraints. 
    The objective is to determine a policy that minimizes the expected sum of stage costs 
    (plus any terminal cost) subject to the stochastic dynamics and feasibility requirements.

    As our goal is to compute usage values, we are not primarily interested in solving the {MSOP} directly,
    but in applying temporal decomposition methods that compute cost-to-go functions from which usage values can be derived.
    Under certain assumptions, {MSOP}s can be temporally decomposed
    using dynamic programming
    techniques that break the problem into one-stage subproblems according to Bellman's principle of
    optimality~\cite{Bellman:1957},~\cite{Bertsekas-Shreve:1996}.
    To this end, cost-to-go functions are introduced: state-dependent functions at each stage of the problem
    that give the optimal cost incurred from that stage to the horizon.
    Usage values are derived from these cost-to-go functions.
    Even if, theoretically, the temporal decomposition of large-scale {MSOP}s is possible, 
    the curse of dimensionality~\cite{Bellman:1957} makes it intractable in practice when 
    the state space (number of storage facilities) is large (a few units), as is the case in the resource adequacy 
    problem in the European domain.
    
    The problem we are interested in solving is challenging because
    it exhibits three axes of complexity, as follows.
    \begin{itemize}
        \item[\textbf{1.}] \textbf{Temporal.}
        A two-timescale multistage optimisation problem that models 52~weeks of 168~hours each, in which 
        hourly and weekly decisions are coupled through the storage dynamics equations (storage energy balances).

        \item[\textbf{2.}] \textbf{Spatial.}
        Multiple market zones (nodes) connected by representative interconnections (arcs) compose 
        the spatial dimension; these zones are coupled through the {node/arc balancing equations},
        which enforce the supply–demand balance for the whole system.

        \item[\textbf{3.}] \textbf{Stochastic.}
        At each hour, demand, the availability of dispatchable units, and storage inflows are uncertain.
        Information constraints model the available 
        information to the decision maker at each stage of the multistage optimisation problem,
        that is, decisions should be nonanticipative.
\end{itemize}

    In this work, we address the full complexity of the problem, introducing
    spatial decomposition methods to be applied alongside temporal decomposition methods.
    
    \subsection{Contributions and outline of the paper}
    \label{Contributions_and_outline_of_the_paper}
    
    This work makes the following contributions to the field of two-timescale 
    multistage stochastic optimization
    problems, with a focus on applications in prospective studies of 
    energy systems.

        We adapt the Dual Approximate Dynamic Programming ({DADP}) method~\cite{pacaud2021distributed}
        to two-timescale multistage stochastic optimization problems under different information structures,
        presenting the interactions between spatial and temporal decompositions.
        {DADP}
        consists of using a Lagrangian relaxation of the energy balancing spatial coupling constraint, 
        to get a lower bound of the global optimisation problem.
        When relaxing the coupling constraint,
        we introduce a Lagrangian multiplier~$\price$, the \emph{price decomposition process},
        {that we consider in the deterministic class}.
        Once the coupling constraint is relaxed, it is possible to solve independently the nodal production problems
        and the transport problem in the arcs, by Stochastic Dynamic Programming ({SDP}) at each node --- where the nodal state is the
        level in the local storage (made possible as the {price decomposition process}~$\price$ is
        a deterministic vector).
        We also provide practical guidelines for selecting the price decomposition process
        in the {DADP} method, detailing the trade-off between computational efficiency and
        accuracy in estimating cost-to-go functions and usage values for storage.
        The main numerical interest in the {DADP} method 
        is that, as a consequence of the spatial decomposition, 
        the nodal problems are independent and can be solved using univariate SDP during the
        price decomposition process selection. 
        Once the price decomposition process is selected,
        the we use the sum of the nodal cost-to-go functions to approximate the global cost-to-go functions,
        and use them to compute policies for the global problem.

        \bigskip

        The paper is organized as follows.
        
        In Section~\ref{section:Multistage_multinode_stochastic_optimisation_problem}, 
        we introduce the two-timescale multistage stochastic optimization problem with spatial coupling
        constraints, providing a global formulation of the problem using local formulations for each node. 

        In Section~\ref{section:Mixing_spatial_and_temporal_decomposition},
         we present the {DADP} method adapted to
        two-timescale multistage stochastic optimization problems, detailing the interactions between
        spatial and temporal decompositions, as well as different structures for the price decomposition process (multiplier).
        We also discuss how to improve this price decomposition process. By comparison with
        \cite{pacaud2021distributed}, we use a Bellman-like algorithm to produce proxies of
        gradients~\cite{franc2023differentiabilityregularizationparametricconvex}
        to be used in a gradient-like method.

    In Section~\ref{section:Modelling_and_numerical_study_in_multinode_systems}, 
    we benchmark the {DADP} method for
    two-timescale multistage stochastic optimization problems, using different price decomposition process structures, 
     with Stochastic Dual Dynamic Programming ({SDDP})~\cite{Pereira-Pinto:1991}.
    We conduct a large-scale numerical study of the resource adequacy problem
    considering the central-western European power system at high
    spatial resolution, representing 30 market zones — 20 of which include storage facilities —
    over a one-year timespan, with a weekly temporal decomposition and hourly decisions and constraints.

    Finally, Section~\ref{section:Conclusion_and_perspectives} concludes the article
    and discusses future research directions.

\section{Multistage multinode stochastic optimisation problem}
\label{section:Multistage_multinode_stochastic_optimisation_problem}

In this section, we introduce 
the two-timescale multistage stochastic optimization problem
with spatial coupling constraints.
In~\S\ref{subsection:Mathematical_formulation_of_the_problem}, 
we define the mathematical notations and present the physical and economic 
modelling of the multinode energy system.
In~\S\ref{subsection:Global_optimisation_problem_formulation}, 
we state the global formulation of the two-timescale multistage stochastic optimization problem
with spatial coupling constraints,
defining local formulations for each node and the transport problem.
In~\S\ref{subsection:Temporal_decomposition}, we detail the temporal decomposition of the global problem.

\subsection{Definitions and notations}
\label{subsection:Mathematical_formulation_of_the_problem}

In~\S\ref{subsubsection:Timescale_graph_and_variables_definition}, 
we define the timescales, the directed graph modelling the multinode system,
and the decision, uncertain, and state variables.
In~\S\ref{subsubsection:Physical_and_economic_constraints_in_a_multi_node_system},
we present the physical and economic modelling of the multinode energy system.
In~\S\ref{subsubsection:Random_variables_multi_node_system},
we introduce the stochastic modelling of the uncertainties and the information constraints.

\subsubsection{Timescale, graph and variables definition}
\label{subsubsection:Timescale_graph_and_variables_definition}

In this part, we define the timeline with two timescales. 
We then introduce the directed graph that models the multinode system. 
For each node, we define the nonanticipative control variables, the uncertain variables, 
the recourse control variables, and the storage level variables. 
We also introduce the nodal flow variables, representing energy imports and exports at each node,
as well as the flow variables associated with the arcs of the directed graph.

\phantomsection

\paragraph{Timescale definition.}

We consider a timeline with a long timescale and a short timescale.
The short and long timescales could be any two scales, as long as one is larger than the other.
In this work, the long timescale is given by weeks that are represented by a finite totally ordered set
$\np{\WEEK, \preceq}$,
where $\week\successor$ is the successor of $\week\in \WEEK$  and $\week\predecessor$
its predecessor: \(\binf{\week} \prec \cdots \prec \week\predecessor \prec \week
\prec \week\successor \prec \cdots \prec \bsup{\week}\) (where $\prec$ is the strict order associated with the order $\preceq$).
Then, $\WEEK = \nc{\binf{\week}, \bsup{\week}}$.
The short timescale, hours in this case, is represented by a finite totally ordered set $\np{\HOUR, \preceq}$:
\(\binf{\hour} \prec  \cdots  \prec \hour\predecessor \prec\hour \prec\hour\successor \prec \cdots
\prec \bsup{\hour}\).
Then, $\HOUR = \nc{\binf{\hour}, \bsup{\hour}}$.

 To unify the timescale we consider the product set $\timeline$ ordered as follows:
\begin{align}
  \label{eq:timeline}
  \winfhinf&
          \prec\cdots\prec \np{\week\predecessor,\bsup{\hour}}
          \prec \np{\week, \binf{\hour}}
          \prec \np{\week,\binf{\hour}\successor}\prec \cdots
          \nonumber\\
  \cdots &
           \np{\week, \bsup{\hour}\predecessor}\prec\np{\week,\bsup{\hour}}
           \prec \np{\week\successor,\binf{\hour}}\prec \cdots \prec \wsuphsup
           \eqfinp
\end{align}
We consider a timespan of one year, and $\winfhinf$ is the instant corresponding
to the first hour of the first week of the year, and
$(\bsup{\week},\bsup{\hour})$ is the last hour of the last week
of the year.
We need to define an extra time $\whlast$ at its
end to handle the resulting state of the last decision.
The extended unified timeline $\extimeline$ is defined as $\timeline \cup
\na{\whlast}$.

We define  \(\closedWopen{\week} =\bp{ {\whinf}, {\whinfsuccessor}, \dots,  {\whsup}} \)
and  \( \allowbreak \openWclosed{\week} = \left({\whinfsuccessor}, \dots,  {\whsup},\right.\) \(\left.{\wsuccessorhinf}\right) \).
Thus, we use a simple bracket $\left[\right.$ or $\left.\right]$ to denote intervals of the
elementary timelines $\np{\HOUR, \preceq}$ and  $\np{\WEEK, \preceq}$. By contrast, we use
double brackets $\llbracket$ or $\rrbracket$  for the composite (product) timeline $\np{\extimeline, \preceq}$.
The different possibilities to index a variable (respectively a function) by time are detailed in
Table~\ref{Table:variablesnotationresume}
(respectively in  Table~\ref{table:functionsnotationresume}).
\begin{table}[!ht]
  \renewcommand{\arraystretch}{1.4}
  \centering
\resizebox{0.49\textwidth}{!}{
  \begin{subtable}[t]{0.46\textwidth}
    \centering
    \small
    \begin{tabular}{|c|c|p{4.2cm}|}
      \hline
      Index & Notation & Description \\ \hline \hline
      $\wh$ & $z_{\wh}$ & Variable at $\wh$ \\ \hline
      $\week$ & $z_{\week}$ & Representative variable for the week $s$ corresponding to the variable at $\whinf$ \\ \hline
      $\closedWopen{\week}$ & $z_{\closedWopen{\week}}$ & Sequence of hourly variables given by $\bp{ z_{\whinf}, z_{\whinfsuccessor}, \dots,  z_{\whsup}}$ \\ \hline
      $\openWclosed{\week}$ & $z_{\openWclosed{\week}}$ & Sequence of hourly variables given by $\bp{z_{\whinfsuccessor}, \dots,  z_{\whsup}, z_{\wsuccessorhinf}}$ \\ \hline
    \end{tabular}
    \caption{\small Variables notation}
    \label{Table:variablesnotationresume}
  \end{subtable}}
  \hfill
  \resizebox{0.49\textwidth}{!}{
  \begin{subtable}[t]{0.46\textwidth}
    \centering
    \small
    \begin{tabular}{|c|c|p{4.2cm}|}
      \hline
      Index & Notation & Description \\ \hline \hline
      $\wh$ & $\zeta_{\wh}$ & Function expression at $\wh$ \\ \hline
      $\week$ & $\zeta_{\week}$ & Characteristic aggregation of the hourly functions $\zeta_{\wh}$ for the week $s$ \\ \hline
      $\closedWopen{\week}$ & $\zeta_{\closedWopen{\week}}$ & Sequence of hourly functions $\zeta_{\wh}$ given by $\bp{ \zeta_{\whinf}, \zeta_{\whinfsuccessor}, \dots,  \zeta_{\whsup}}$ \\ \hline
      $\openWclosed{\week}$ & $\zeta_{\openWclosed{\week}}$ & Sequence of hourly functions $\zeta_{\wh}$ given by $\bp{ \zeta_{\whinfsuccessor},  \dots, \zeta_{\whsup}, \zeta_{\wsuccessorhinf}}$ \\ \hline
    \end{tabular}
    \caption{\small Functions notation}
    \label{table:functionsnotationresume}
  \end{subtable}
  }
  \caption{Summary of variable and function notations}
  \label{table:variables_functions_summary}
\end{table}
The characteristic aggregation in Table~\ref{table:functionsnotationresume} could be a sum over~$\hour\in\HOUR$, a composition
with respect to the state, or a combination of both.

\paragraph{Directed graph modelling the multinode system.}
  We model the electrical system as a directed graph  $\graph\np{\NODE, \ARC}$ where a
  node in~$\NODE$ represents a given group of  production units, load, and storage,  
  and an arc in~$\ARC$ represents the link between two nodes of the network. At each node, 
  there may be different types of storage and  more than one of each type. 
  An individual node in $\NODE$ is denoted by $\node$ and an individual  arc in $\ARC$ by $\arc$.

\paragraph{Nodal uncertain variables.}

We consider, for all~$\node\in\NODE$ and for all~$\wh\in\extimeline$, 
the uncertain variable~$\uncertain_{\wh}^{\node} \in \UNCERTAIN_{\wh}^{\node}$
at node~$\node$, {where $\UNCERTAIN_{\wh}^{\node}$ is a measurable space}.
The collection of uncertainties for week $\week$ at 
node $\node$ is denoted by $ \uncertain_{\openWclosed{\week}}^{\node}
        =
            \bp{
        \uncertain_{\whinfsuccessor}^{\node},
            \dots, 
            \uncertain_{\wsuccessorhinf}^{\node}}
        \in 
        \UNCERTAIN_{\openWclosed{\week}}^{\node}
                =
            \UNCERTAIN_{\whinfsuccessor}^{\node}
            \times
            \dots
            \times
            \UNCERTAIN_{\wsuccessorhinf}^{\node}$.
The collection of uncertainties for all
nodes $\node\in \NODE$  by~$
    \uncertain_{\wh}
        =
            \bseqp{\uncertain_{\wh}^{\node}}{\node\in\NODE}    
                \in
            \prod_{\node\in\NODE}\UNCERTAIN_{\wh}^{\node} \eqsepv \forall \wh\in\extimeline$.

\paragraph{Nodal recourse control variables.} 
\label{subsection:Nodal_recourse_control_variables}

We consider, for all~$\node\in\NODE$ and for all~$\wh\in\timeline$,
the recourse control~$\recoursecontrol_{\wh\successor}^{\node} \in
\RECOURSECONTROL_{\wh\successor}^{\node}$
at node~$\node$.
More precisely, during the interval $\ClosedIntervalOpen{\wh}{\wh\successor}$,
the noise $\uncertain_{\wh\successor}$ occurs, 
and a recourse decision $\recoursecontrol_{\wh\successor}$ 
is taken knowing $\uncertain_{\wh\successor}$. 
The collection of recourse controls for week $\week$ at node 
$\node$ is denoted by~$\recoursecontrol_{\openWclosed{\week}}^{\node}
            =
                \bp{
            \recoursecontrol_{\whinfsuccessor}^{\node},
                \dots, 
                \recoursecontrol_{\wsuccessorhinf}^{\node}}
            \in 
            \RECOURSECONTROL_{\openWclosed{\week}}^{\node}
                    =
                \RECOURSECONTROL_{\whinfsuccessor}^{\node}
                \times
                \dots
                \times
                \RECOURSECONTROL_{\wsuccessorhinf}^{\node}$.
The collection of recourse controls for all nodes
$\node\in \NODE$ by~$ \recoursecontrol_{\wh\successor  }
        =
            \bseqp{\recoursecontrol_{\wh\successor}^{\node}}{\node\in\NODE}    
                \in
            \prod_{\node\in\NODE}\RECOURSECONTROL_{\wh\successor}^{\node}
            \eqsepv \forall \wh\in\timeline$.

\paragraph{Nodal storage level variables.}

We consider, for all~$\node\in\NODE$ and for all~$\wh\in\extimeline$, 
the storage level variable~$\stock_{\wh}^{\node} \in \STOCK_{\wh}^{\node}$
at node~$\node$.
In general, $\STOCK_{\wh}^{\node} \subseteq \RR_+^d$, with $d$ 
equal to the number of storages at the node $\node$.
The collection of storage level variables for week $\week$ at node $\node$ is denoted 
by~$ \stock_{\openWclosed{\week}}^{\node}
         = \bp{
             \stock_{\whinfsuccessor}^{\node},
                 \dots, 
                 \stock_{\wsuccessorhinf}^{\node}}
         \in 
         \STOCK_{\openWclosed{\week}}^{\node}
         =\STOCK_{\whinfsuccessor}^{\node}
             \times
             \dots
             \times
             \STOCK_{\wsuccessorhinf}^{\node}$.
The collection of storage level variables for all nodes
    $\node\in \NODE$ by~$
    \stock_{\wh}
        =
            \bseqp{\stock_{\wh}^{\node}}{\node\in\NODE}    
                \in
            \prod_{\node\in\NODE}\STOCK_{\wh}^{\node} \eqsepv \forall \wh\in\extimeline$.

\paragraph{Nodal flows: import and export variables.}

Each node~$\node$ is capable of importing or exporting through the arcs
connected to it. We consider for all~$\node\in\NODE$ and for all~$\wh\in\timeline$,
the flow variable~$\flownode_{\wh\successor}^{\node} \in \RR$ in node~$\node$.
We denote the collection of flow variables for week $\week$ at node $\node$ 
 by~$ \flownode_{\openWclosed{\week}}^{\node}
            = \bp{
                \flownode_{\whinfsuccessor}^{\node},
                    \dots, 
                    \flownode_{\wsuccessorhinf}^{\node}}
            \in 
           \RR^{\cardinality{\HOUR}} $.
We denote the collection of resulting flow variables for all nodes
$\node\in \NODE$  by~$ 
 \flownode_{\wh\successor}
        =
            \bseqp{\flownode_{\wh\successor}^{\node}}{\node\in\NODE}    
                \in
            \RR^{\cardinality{\NODE}} \eqsepv \forall \wh\in\timeline$.
We denote by $\flownode$ the collection of nodal flows for all hours in the year
and all nodes in the system, that is,~$
    \flownode
        =
            \bseqp{\flownode_{\wh\successor}}{\wh\in\timeline}
                \in
            \RR^{\cardinality{\HOUR}\times\cardinality{\WEEK}\times\cardinality{\NODE}}$.

By convention,  $\flownode^{\node}<0$ denotes an exportation (additional demand)
from the node $\node$ and $\flownode^{\node}>0$ denotes an importation (additional production)
to the node $\node$.
From a modelling point of view, it is natural to consider the import/export flows
as {recourse} variables, since they are determined after the uncertainties occur.

\paragraph{Flows through the arcs.}

We consider, for all~$\arc\in\ARC$ and for all~$\wh\in\timeline$,
the arc flow variable~$\flowarc_{\wh\successor}^{\arc} \in \RR$
in the arc~$\arc$.
We denote the collection of arc flow variables for week $\week$ at arc $\arc$  by
$ \flowarc_{\openWclosed{\week}}^{\arc}
            = \bp{
                \flowarc_{\whinfsuccessor}^{\arc},
                    \dots, 
                    \flowarc_{\wsuccessorhinf}^{\arc}}
            \in 
            \RR^{\cardinality{\HOUR}} $.
We denote
the collection of arc flow variables for all arcs
$\arc\in \ARC$  by~$
    \flowarc_{\wh\successor}
        =
            \bseqp{\flowarc_{\wh\successor}^{\arc}}{\arc\in\ARC}    
                \in
            \RR^{\cardinality{\ARC}} \eqsepv \forall \wh\in\timeline$.
We denote by $\flowarc$ the collection of arc flows for all hours in the year
and all arcs in the system, that is,~$
    \flowarc
        =
            \bseqp{\flowarc_{\wh\successor}}{\wh\in\timeline}
                \in
            \RR^{\cardinality{\HOUR}\times\cardinality{\WEEK}\times\cardinality{\ARC}}$.

If the arc~$\arc$ is directed from node $i$ to node  $j$, by convention, 
$\flowarc^{\arc} > 0$ is a flow from node $i$ to node $j$, and
$\flowarc^{\arc} < 0$ is a flow from node $j$ to node $i$.
From a modelling point of view, it is natural to consider the flows through the arcs 
as {recourse} variables, since they are determined after the uncertainties occur.

\subsubsection{Physical and economic constraints in a multinode system}
\label{subsubsection:Physical_and_economic_constraints_in_a_multi_node_system}
This~\S\ref{subsubsection:Physical_and_economic_constraints_in_a_multi_node_system}
 introduces the main physical and economic constraints for the multinode system. 
We first describe the nodal storage dynamics and the nodal energy balance. 
Next, we present the spatial coupling constraint that couples the nodes and arcs of the directed graph. 
We then specify the box constraints for the nodal and arc variables. 
Finally, we define the nodal instantaneous cost and the transportation cost associated with the arcs.

\phantomsection
\paragraph{Nodal storage dynamics.}
 We consider a family~$\bseqp{\nseqp{\dynamics_{\wh}^{\node}}{\node\in\NODE}}{\wh\in\timeline}$ of 
nodal dynamics measurable mappings~$  \dynamics_{\wh}^{\node}: 
    \STOCK_{\wh}^{\node}
        \times 
    \UNCERTAIN_{\wh\successor}^{\node}  
    \times 
    \RECOURSECONTROL_{\wh\successor}^{\node}
  \to \STOCK_{\wh\successor}^{\node}$, 
that respresent the nodal storage evolution at the hourly timescale, 
for all time $\wh\in\timeline$
and all node~$\node\in \NODE$, given by
\begin{subequations} \label{eq:dynamics}
\begin{equation}\label{eq:HourlyDynamicsNodal}
     \stock_{\wh\successor}^{\node} =  \dynamics_{\wh}^{\node}\bp{
                                        \stock_{\wh}^{\node},
                                        \uncertain_{\wh\successor}^{\node},
                                        \recoursecontrol_{\wh\successor}^{\node}
                                    }
                                    \eqfinp
\end{equation}

Each hourly storage level $\bseqp{\stock_{\wh\successor}^{\node}}{\hour\in\HOUR}$ 
is obtained from the stock level 
  composition of the hourly dynamics $\dynamics_{\wh}^{\node}$    from the first hour of the week.
  This composition is made only with respect to the storage levels. 
In particular, weekly dynamics~$\dynamics_{\week}^{\node}:
    \STOCK_{\whinf}^{\node}
    \times
    \UNCERTAIN_{\openWclosed{\week}}^{\node}
        \times 
    \RECOURSECONTROL_{\openWclosed{\week}}^{\node}
    \to \STOCK_{\wsuccessorhinf}^{\node}
$ for the node $\node$,
\begin{equation}\label{eq:WeeklyDynamicsNodal}
   \stock_{\wsuccessorhinf}^{\node}
   =
    \dynamics_{\week}^{\node}\Bp{
                                        \stock_{\whinf}^{\node},
                                        \uncertain_{\openWclosed{\week}}^{\node},
                                        \recoursecontrol_{\openWclosed{\week}}^{\node}
                                    }
    \eqfinv
\end{equation}
is obtained
by composition of the nodal hourly dynamics $\dynamics_{\wh}^{\node}$ for $\hour\in\HOUR$ as follows.
For each node~$\node$, 
 from the storage level~$\stock_{\whinf}^{\node}$
at the first hour of the current week and at
node~${\node}$, 
the collection of 
recourse controls
$\recoursecontrol_{\openWclosed{\week}}^{\node}$ for the week $\week$ at node
${\node}$, and  the collection  of uncertainties $\uncertain_{\openWclosed{\week}}^{\node}$ for the week $\week$ at node
${\node}$,  we obtain the storage level at the first hour of the following week~$\stock_{\wsuccessorhinf}^{\node}$.
 \end{subequations}
We stress that,   in Equation~\eqref{eq:HourlyDynamicsNodal} and Equation~\eqref{eq:WeeklyDynamicsNodal}, 
  the dynamics at node~$\node$ depend only on the storage levels,  controls, 
  and uncertainties at node~$\node$.

 Notice that, while in the hourly approach in Equation \eqref{eq:HourlyDynamicsNodal} the 
 uncertain term is $\uncertain_{\wh\successor}^{\node}$, 
 that is, the uncertainty in the interval between
$\wh$ and ${\wh}\successor$, in the weekly approach in Equation  \eqref{eq:WeeklyDynamicsNodal} 
the uncertainty term is \( \uncertain_{\openWclosed{\week}}^{\node} = \np{ 
\uncertain_{\couple{\week}{\binf{\hour}\successor}}^{\node}, \dots, 
\uncertain_{\couple{\week\successor}{\binf{\hour}}}^{\node}}\), that is, the uncertainty
during the current week (which is the time interval considered in this approach).

\paragraph{Nodal energy balance.}
We consider a family~$\bseqp{\nseqp{\nodeBalance_{\wh}^{\node}}{\node\in\NODE}}{\wh \in\timeline}$ 
of nodal balance measurable functions at the hourly timescale given,
for all time $\wh \in \timeline$ and for all node $\node\in\NODE$, by 
functions~$  \nodeBalance_{\wh}^{\node}:
        \UNCERTAIN_{\wh\successor}^{\node}
        \times
        \RECOURSECONTROL_{\wh\successor}^{\node}
    \to  \RR$.
\begin{subequations}
The hourly balance constraint is
\begin{align}
    \nodeBalance_{\wh}^{\node}\Bp{
                    \uncertain_{\wh\successor}^{\node}, 
                    \recoursecontrol_{\wh\successor}^{\node}
                                 } =  \flownode_{\wh\successor}^{\node} 
                                 \eqfinp
\end{align}
We write a weekly version of the nodal energy balance considering the 
functions $ \nodeBalance_{\closedWopen{\week}}:
        \UNCERTAIN_{\openWclosed{\week}}^{\node}
        \times
        \RECOURSECONTROL_{\openWclosed{\week}}^{\node}    
    \to 
    \RR^{\cardinal{\HOUR}}$,
    {obtained by concatenation of the hourly balance functions}.
The weekly balance constraint for all~$\week \in \WEEK$ and for all~$\node\in\NODE$ is
\begin{align}
  \label{eq:WeeklyNodalBalanceConstraint}
\nodeBalance_{\closedWopen{\week}}^{\node}\bp{
                            \uncertain_{\openWclosed{\week}}^{\node},
                            \recoursecontrol_{\openWclosed{\week}}^{\node} 
                        }
           = \flownode_{\openWclosed{\week}}^{\node} 
           \eqfinp          
\end{align}
\end{subequations}
Note that the weekly nodal balance in Equation \eqref{eq:WeeklyNodalBalanceConstraint}
is obtained as the collection of the hourly nodal balances for all hours in the week.

\paragraph{Spatial coupling constraint: Kirchhoff's law.}
The node flows $\flownode_{\wh}^{\node}$ and the flows through the arcs
$\flowarc_{\wh}^{\arc}$ are related by Kirchhoff’s law (nodal energy balance at each hour).

To write the spatial coupling constraint, we consider the node-arc incidence
matrix~$\NodeArcMatrix \in \na{-1, 0, 1}^{\NODE \times \ARC}$ of the directed $\graph\np{\NODE,\ARC}$. 
The matrix~$\NodeArcMatrix$ has $\cardinality{\NODE}$ rows and 
$\cardinality{\ARC}$ columns. Each entry of the matrix indicates if the arc at column $\arc$ is
heading from ($\NodeArcMatrix_{\couple{\node}{\arc}}=1$) or to 
($\NodeArcMatrix_{\couple{\node}{\arc}}=-1$) the node in row $\node$ or not connected to it
($\NodeArcMatrix_{\couple{\node}{\arc}}=0$).
We write the balance equation as~$  \flownode_{\wh\successor}
        -   \NodeArcMatrix \flowarc_{\wh\successor} 
          = 0 \eqsepv 
        \forall \wh \in \timeline$,
or, with the global definitions
of the node flows and arc flows in~\S\ref{subsubsection:Timescale_graph_and_variables_definition}, as
\begin{equation}
    \label{eq:CouplingConstraint}
        \flownode
        -\NODEARCMATRIX \flowarc 
         = 0 \eqsepv
   \end{equation}
where $\NODEARCMATRIX$ is a block diagonal matrix with the node-arc incidence 
matrix $\NodeArcMatrix$ as its diagonal blocks.

\paragraph{Box constraints in nodal and arc variables.}
When we use the notation $\bc{\binf{z},\bsup{z}}$ for vectors $\binf{z}, \bsup{z} \in \RR^d$, 
that means the box \(\prod^d_{i=1} \bc{\binf{z}_i,\bsup{z}_i}\).

We consider that the nodal storage level variables,  the nodal recourse control variables 
and the arc flow variables are bounded as follows (with obvious notation for the bounds).
The box constraints are given in the hourly scale for 
all~$\wh \in \timeline$:~$  \stock_{\wh\successor}^{\node} \in
            \bc{\binf{\stock}_{\wh\successor}^{\node},
            \bsup{\stock}_{\wh\successor}^{\node}}
            \subset \STOCK_{\wh\successor}^{\node}$,$\recoursecontrol_{\wh\successor}^{\node} \in
            \bc{\binf{\recoursecontrol}_{\wh\successor}^{\node},
            \bsup{\recoursecontrol}_{\wh\successor}^{\node}}
            \subset \RECOURSECONTROL_{\wh\successor}^{\node}$,
and~$  \flowarc_{\wh\successor}^{\arc} \in
            \bc{\binf{\flowarc}_{\wh\successor}^{\arc},
            \bsup{\flowarc}_{\wh\successor}^{\arc}}
            \subset \RR$.
Equivalently, we write the box constraints in a weekly vectorized form,
for all $\week \in \WEEK$, as
 \begin{subequations} \label{eq:box_weekly}
        \begin{align}
        \stock_{\openWclosed{\week}}^{\node} &\in
            \bc{\binf{\stock}_{\openWclosed{\week}}^{\node},
            \bsup{\stock}_{\openWclosed{\week}}^{\node}}
            \subset \STOCK_{\openWclosed{\week}}^{\node} \label{eq:StorageLevelBounds_weekly} \eqfinv \\
        \recoursecontrol_{\openWclosed{\week}}^{\node} &\in
            \bc{\binf{\recoursecontrol}_{\openWclosed{\week}}^{\node},
            \bsup{\recoursecontrol}_{\openWclosed{\week}}^{\node}}
            \subset \RECOURSECONTROL_{\openWclosed{\week}}^{\node} \label{eq:RecourseControlBounds_weekly} \eqfinv  \\
        \flowarc_{\openWclosed{\week}}^{\arc} &\in
            \bc{\binf{\flowarc}_{\openWclosed{\week}}^{\arc},
            \bsup{\flowarc}_{\openWclosed{\week}}^{\arc}}
            \subset \RR^{\cardinality{\HOUR}} \label{eq:ArcFlowBounds_weekly} \eqfinp
        \end{align}
\end{subequations}

\paragraph{Nodal instantaneous cost.}
 We consider a family~$\bseqp{\nseqp{ \InstantaneousCost_{\wh}^{\node}}{\node\in\NODE}}{\wh\in\timeline}$
 of nodal hourly instantaneous cost functions~$ \InstantaneousCost_{\wh}^{\node}: 
            \STOCK_{\wh}^{\node}
             \times 
            \UNCERTAIN_{\wh\successor} ^{\node}
                \times 
            \RECOURSECONTROL_{\wh\successor}^{\node}   
    \to 
            \RR \cup \na{+ \infty}$,
that represent the nodal cost 
$\InstantaneousCost_{\wh}^{\node} \bp{
        \stock_{\wh}^{\node},
        \uncertain_{\wh\successor}^{\node}, 
        \recoursecontrol_{\wh\successor}^{\node}}$
at node~$\node$, when the storage level is
$\stock_{\wh}^{\node}$, the uncertain variable is $\uncertain_{\wh\successor}^{\node}$, 
and the recourse control is
$\recoursecontrol_{\wh\successor}^{\node}$.
The value~$+\infty$ corresponds to infeasibility (constraints not satisfied).

The weekly nodal cost~$\InstantaneousCost_{\week}^{\node}$ is defined as the sum of the hourly 
costs~$\InstantaneousCost_{\wh}^{\node}$ over all hours of the week
\begin{align}
    \InstantaneousCost_{\week}^{\node}
        \bp{
            \stock_{\whinf}^{\node},
            \uncertain_{\openWclosed{\week}}^{\node}, 
            \recoursecontrol_{\openWclosed{\week}}^{\node}
        }
        &=
        \sum_{\hour \in \HOUR} \InstantaneousCost_{\wh}^{\node} \bp{
            \stock_{\wh}^{\node},
            \uncertain_{\wh\successor}^{\node}, 
            \recoursecontrol_{\wh\successor}^{\node}
        }\eqfinv
\end{align} 
where each hourly storage level $\bseqp{\stock_{\wh\successor}^{\node}}{\hour\in\HOUR}$ 
is obtained from the stock level 
  composition of the hourly dynamics $\dynamics_{\wh}^{\node}$
   (see Equation~\eqref{eq:WeeklyDynamicsNodal}).
Therefore, the weekly cost is a function of the initial storage level $\stock_{\whinf}^{\node}$ for the week,
the collection of uncertainties $\uncertain_{\openWclosed{\week}}^{\node}$ for the week,
 and the collection of recourse controls $\recoursecontrol_{\openWclosed{\week}}^{\node}$ for the week.

\paragraph{Transportation cost in arcs.}

We consider a family~$\bseqp{\nseqp{\InstantaneousCost_{\wh}^{\arc}}{\arc\in\ARC}}{\wh\in\timeline}$
 of hourly transportation cost functions through the arcs of the directed graph $\graph\np{\NODE,\ARC}$ given,
 for all $\wh \in \timeline$ and $\arc\in\ARC$, by a function~$  \InstantaneousCost_{\wh}^{\arc}: 
            \RR
    \to 
            \RR \cup \na{+ \infty}$.
The transportation cost is not a real cost in the system, but a cost added to 
reduce the flow through the arcs, avoiding parasite flows in the system.

As the flow~$\flowarc_{\wh\successor}^{\arc}$ can take positive values and negative values, 
$\InstantaneousCost_{\wh}^{\arc}\bp{ \flowarc_{\wh\successor}^{\arc}}$ is modeled as a quadratic function
of the arc flow.
We impose the bound constraints given 
in Equation~\eqref{eq:ArcFlowBounds_weekly}
as a penalty term
in the cost function 
using a characteristic function of the interval
$\nc{\binf{\flowarc}_{\openWclosed{\week}}^{\arc}, \bsup{\flowarc}_{\openWclosed{\week}}^{\arc}}$.

We write the weekly transportation cost~$\InstantaneousCost_{\week}^{\arc}$
through the arc $\arc$~$  \InstantaneousCost_{\week}^{\arc}: 
    \RR^{\cardinal{\HOUR}}                                                                                                                                                                                                                                                                                                                                                                                                                                                                                                                                                                                                                                                                                                                                                                                                                                                                                                                                                                                                                                                                                                                                                                                                                                                                                                                                                                                                                                                                                                                                                                                                                                                                                                                                                                                                                                                                                                                                                                                                                                                                                                                                                                                                                                                                                                                                                                                                                                                                                                                                                                                                                                                                                                                                                                                                                                                                                                                                                                                                                                                                                                                                                                                                                                                                                                                                                                                                                                                                                                                                                                                                                                                                                                                                                                                                                                                                                                                                                                                                                                                                                                                                                                                                                                                                                                                                                                                                                                                                                                                                                                                                                                                                                                                                               
    \to 
            \RR \cup \na{+ \infty}$, 
obtained as the sum of the hourly transportation costs~$\InstantaneousCost_{\wh}^{\arc}$
 over all hours of the week, as
\begin{equation}
    \InstantaneousCost_{\week}^{\arc}\bp{
            \flowarc_{\openWclosed{\week}}^{\arc}
            }
            =
            \sum_{\hour\in\HOUR}
                \InstantaneousCost_{\wh}^{\arc}\bp{
                \flowarc_{\wh\successor}^{\arc}
                }
                    \eqfinp
        \label{eq:WeeklyTransportationCost}
\end{equation}

\subsubsection{Random variables}
\label{subsubsection:Random_variables_multi_node_system}

\phantomsection
\subsubsubsection{Probability spaces}

{We introduce the measurable space~$\np{\Omega, \trib}$, where
  \begin{subequations}
    \begin{equation}
      \Omega= \prod_{\wh\in\extimeline} \UNCERTAIN_{\wh}
      \eqfinv
   \label{eq:Omega}
    \end{equation}
    equipped with the product $\sigma$-algebra~$\trib$
    (recall that each $\UNCERTAIN_{\wh}^{\node}$ is a measurable space).
    Then, we denote by 
\begin{align}
 &\va{\Uncertain} = \bp{\va{\Uncertain}_{\winfhinf}, \cdots,
  \va{\Uncertain}_{\whlast}} :
   \Omega \to \prod_{\wh\in\extimeline} \UNCERTAIN_{\wh}  
   \label{eq:uncertainty_process_multinode}
   \intertext{the canonical random process made of the coordinate mappings ($\va{\Uncertain}$ is the identity mapping),
   and we also denote}
&\va{\Uncertain}_{\wh} = \bseqp{\va{\Uncertain}_{\wh}^{\node}}{\node\in\NODE} \eqsepv \forall\wh\in\extimeline
                  \label{eq:uncertainty_process_multinode_hour}
\end{align}
\end{subequations}
which represents the collection of uncertainties for all nodes $\node\in\NODE$ at $\wh$.
As we consider a centralized information structure, the information constraints in
this modelling will encompass the collection of all nodal uncertainties,
hence it will be practical to use the notation~\eqref{eq:uncertainty_process_multinode_hour}.
The measurable space~$\np{\Omega, \trib}$ in~\eqref{eq:Omega} is equipped with a probability~$\prbt$.
}

\begin{remark}\label{remark:empirical_probability_multinode}
    Suppose given a set $\CHRONICLE$ that indexes so-called
    \emph{uncertainty chronicles}~$\uncertain^{\chronicle}= \bp{
                                    \np{\uncertain_{\openWclosed{\weekinf}}}^{\chronicle}, \dots,
                                    \np{\uncertain_{\openWclosed{\week}}}^{\chronicle}
                                    , \dots,
                                    \np{\uncertain_{\openWclosed{\weeksup}}}^{\chronicle}}
            \eqsepv \forall \chronicle \in \CHRONICLE$,
    \begin{subequations}\label{eq:empirical_probability_multinode}
            where $\np{\uncertain_{\openWclosed{\week}}}^{\chronicle}=
        \bp{\nseqp{\uncertain_{\openWclosed{\week}}^{\node}}{\node\in\NODE}}^{\chronicle}$.
        From~$\CHRONICLE$, we build, 
        on the sample space~$\Omega$ in~\eqref{eq:Omega},
what we call the \emph{empirical probability}
        \begin{equation}\label{eq:empirical_probability_Prbt_multinode}
            \widetilde{\prbt} = \frac{1}{\cardinality{\CHRONICLE}} \sum_{c\in\CHRONICLE}
            \delta_{\np{\np{\uncertain_{\openWclosed{\weekinf}}}^{\chronicle}, \dots,
                                    \np{\uncertain_{\openWclosed{\week}}}^{\chronicle}
                                    , \dots,
                                    \np{\uncertain_{\openWclosed{\weeksup}}}^{\chronicle}}}
            \eqfinv
        \end{equation}
        where $\delta$ denotes the the Dirac mass.
        The support of~$\widetilde{\prbt}$ is made of $\cardinality{\CHRONICLE}$ atoms.
        Then, the expectation~$\widetilde{\espe}$ is defined with respect to the probability~$\widetilde{\prbt}$ 
        as a finite sum over the uncertainty chronicles in~$\CHRONICLE$
        \begin{equation}\label{eq:empirical_probability_Expectation_multinode}
            \widetilde{\espe}\nc{\zeta(\va{\Uncertain})} = 
            \frac{1}{\cardinality{\CHRONICLE}} \sum_{c\in\CHRONICLE} 
                            \zeta\bp{{\np{\uncertain_{\openWclosed{\weekinf}}}^{\chronicle}, \dots,
                                    \np{\uncertain_{\openWclosed{\week}}}^{\chronicle}
                                    , \dots,
                                    \np{\uncertain_{\openWclosed{\weeksup}}}^{\chronicle}}}
            \eqfinp    
        \end{equation}
    \end{subequations}
In Remark~\ref{remark:WeeklyProbability_multinode}, we define another 
probability, $\widehat{\prbt}$, on the same sample space~$\np{\Omega,\trib}$. 
 \end{remark}

\begin{remark}\label{remark:WeeklyProbability_multinode} 
    We build on the sample space~$\np{\Omega,\trib}$ in~\eqref{eq:Omega}, 
   from the same set~$\CHRONICLE$ (or a subset of it)
   of uncertainty chronicles, another probability
 \begin{subequations}
   \begin{equation}
     \label{eq:weekly_probability_Prbt_multinode}
      \widehat{\prbt} = \bigotimes_{\week \in \WEEK} \frac{1}{\cardinality{\CHRONICLE}} 
      \sum_{{\chronicle}\in\CHRONICLE} 
      \delta_{\np{\uncertain_{\openWclosed{\week}}}^{\chronicle}}
      \eqfinv
  \end{equation}
  which, by construction, makes the coordinate mappings in Equation~\eqref{eq:uncertainty_process_multinode}
  (weekly)~independent,
  and that we call the \emph{empirical product probability}.

  The support of~$\widehat{\prbt}$ is made of $\cardinality{\CHRONICLE}^{\cardinal{\WEEK}}$ atoms
  (exponentially higher than the support of~$\widetilde{\prbt}$).
  This probability~$\widehat{\prbt}$ will only 
  serve to evaluate expectations, denoted~$\widehat{\espe}$, of functions depending on the weekly uncertainties, using the expression
  \begin{align}
    \label{eq:weekly_probability_Expectation_multinode}
        \widehat{\espe}\nc{\zeta(\va{\Uncertain}_{\openWclosed{\week}})} = 
          \frac{1}{\cardinality{\CHRONICLE}}  \sum_{c\in\CHRONICLE} 
                                \zeta\bp{\np{\uncertain_{\openWclosed{\week}}}^{\chronicle}}
        \eqfinp    
    \end{align} 
   \end{subequations}
\end{remark}

\paragraph{Information structure}
We consider a \emph{weekly hazard-decision ({HD})} information structure \cite{martinezparra:hal-04681616}.

In this structure, the recourse controls $\recoursecontrol_{\openWclosed{\week}}^{\node}$, and 
the flows decisions $\flownode_{\openWclosed{\week}}^{\node}$ and $\flowarc_{\openWclosed{\week}}^{\arc}$
are made knowing the uncertainties $\uncertain_{\openWclosed{\week}}$ for the current week~$\week$.
With the weekly notations in the variables definition in~\S\ref{subsubsection:Timescale_graph_and_variables_definition},
we write the information constraints~$\forall \week\in\WEEK$, 
\begin{subequations}\label{eq:InfoConstraints}
 \begin{align}
      \bsigmaf{\va{\RecourseControl}_{\openWclosed{\week}}^{\node}}
        &\subseteq 
           \bsigmaf{\va{\Uncertain}_{\winfhinf}, \va{\Uncertain}_{\openWclosed{\binf{\week}}}, 
                                                \cdots, \va{\Uncertain}_{\openWclosed{\week\predecessor}},
                                                \va{\Uncertain}_{\openWclosed{\week}}}
        \eqsepv \forall \week\in\WEEK
        \eqfinv \label{eq:Weekly_InfoConstraint_RecourseControl}
      \\
    \bsigmaf{\va{\FlowNode}_{\openWclosed{\week}}^{\node}}
        &\subseteq 
           \bsigmaf{\va{\Uncertain}_{\winfhinf}, \va{\Uncertain}_{\openWclosed{\binf{\week}}}, 
                                                \cdots, \va{\Uncertain}_{\openWclosed{\week\predecessor}},
                                                \va{\Uncertain}_{\openWclosed{\week}}}
        \eqsepv \forall \week\in\WEEK
       \eqfinv \label{eq:Weekly_InfoConstraint_FlowNode}
     \\
    \bsigmaf{\va{\FlowArc}_{\openWclosed{\week}}^{\arc}}\label{eq:Weekly_InfoConstraint_FlowArc2} 
        &\subseteq 
           \bsigmaf{\va{\Uncertain}_{\winfhinf}, \va{\Uncertain}_{\openWclosed{\binf{\week}}}, 
                                                \cdots, \va{\Uncertain}_{\openWclosed{\week\predecessor}},
                                                \va{\Uncertain}_{\openWclosed{\week}}}
        \eqsepv \forall \week\in\WEEK 
        \eqfinp
\end{align}
\end{subequations}

\subsection{Global optimisation problem formulation}
\label{subsection:Global_optimisation_problem_formulation}
In~\S\ref{subsubsection:Production_problems_at_each_node}
 and~\S\ref{subsubsection:Transport_cost_in_the_arcs}, we present respectively the nodal production
 problems and the transport cost in the arcs.
 The global optimisation problem is stated in~\S\ref{subsubsection:Global_optimisation_problem}
 as the sum of the nodal production problems and the transport cost in the arcs.

\subsubsection{Production problems at each node}
\label{subsubsection:Production_problems_at_each_node}
We present here the nodal production problems at each node $\node \in \NODE$.

 We define~$\va{\FlowNode}^{\node}  = \bp{ \va{\FlowNode}_{\openWclosed{\binf{\week}}}^{\node},
  \dots, \va{\FlowNode}_{\openWclosed{\bsup{\week}}}^{\node}}$,
 the node 
flow process at each $\node \in \NODE$ between weeks $\binf{\week}$ and $\bsup{\week}$,
and~$\stock_{\winfhinf}^{\node}$ the initial storage level at the node $\node$.
We call \emph{optimal nodal production cost} the expression
\begin{subequations}
    \label{eq:nodalProductionProblemHD}
 \begin{align}
   \NodalCost^{{\node}}\bp{  \stock_{\winfhinf}^{\node},  \va{\FlowNode}^{\node}}=
   \min  \mathbb{E} 
       & \Biggl[
            \sum_{\week\in\WEEK} 
                    \Bigl(
                     \InstantaneousCost_{\week}^{\node}\bigl(
                            \va{\Stock}_{\whinf}^{\node}, 
                           \va{\Uncertain}_{\openWclosed{\week}}^{\node},
                            \va{\RecourseControl}_{\openWclosed{\week}}^{\node}
                            \bigr)
+
                        \FinalCost^{\node}\bp{\va{\Stock}_{\whlast}^{\node}}
                    \Bigr)
                   \Biggr]
        \label{eq:GSFWeeklyHD-1}
\\
\text{s.t.} &
\nonumber
    \\
     &
    \quad 
        \va{\Stock}_{\winfhinf}^{\node} 
            = 
            \stock_{\winfhinf}^{\node}
                \eqfinv \label{eq:GSFWeeklyHD-InitialState} 
    \\
    & 
    \quad   
        \eqref{eq:WeeklyDynamicsNodal}
        \eqsepv 
        \eqref{eq:WeeklyNodalBalanceConstraint}
        \eqsepv
        \eqref{eq:StorageLevelBounds_weekly}
        \eqsepv
        \eqref{eq:RecourseControlBounds_weekly}
        \eqsepv 
        \eqref{eq:Weekly_InfoConstraint_RecourseControl}
        \eqsepv 
        \forall \week \in \WEEK
        \eqfinp
\end{align}
\end{subequations} 

\subsubsection{Transport cost in the arcs}
\label{subsubsection:Transport_cost_in_the_arcs}
We define the arc flow process~$\va{\FlowArc}^{\arc}  = \bp{ \va{\FlowArc}_{\openWclosed{\binf{\week}}}^{\arc}, \dots, \va{\FlowArc}_{\openWclosed{\bsup{\week}}}^{\arc}}$,
  at each $\arc \in \ARC$ between weeks $\binf{\week}$ and $\bsup{\week}$ and, 
the collection of all arc flow processes~$\va{\FlowArc} = \bseqp{\va{\FlowArc}^{\arc}}{\arc\in\ARC}$.
Then, the transport cost is defined as 
\begin{align} \label{eq:ArcTransportationCost}
    \ArcCost\bp{ \va{\FlowArc}} = 
  \mathbb{E} &
     \left[
        \sum_{\arc\in\ARC} \sum_{\week\in\WEEK} 
                    \InstantaneousCost_{\week}^{\arc}\bp{
                        \va{\FlowArc}_{\openWclosed{\week}}^{\arc}    
                    }
                 \right]\eqfinp
\end{align}

\subsubsection{Global optimisation problem}
\label{subsubsection:Global_optimisation_problem}
Recalling that   $\displaystyle \stock_{\winfhinf} = \bseqp{\stock_{\winfhinf}^{\node}}{\node\in\NODE} \in
\prod_{\node\in\NODE}\STOCK_{\winfhinf}^{\node} $, we write the global optimisation problem as
\begin{subequations}
    \label{eq:GlobalProblem}
\begin{align}
        \GlobalProblem\bp{\stock_{\winfhinf}} =& \min_{\va{\FlowNode}, \va{\FlowArc}} \sum_{\node\in\NODE} 
    \NodalCost^{\node}\bp{   \stock_{\winfhinf}^{\node},    \va{\FlowNode}^{\node}} + 
    \ArcCost\bp{ \va{\FlowArc}}
    \\
    \st \nonumber
  \\
  &
  \va{\FlowNode}
  -
  \NODEARCMATRIX \va{\FlowArc } = 0
  \eqfinv \label{eq:globalcoupling}
  \\
   & \eqref{eq:Weekly_InfoConstraint_FlowNode}
  \text{ } \forall \node\in\NODE \eqsepv \text{ and, } 
    \eqref{eq:Weekly_InfoConstraint_FlowArc2}
  \text{ }  \forall \arc\in\ARC
    \eqfinv \label{eq:globalinformationconstraint-arcflow}
\end{align}
\end{subequations}
where~$\NodalCost^{\node}\bp{   \stock_{\winfhinf}^{\node},    \va{\FlowNode}^{\node}}$
is the optimal nodal production cost
in Equation~\eqref{eq:nodalProductionProblemHD}
and~$\ArcCost\bp{ \va{\FlowArc}}$ is the transport cost defined in Equation~\eqref{eq:ArcTransportationCost}.
The global optimisation problem~\eqref{eq:GlobalProblem} combines, in its objective 
function, the nodal production problems and the transport problem along the arcs.
Equation~\eqref{eq:globalcoupling} couples the
nodal flows and the arc flows for all $\wh\successor \in \timeline$.

\subsection{Temporal decomposition (by dynamic programming)}
\label{subsection:Temporal_decomposition}

In~\S\ref{subsubsection:Dynamic_programming_equations_for_the_global_optimisation_problem},
we present the cost-to-go functions and the corresponding dynamic programming equations
for the global optimisation problem~\eqref{eq:GlobalProblem}.
In~\S\ref{subsubsection:Policies_and_simulation_algorithms}, we present the algorithms
for computing and simulating policies
for the global optimisation problem~\eqref{eq:GlobalProblem} in a multinode system
using temporal decomposition.

\subsubsection{Dynamic programming equations for the global optimisation problem}
\label{subsubsection:Dynamic_programming_equations_for_the_global_optimisation_problem}

At each stage (week~$\week$), we define the dynamic programming state~$\statew$  as 
the collection of initial nodal storage levels at the beginning of the week~$\stock_{\whinf}$,
that is,~$\statew = \nseqp{\stock^{\node}_{\whinf}}{\node\in\NODE} \in
\STATE_{\week} = \prod_{\node\in\NODE}\STOCK_{\whinf}^{\node}$.
Then, we define a sequence~$\bseqp{\nBellmanf{\week}{}}{\week\in\WEEK\cup \na{\weeklast}}$ 
of global cost-to-go functions~$ \nBellmanf{\week}{}:\STATE_{\week} \to \RR \cup \na{+\infty} $
that are given, for all~$\week\in\WEEK\cup \na{\weeklast}$ and for all~$\state\in\STATE_{\week}$, by
\begin{subequations}
    \label{eq:GlobalCostToGoHD}
    \begin{align}
     &   \nBellmanf{\week}{}\bp{\state} = \label{eq:GlobalCostToGoHD-1} 
     \min
        \EE\vast[ \sum_{\week'=\week}^{\bsup{\week}}
      \Biggl(  \sum_{\node\in\NODE}       \left(
                    \InstantaneousCost_{\week'}^{\node}\bp{
                            \va{\State}_{\week'}^{\node}, 
                            \va{\Uncertain}_{\openWclosed{\week'}}^{\node},
                            \va{\RecourseControl}_{\openWclosed{\week'}}^{\node}}
                        +
                        \FinalCost^{\node}\bp{\va{\State}_{\weeklast}^{\node}}
                    \right)
                    +
                    \sum_{\arc\in\ARC}    
                    \InstantaneousCost_{\week'}^{\arc}\bp{
                        \va{\FlowArc}_{\openWclosed{\week'}}^{\arc}    
                    }
               \Biggr)
        \vast] 
        \\
       & \st \nonumber
          \\
    & 
    \va{\State}_{\week}            = 
            \state
                \eqfinv \label{eq:GlobalCostToGoHD-2}
    \\
    &
    \text{and, for } \week'\in \ic{\week,\bsup{\week}}\eqfinv \nonumber
    \\
    &
        \va{\State}_{{\week'}\successor}^{\node} 
            = 
            \dynamics_{\week'}^{\node}\Bp{
                \va{\State}_{\week'}^{\node},
                \va{\Uncertain}_{\openWclosed{\week'}}^{\node},
                \va{\RecourseControl}_{\openWclosed{\week'}}^{\node}
            }
            \eqsepv \forall \node\in\NODE \eqfinv  \label{eq:GlobalCostToGoHD-3}
    \\
     & 
     \nodeBalance_{\openWclosed{\week'}}^{\node}\bp{
                    \va{\Uncertain}_{\openWclosed{\week'}}^{\node},
                    \va{\RecourseControl}_{\openWclosed{\week'}}^{\node}
     }  = \va{\FlowNode}_{\openWclosed{\week'}}^{\node}
            \eqsepv \forall \node\in\NODE 
            \eqfinv \label{eq:GlobalCostToGoHD-4}
      \\
     & 
        \bsigmaf{\va{\RecourseControl}_{\openWclosed{\week'}}^{\node}}
        \subseteq 
      \bsigmaf{\va{\Uncertain}_{\openWclosed{{\week}}}, \cdots, 
                                 \va{\Uncertain}_{\openWclosed{{\week'}\predecessor}},  
                                 \va{\Uncertain}_{\openWclosed{\week'}}}
            \eqsepv \forall \node\in\NODE \eqfinv \label{eq:GlobalCostToGoHD-6}
     \\ 
        & 
        \bsigmaf{\va{\FlowNode}_{\openWclosed{\week'}}^{\node}}
        \subseteq 
        \bsigmaf{\va{\Uncertain}_{\openWclosed{{\week}}}, \cdots, 
                                 \va{\Uncertain}_{\openWclosed{{\week'}\predecessor}},  
                                 \va{\Uncertain}_{\openWclosed{\week'}}}
            \eqsepv \forall \node\in\NODE \eqfinv \label{eq:GlobalCostToGoHD-7}
                    \\
        & 
        \bsigmaf{\va{\FlowArc}_{\openWclosed{\week'}}^{\arc}}
        \subseteq  \bsigmaf{\va{\Uncertain}_{\openWclosed{{\week}}}, \cdots, 
                                 \va{\Uncertain}_{\openWclosed{{\week'}\predecessor}},  
                                 \va{\Uncertain}_{\openWclosed{\week'}}} 
            \eqsepv \forall \arc\in\ARC \eqfinv \label{eq:GlobalCostToGoHD-8}
    \\
    &
      \eqref{eq:StorageLevelBounds_weekly} \eqsepv 
       \eqref{eq:RecourseControlBounds_weekly}
          \eqsepv \forall \node \in \NODE \label{eq:GlobalCostToGoHD-9}
            \eqfinv \\
    &\va{\FlowNode}_{\openWclosed{\week'}}  - \NODEARCMATRIX \va{\FlowArc }_{\openWclosed{\week'}} = 0
    \eqfinp \label{eq:GlobalCostToGoHD-10}
    \end{align}
\end{subequations}

\begin{assumption}
  \label{assumption:weekly_independence}
  The measurable space~$\np{\Omega, \trib}$ in~\eqref{eq:Omega} is equipped with a probability~$\prbt$
  such that     the sequence 
    $\bp{\va{\Uncertain}_{\openWclosed{\binf{\week}}}, 
      \cdots, \va{\Uncertain}_{\openWclosed{\week}}, \cdots, \va{\Uncertain}_{\openWclosed{\bsup{\week}}}}$
    of random variables is weekly independent, i.e., 
    \begin{equation}
        \prbt_{\np{\va{\Uncertain}_{\winfhinf}, \cdots, \va{\Uncertain}_{\whlast}}} =
        \prbt_{\np{\va{\Uncertain}_{\winfhinf}}}
        \otimes
        \bigotimes_{\week \in \WEEK} 
        \prbt_{\va{\Uncertain}_{\openWclosed{\week}}}
        \eqfinp
      \end{equation}
\end{assumption}
For instance, the probability~\( \widehat{\prbt} \)
in~\eqref{eq:weekly_probability_Prbt_multinode} satisfies Assumption~\ref{assumption:weekly_independence}.

Under Assumption~\ref{assumption:weekly_independence}, 
the global cost-to-go functions~$\bseqp{\nBellmanf{\week}{}}{\week\in\WEEK\cup\na{\weeklast}}$ 
satisfy the following dynamic programming equations
\begin{subequations}
    \label{eq:GlobalDP_HD}
    \begin{align}
        &
        \nBellmanf{\weeklast}{}\np{\state} = 
        \sum_{\node\in\NODE} \FinalCost^{\node}\bp{\state^{\node}} \eqfinv \label{eq:GlobalDP_HD-1}
        \\
        \begin{split} 
         &   \nBellmanf{\week}{}\np{\state} =    \EE\vast[    
             \min
             \sum_{\node\in\NODE}       
                    \InstantaneousCost_{\week}^{\node}\bp{
                           \state^{\node}, 
                                                            \va{\Uncertain}_{\openWclosed{\week}}^{\node},
                            \recoursecontrol_{\openWclosed{\week}}^{\node}} 
               +
                \sum_{\arc\in\ARC}   
                    \InstantaneousCost_{\week}^{\arc}\bp{
                        \flowarc_{\openWclosed{\week}}^{\arc}    
                    }
               \\
               &
          \qquad\qquad    \qquad    \qquad    \qquad  
             +
               \nBellmanf{\week\successor}{}\Bp{\bseqp{ \dynamics_{\week}^{\node}\np{
                                                            \state^{\node},
                                                            \va{\Uncertain}_{\openWclosed{\week}}^{\node},
                                                            \recoursecontrol_{\openWclosed{\week}}^{\node}
                                                        }}{\node\in\NODE}}               
                    \vast] \eqfinv
            \end{split}\label{eq:GlobalDP_HD-2}
         \\
       & \text{subject to:}\nonumber \\
       &
       \qquad \va{\FlowNode}_{\openWclosed{\week}}  - \NODEARCMATRIX \va{\FlowArc}_{\openWclosed{\week}} = 0\eqsepv
   \label{eq:GlobalDP_HD-4}  \text{and, }   
      \eqref{eq:WeeklyNodalBalanceConstraint} \eqsepv 
       \eqref{eq:StorageLevelBounds_weekly} \eqsepv 
       \eqref{eq:RecourseControlBounds_weekly}
        \eqsepv
         \forall \node\in\NODE 
            \eqfinp   
    \end{align}
\end{subequations}

As is well known (\cite{Bellman:1957, Bertsekas-Shreve:1996, Carpentier-Chancelier-DeLara-Martin-Rigaut:2023}),
 if the sequence $\bp{\va{\Uncertain}_{\openWclosed{\binf{\week}}}, 
\cdots, \va{\Uncertain}_{\openWclosed{\week}}, \cdots, \va{\Uncertain}_{\openWclosed{\bsup{\week}}}}$
of uncertainties is weekly independent as in Assumption~\ref{assumption:weekly_independence},
 then
the initial global cost-to-go function~$\nBellmanf{\binf{\week}}{}$ 
provides an optimal solution for 
the problem~\eqref{eq:GlobalProblem}:
\begin{align}
    \label{eq:GlobalProblemOptimalityCostToGo}
    \GlobalProblem\bp{\stock_{\winfhinf}} = \nBellmanf{\binf{\week}}{}\bp{\stock_{\winfhinf}} \eqfinp
\end{align}

The cost-to-go functions
$\nBellmanf{\week}{}$ can be computed recursively through 
dynamic programming equations~\eqref{eq:GlobalDP_HD}.

\subsubsection{Policies computation and simulation algorithms}
\label{subsubsection:Policies_and_simulation_algorithms}
In this~\S\ref{subsubsection:Policies_and_simulation_algorithms}, we present the algorithms
for computing and simulating the hazard-decision policies
for the global optimisation problem~\eqref{eq:GlobalProblem} in a multinode system.
 
Given an uncertainty chronicle~$\np{\uncertain}^\chronicle$ as defined in
Remark~\ref{remark:empirical_probability_multinode},
for any policy we compute the controls~$ \bp{
                    \np{\recoursecontrol}^{\chronicle},
                    \np{\flownode}^{\chronicle},
                    \np{\flowarc}^{\chronicle}}$,
where 
$ \np{\recoursecontrol}^\chronicle = \bp{\np{{\recoursecontrol_{\openWclosed{\weekinf}}}}^{\chronicle},
                                    \dots, 
                                   \bp{{\recoursecontrol_{\openWclosed{\week}}}}^{\chronicle}, 
                                    \dots,
                                    \bp{{\recoursecontrol_{\openWclosed{\weeksup}}}}^{\chronicle}
}$, with, $\np{{\recoursecontrol_{\openWclosed{\week}}}}^{\chronicle} 
                                = \bp{\nseqp{\recoursecontrol_{\openWclosed{\week}}^{\node}}{\node\in\NODE}}^{\chronicle} $,
$\np{\flownode}^\chronicle = \bp{\np{{\flownode_{\openWclosed{\weekinf}}}}^{\chronicle},
                                    \dots, 
                                   \bp{{\flownode_{\openWclosed{\week}}}}^{\chronicle}, 
                                    \dots,
                                    \bp{{\flownode_{\openWclosed{\weeksup}}}}^{\chronicle}
}$, with, $  \np{{\flownode_{\openWclosed{\week}}}}^{\chronicle} 
                                = \bp{\nseqp{\flownode_{\openWclosed{\week}}^{\node}}{\node\in\NODE}}^{\chronicle} $,
and,
$\np{\flowarc}^\chronicle = \bp{\np{{\flowarc_{\openWclosed{\weekinf}}}}^{\chronicle},
                                    \dots, 
                                   \bp{{\flowarc_{\openWclosed{\week}}}}^{\chronicle}, 
                                    \dots,
                                    \bp{{\flowarc_{\openWclosed{\weeksup}}}}^{\chronicle}
}$, with, $\np{{\flowarc_{\openWclosed{\week}}}}^{\chronicle} 
                                 \bp{\nseqp{\flowarc_{\openWclosed{\week}}^{\arc}}{\arc\in\ARC}}^{\chronicle} $.
By applying the controls~$\bp{
                    \np{\recoursecontrol}^{\chronicle},
                    \np{\flownode}^{\chronicle},
                    \np{\flowarc}^{\chronicle}}$, 
                    we obtain the state trajectory~$\np{\state}^{\chronicle}= \bp{\np{{\state_{{\weekinf}}}}^{\chronicle},
                                    \dots, 
                                   \bp{{\state_{{\week}}}}^{\chronicle}, 
                                    \dots,
                                    \bp{{\state_{{\weeksup}}}}^{\chronicle}
}$, where $\np{{\state_{{\week}}}}^{\chronicle} 
                                = \bp{\nseqp{\state_{{\week}}^{\node}}{\node\in\NODE}}^{\chronicle} \eqsepv
\forall\week\in \WEEK$.

The sequential {HD} policy computation and simulation algorithm is presented in Algorithm~\ref{alg:SimulationHD_multinode},
where~$\bseqp{\textit{cost-to-go}_{\week}}{\week\in\WEEK\cup\na{\weeklast}}$ is a sequence of global cost-to-go functions.
\begin{algorithm}[!ht]
  \caption{\small Policy and simulation in a multinode system}\label{alg:SimulationHD_multinode}
    \KwData{
          cost-to-go functions~$\bseqp{\textit{cost-to-go}_{\week}}{\week\in\WEEK\cup\na{\weeklast}}$,\\
          \hspace{1.5cm}simulation chronicle $\np{\uncertain}^{\chronicle}$,\\
          \hspace{1.5cm}initial condition for the state $\bseqp{\state_{0}^\node}{\node\in\NODE}$}
   \KwResult{\(\np{\state}^{\chronicle}, 
                    \np{\recoursecontrol}^{\chronicle},
                    \np{\flownode}^{\chronicle},
                    \np{\flowarc}^{\chronicle}\)}
   $\np{\state_{\binf{\week}}^{\node}}^{\chronicle} = \state_{0}^{\node}\eqsepv \forall \node\in\NODE\eqfinv$\\
   \For{$\week = \binf{\week},  \dots, \bsup {\week}$}{
     \texttt{compute controls:}\\
      \(\displaystyle 
        \Bp{
        \np{\recoursecontrol_{\openWclosed{\week}}}^{\chronicle},
        \np{\flownode_{\openWclosed{\week}}}^{\chronicle},
        \np{\flowarc_{\openWclosed{\week}}}^{\chronicle}}
             \in    \argmin_{\substack{
                                 \recoursecontrol_{\openWclosed{\week}}\\
                             \flownode_{\openWclosed{\week}}, \flowarc_{\openWclosed{\week}}}}
             \Biggl\{
                   \sum_{\node\in\NODE}       
                    \InstantaneousCost_{\week}^{\node}\bp{
                         \np{\state_{\week}^{\node}}^{\chronicle}, 
                            \np{\uncertain_{\openWclosed{\week}}^{\node}}^{\chronicle},
                            \recoursecontrol_{\openWclosed{\week}}^{\node}} 
         \)
        \\
        \(\displaystyle
         \qquad   \qquad   \qquad  \qquad    \qquad 
               + 
                \sum_{\arc\in\ARC}   
                    \InstantaneousCost_{\week}^{\arc}\bp{
                        \flowarc_{\openWclosed{\week}}^{\arc}    
                    }
               + 
              \textit{cost-to-go}_{\week\successor}\Bp{\bseqp{ \dynamics_{\week}^{\node}\np{
                                                            \np{\state_{\week}^{\node}}^{\chronicle}, 
                            \np{\uncertain_{\openWclosed{\week}}^{\node}}^{\chronicle},
                            \recoursecontrol_{\openWclosed{\week}}^{\node}
                                                        }}{\node\in\NODE}}  \eqsepv  \text{subject to: }   \)
        \\
        \(\displaystyle   {\flownode}_{\openWclosed{\week}}  - \NODEARCMATRIX \va{\flowarc }_{\openWclosed{\week}} = 0
        \eqsepv
        \text{and, }
            \eqref{eq:WeeklyNodalBalanceConstraint} \eqsepv 
                            \eqref{eq:StorageLevelBounds_weekly} \eqsepv
                             \eqref{eq:RecourseControlBounds_weekly}
       \eqsepv  \forall \node\in\NODE
         \eqfinp
         \Biggr\} 
          \) 
       \\
          \texttt{update state:}\\
    \(\np{\state_{\week\successor}^{\node}}^{\chronicle} = 
    \dynamics_{\week}^{\node}\bp{
                                                            \np{\state_{\week}^{\node}}^{\chronicle}, 
                            \np{\uncertain_{\openWclosed{\week}}^{\node}}^{\chronicle},
                              \np{\recoursecontrol_{\openWclosed{\week}}^{\node}}^{\chronicle}
    }\eqsepv \forall\node\in\NODE\)
       }
\end{algorithm}
Note that the policy computation and simulation Algorithm~\ref{alg:SimulationHD_multinode}
is based on the dynamic programming temporal decomposition but, 
it can be applied {(as a heuristics, without any claim of optimality)} using  any sequence of global cost-to-go
functions~$\bseqp{\textit{cost-to-go}_{\week}}{\week\in\WEEK\cup\na{\weeklast}}$.
{In particular, using the global cost-to-go functions computed by solving
the dynamic programming equations~\eqref{eq:GlobalDP_HD}, under
Assumption~\ref{assumption:weekly_independence},
the policy computed by Algorithm~\ref{alg:SimulationHD_multinode} is optimal for
the global optimisation problem~\eqref{eq:GlobalProblem}.}

\subsubsection{Motivation for spatial decomposition in the context of prospective studies}
\label{subsubsection:Motivation_for_spatial_decomposition_in_the_context_of_prospective_studies}

In problems with a small number of nodes (few storage facilities), 
the cost-to-go functions can be computed recursively
using Stochastic Dynamic Programming (SDP) or Stochastic Dual Dynamic Programming (SDDP), 
and these functions can be used to design policies for simulation. 
However, in large-scale problems,
the curse of dimensionality makes this recursion computationally challenging.

Due to the complexity of the global optimisation problem~\eqref{eq:GlobalProblem}, 
we present in the forthcoming Sect.~\ref{section:Mixing_spatial_and_temporal_decomposition}
a spatial decomposition approach to tackle the curse of dimensionality
by relaxing the spatial coupling constraint~\eqref{eq:globalcoupling}
 in the global optimisation problem~\eqref{eq:GlobalProblem}.
This relaxation allows us to set lower bounds on the global optimisation problem as well as
on the cost-to-go functions~$\bseqp{\nBellmanf{\week}{}}{\week\in\WEEK\cup\na{\weeklast}}$.

Spatial decomposition is particularly relevant in the context of prospective studies because it limits the coordination between 
different market zones to the local storage levels, while accounting for interactions with other nodes 
through decomposition prices (that represent marginal costs). 
This structure enables the use of Stochastic Dynamic Programming (SDP) 
independently at each node
--- avoiding the curse of dimensionality as each node is made of a small number of storages ---
and allowing the computation of cost-to-go functions on regular grids.
In addition, it results in univariate cost-to-go functions that are easier to handle and evaluate during simulation phases.
It also opens the possibility of incorporating nonlinearities in local cost functions,
a topic which remains to be studied in the future.

\section{Mixing spatial and temporal decomposition}
\label{section:Mixing_spatial_and_temporal_decomposition}

When dealing with large-scale, multistage stochastic optimization problems, 
it is well known that classical stochastic dynamic programming approaches 
face the curse of dimensionality~\cite{Bellman:1957}.
One alternative to tackle this issue is to use Stochastic Dual Dynamic Programming (SDDP)~\cite{Pereira-Pinto:1991},
but even this approach can suffer from the curse of dimensionality when the number of storages is large.

Given that we are addressing a so-called weekly-coupled dynamic optimization problem (as defined in the PhD thesis~\cite{Hawkins:2003}),
 a decomposition strategy is a choice of particular interest. The spatial decomposition approach, known as 
 \emph{Dual Approximate Dynamic Programming}  (DADP)~\cite{pacaud2021distributed}, consists of using a 
 Lagrangian relaxation of the spatial coupling constraint in~\eqref{eq:globalcoupling}, to get a lower bound of the 
 global optimisation problem~\eqref{eq:GlobalProblem}. When relaxing the coupling constraint, we introduce a Lagrangian 
 multiplier~$\price$, the \emph{price decomposition process}, that we consider in the deterministic class
 (hence not looking for a random optimal multiplier). 
 Once the coupling constraint is relaxed, it is possible to solve the nodal production problems and the transport 
 problems in the arcs independently, using Stochastic Dynamic Programming (SDP) at each node, when $\price$ 
 is a deterministic vector.

In~\S\ref{subsection:Spatial_decomposition_by_prices}, 
we introduce the spatial decomposition by prices, including 
the temporal decomposition of the nodal production problems.
In~\S\ref{subsection:Relation_between_nodal_and_global_cost_to_go_functions},
    we present the relation between the nodal cost-to-go functions and the global cost-to-go functions.
In~\S\ref{subsection:Designing_global_policies_from_the_nodal_cost_to_go_functions},
we show how the univariate nodal usage values 
from {DADP} are used to design global policies.
In~\S\ref{subsection:Challenges_in_improving_the_price_decomposition_process}, 
we discuss the challenges in improving the price decomposition process,
including the non-smoothness of the nodal price value functions
and the large size of the price decomposition process.

\subsection{Spatial decomposition by prices}
\label{subsection:Spatial_decomposition_by_prices}

In~\S\ref{subsubsection:Lagrangian_relaxation_of_the_spatial_coupling_constraint}, 
we introduce the price decomposition process and
the Lagrangian relaxation of the spatial coupling constraint~\eqref{eq:globalcoupling}.
In~\S\ref{subsubsection:Temporal_decomposition_of_nodal_production_problems}, 
we present the temporal decomposition of the nodal production problems.

\subsubsection{Lagrangian relaxation of the spatial coupling constraint}
\label{subsubsection:Lagrangian_relaxation_of_the_spatial_coupling_constraint}

We start by relaxing the spatial coupling constraint~\eqref{eq:globalcoupling} 
in the global optimisation problem~\eqref{eq:GlobalProblem}.
For this purpose, we introduce the inner product~$\proscal{\cdot}{\cdot}$ on
\( \RR^{\cardinal{\WEEK} \times \cardinal{\HOUR} \times \cardinal{\NODE}} \),
and a Lagrangian multiplier~$\price$,
 which we call \emph{price decomposition process}
\begin{align}
    \label{eq:PriceDecompositionProcess}
    {\price}= \nseqp{\price^{\node}}{\node \in \NODE} \in  \RR^{\cardinal{\WEEK} \times \cardinal{\HOUR} \times \cardinal{\NODE}  }
    \eqfinp
\end{align}
Here, the vector~$\price$ is not indexed by scenarios, hence is deterministic.  We choose the price decomposition process 
$\price$ from the deterministic class, as in~\cite{pacaud2021distributed}, to enable the use of {SDP} with local state 
variables at each node.
We define the \emph{global price value function} $\ValueInf\nc{\price}: \STOCK_{\winfhinf} \to \RR \cup \na{+\infty}$
for all $\stock_{\winfhinf} \in \nseqp{\STOCK_{\winfhinf}}{\node\in\NODE}$ as
\begin{subequations}
\label{eq:GlobalPriceValueFunction-1}
 \begin{align}
    \ValueInf\nc{{\price}}\bp{\stock_{\winfhinf}}
    = \min_{\va{\FlowNode}, \va{\FlowArc}}
    &\sum_{\node\in\NODE}\NodalCost^{\node}\bp{\stock_{\winfhinf}^{\node},\va{\FlowNode}^{\node}}
    + \ArcCost\bp{\va{\FlowArc}} + \EE \bc{\proscal{{\price}}{\va{\FlowNode} - \NODEARCMATRIX \va{\FlowArc} }}
    \\
    & \st \quad
     \eqref{eq:Weekly_InfoConstraint_FlowNode} \text{ } \forall \node\in\NODE \eqsepv \text{ and, } 
        \eqref{eq:Weekly_InfoConstraint_FlowArc2}  \text{ }  \forall \arc\in\ARC    \eqfinp
 \end{align}
\end{subequations}
For any value of the price decomposition process~$\price$, 
the Lagrangian relaxation gives a lower bound of the global optimisation problem~\eqref{eq:GlobalProblem}.

Working on the expression~\eqref{eq:GlobalPriceValueFunction-1}, we can separate the nodal 
problems and the transport problem in the arcs, since the coupling constraint is now in the objective function
as follows
\begin{subequations}
\label{eq:GlobalPriceValueFunction-2}
 \begin{align}
    &\ValueInf\nc{{\price}}\bp{\stock_{\winfhinf}}
    = \min_{\va{\FlowNode}}
    \sum_{\node\in\NODE}\NodalCost^{\node}\bp{\stock_{\winfhinf}^{\node},\va{\FlowNode}^{\node}}
     +  \EE \bc{\proscal{{\price}}{\va{\FlowNode} }}
   \\
   \nonumber
   &\qquad\qquad\qquad
    + \min_{\va{\FlowArc}}\ArcCost\bp{\va{\FlowArc}} + \EE \bc{\proscal{{\price}}{ - \NODEARCMATRIX \va{\FlowArc} }}
    \\
       &\qquad\qquad\qquad 
        \st \quad
     \eqref{eq:Weekly_InfoConstraint_FlowNode} \text{ } \forall \node\in\NODE \eqsepv \text{ and, } 
        \eqref{eq:Weekly_InfoConstraint_FlowArc2}  \text{ }  \forall \arc\in\ARC    \eqfinp
 \end{align}
\end{subequations}
Now, we rewrite the global price value function~\eqref{eq:GlobalPriceValueFunction-1}--\eqref{eq:GlobalPriceValueFunction-2}
as the sum of nodal price value functions and an arc price value function
\begin{align}
    \label{eq:GlobalPriceValueFunction-3}
     \ValueInf\nc{{\price}}\bp{\stock_{\winfhinf}} 
    = \sum_{\node\in\NODE} \Value^{\node}\bc{{\price}^{\node}}
    \bp{\stock_{\winfhinf}^{\node}}
    +
    \Value^{\ARC}\bc{\price} \eqfinv
\end{align}
where~$\Value^{\node}\bc{{\price}^{\node}}: \STOCK_{\winfhinf}^{\node}\to \RR \cup \na{\infty}$,
for $\node\in\NODE$, 
is a family of \emph{nodal price value functions}, each defined as
    \begin{align}      \label{eq:NodalPriceValueFunction}
 \Value^{\node}\nc{{\price}^{\node}}\bp{\stock_{\winfhinf}^{\node}} = 
    & \min_{\va{\FlowNode}^{\node}} \NodalCost^{\node} \bp{\va{\FlowNode}^{\node}, \stock_{\winfhinf}^{\node} }
    +
    \EE\bc{\proscal{{\price}^{\node}}{\va{\FlowNode}^{\node}}} 
    \eqsepv
   \st   \eqref{eq:Weekly_InfoConstraint_FlowNode}
                                          \eqfinv
\end{align}
and the \emph{arc price value function} is defined as
\begin{align}    \label{eq:ArcPriceValueFunction}
    \Value^{\ARC}\bc{{\price}} =
    & \min_{\va{\FlowArc}} \ArcCost\bp{\va{\FlowArc}} 
    -
    \EE\bc{\proscal{\NODEARCMATRIX\transp {\price}}{\va{\FlowArc}}} \eqsepv
      \st \text{ }   \eqref{eq:Weekly_InfoConstraint_FlowArc2}  \text{ }  \forall \arc\in\ARC   \eqfinp
\end{align}
Using the definition of \(\ArcCost(\va{\FlowArc})\), we obtain that
\begin{subequations}\label{eq:ArcPriceValueFunction2}
    \begin{align}
         \Value^{\ARC}\bc{{\price}}
      &
        =\min_{\va{\FlowArc}} 
     \mathbb{E} 
     \left[
        \sum_{\arc\in\ARC} \sum_{\week\in\WEEK} 
                    \InstantaneousCost_{\week}^{\arc}\bp{
                        \va{\FlowArc}_{\openWclosed{\week}}^{\arc}    
                    }
        -
      \proscal{\NODEARCMATRIX\transp {\price}}{\va{\FlowArc}} \right]
      \eqsepv
        \st \text{ }    \eqref{eq:Weekly_InfoConstraint_FlowArc2}  \text{ }  \forall \arc\in\ARC   \eqfinv
    \\
  &
            = \min_{\flowarc} 
        \left[
        \sum_{\arc\in\ARC} \sum_{\week\in\WEEK} 
                    \InstantaneousCost_{\week}^{\arc}\bp{
                        \flowarc_{\openWclosed{\week}}^{\arc}    
                    }
  -  \proscal{\NODEARCMATRIX\transp {\price}}{\flowarc} \right]\eqfinv  \label{eq:ArcPriceValueFunction3}
    \end{align}
\end{subequations}
where the last equality is easily obtained because a deterministic~$\flowarc$
satisfies the information constraint~\eqref{eq:Weekly_InfoConstraint_FlowArc2}.

For a given price decomposition process~$\price$,
we can solve independently the nodal price value functions~\eqref{eq:NodalPriceValueFunction}
 for each node $\node\in\NODE$ 
and the arc price value function~\eqref{eq:ArcPriceValueFunction} and obtain
a lower bound of the global optimisation problem~\eqref{eq:GlobalProblem} as follows
\begin{equation}
    \label{eq:GlobalProblem_bound}
    \sum_{\node\in\NODE} \Value^{\node}\bc{{\price}^{\node}}
    \bp{\stock_{\winfhinf}^{\node}}
    +
    \Value^{\ARC}\bc{\price} \leq \GlobalProblem\bp{\stock_{\winfhinf}} \eqfinp
\end{equation}

As the nodal price value functions~\eqref{eq:NodalPriceValueFunction} 
depend only on the initial nodal storage level for a given price~$\price$, 
we can perform a temporal decomposition of the nodal price production problems
and obtain  univariate nodal usage values. The procedure to achieve this
is presented in~\S\ref{subsubsection:Temporal_decomposition_of_nodal_production_problems}. 

\subsubsection{Temporal decomposition of nodal production problems}
\label{subsubsection:Temporal_decomposition_of_nodal_production_problems}

In this~\S\ref{subsubsection:Temporal_decomposition_of_nodal_production_problems},
 we present the temporal decomposition of the nodal price value 
function~\eqref{eq:NodalPriceValueFunction}.
The value of the nodal price value function~\eqref{eq:NodalPriceValueFunction} 
is computed through a temporal decomposition of the nodal production problem
with the coupling constraint dualized in the objective function.

It is natural to ask why the price decomposition process~$\price$ is chosen to be deterministic. 
The reason is that not every choice of~$\price$ is compatible with temporal decomposition, 
as dynamic programming requires independence of the underlying noise process. 
In the PhD thesis~\cite{pacaud:tel-02134163}, the author studies two spatial decomposition schemes 
by prices, that are compatible with time decomposition:
the case of \emph{deterministic prices} and the case of \emph{Markovian prices}.
In the PhD thesis~\cite{Hawkins:2003} and the article~\cite{BrownIndexPolicies2020}, 
the authors also consider the case of deterministic multipliers for their Lagrangian relaxations.
In this work, we choose to use a deterministic price decomposition process 
to achieve temporal decomposition of the nodal price value functions,
with the nodal original state,
and obtain lower bounds for the global problem cost-go-functions defined in 
equations~\eqref{eq:GlobalCostToGoHD}.

We define the weekly price decomposition process for the node $\node$ as
\begin{equation}
\price_{\openWclosed{\week}}^{\node} = \nseqp{\price_{\wh\successor}^{\node}}{\hour\in\HOUR} 
\in \RR^{\cardinal{\HOUR}} \text{ with }
\price^{\node} = \nseqp{\price_{\openWclosed{\week}}^{\node}}{\week\in\WEEK} 
    \in \RR^{\cardinal{\WEEK} \times \cardinal{\HOUR}}
    \eqfinp   
\end{equation}
Using the definition for the optimal nodal production cost in Equations~\eqref{eq:nodalProductionProblemHD},
and the definition for the nodal price value function in Equation~\eqref{eq:NodalPriceValueFunction},
we detail the expression of the nodal price value function as
\begin{subequations}
\label{eq:NodalPriceValueFunctionDetailHD}
\begin{align}   
    \label{eq:NodalPriceValueFunctionDetailHD-1}
   \begin{split}
     \Value^{{\node}}\nc{{\price}^{\node}}\bp{\stock_{\winfhinf}^{\node}}  = 
    \min
             \mathbb{E} & \Biggl[
        \sum_{\week\in\WEEK} 
            \Bigl(
                \InstantaneousCost_{\week}^{\node}\bp{
                    \va{\Stock}_{\whinf}^{\node}, 
                    \va{\Uncertain}_{\openWclosed{\week}}^{\node},
                    \va{\RecourseControl}_{\openWclosed{\week}}^{\node}
                } 
        +
            \proscal{\price_{\openWclosed{\week}}^{\node}}{\va{\FlowNode}^{\node}_{\openWclosed{\week}}}
            \Bigr)   
                \\ &  
       \qquad \qquad \qquad \qquad \qquad \quad\quad\quad\quad
            +\FinalCost^{\node}\bp{\va{\Stock}_{\whlast}^{\node}}
          \Biggr] \end{split}
    \\ 
          \st  &\quad 
        \eqref{eq:GSFWeeklyHD-InitialState} \eqsepv 
       \\ 
       &  \quad \eqref{eq:WeeklyDynamicsNodal} \eqsepv
        \eqref{eq:WeeklyNodalBalanceConstraint}  \eqsepv
        \eqref{eq:StorageLevelBounds_weekly} \eqsepv 
        \eqref{eq:RecourseControlBounds_weekly}\eqsepv
        \eqref{eq:Weekly_InfoConstraint_RecourseControl} \eqsepv
        \eqref{eq:Weekly_InfoConstraint_FlowNode} \eqsepv
        \forall \week \in \WEEK \eqfinp
            \label{eq:NodalPriceValueFunctionDetailHD-8}
    \end{align}
\end{subequations}
The nodal price value function~\eqref{eq:NodalPriceValueFunctionDetailHD} 
is a multistage stochastic optimisation problem.
Therefore, we write dynamic programming equations 
by taking as the stage cost the sum of the instantaneous cost~$\InstantaneousCost_{\week}^{\node}$
and $\price^{\node}_{\openWclosed{\week}}\flownode^{\node}_{\openWclosed{\week}}$.

We define the dynamic programming state~$\state_{\week}^{\node}$ as 
the nodal storage level~${\stock}_{\whinf}^{\node}$ at the beginning of the week~$\week$ and
a sequence  
$\bseqp{\nBellmanf{\week}{{\node}}\nc{\price^{\node}}}{\week \in\WEEK \cup  \na{\weeklast}}$
of \emph{nodal price cost-to-go functions}~$
 \nBellmanf{\week}{{\node}}\nc{\price^{\node}}: \STATE_{\week}^{\node}\to \RR \cup \na{+\infty}
$, 
given by:
\begin{subequations}\label{eq:NodalCostToGoHD}
    \begin{align}\label{eq:NodalCostToGoHD-1}
     \nBellmanf{\week}{{\node}}\nc{{\price}^{\node}}\bp{\state^{\node}} & = 
    \min               \mathbb{E} \Biggl[
        \sum_{\week'=\week}^{\bsup{\week}} 
            \Bigl(
                \InstantaneousCost_{\week'}^{\node}\bp{
                    \va{\State}_{\week'}^{\node}, 
                    \va{\Uncertain}_{\openWclosed{\week'}}^{\node},
                    \va{\RecourseControl}_{\openWclosed{\week'}}^{\node}
                } 
                       +
            \proscal{\price_{\openWclosed{\week'}}^{\node}}{\va{\FlowNode}^{\node}_{\openWclosed{\week'}}}
            \Bigr)   
           +\FinalCost^{\node}\bp{\va{\State}_{\weeklast}^{\node}}
          \Biggr]
    \\ 
     \st \eqsepv   & \forall \week'\in\ic{\week, \weeksup} \eqfinv
        \nonumber
    \\
    & 
        \va{\State}_{\week}^{\node} 
            = 
           \state^{\node}
                \eqfinv
               \label{eq:NodalCostToGoHD-2}
    \\
    &
        \va{\State}_{{\week'}\successor}^{\node} 
            = 
            \dynamics_{\week'}^{\node}\Bp{
                \va{\State}_{\week'}^{\node},
                \va{\Uncertain}_{\openWclosed{\week'}}^{\node},
                \va{\RecourseControl}_{\openWclosed{\week'}}^{\node}
            }
        \eqfinv 
        \label{eq:NodalCostToGoHD-3}
          \\
     & 
        \bsigmaf{\va{\RecourseControl}_{\openWclosed{\week'}}^{\node}}
        \subseteq 
        \bsigmaf{\va{\Uncertain}_{\openWclosed{{\week}}}, \cdots, 
        \va{\Uncertain}_{\openWclosed{{\week'}\predecessor}},  \va{\Uncertain}_{\openWclosed{\week'}}}
                 \eqfinv 
                 \label{eq:NodalCostToGoHD-6}
        \\ 
        & 
        \bsigmaf{\va{\FlowNode}_{\openWclosed{\week'}}^{\node}}
        \subseteq 
        \bsigmaf{\va{\Uncertain}_{\openWclosed{{\week}}}, \cdots, 
        \va{\Uncertain}_{\openWclosed{{\week'}\predecessor}},  \va{\Uncertain}_{\openWclosed{\week'}}}
                 \eqfinv 
                 \label{eq:NodalCostToGoHD-7}
        \\
    &
      \eqref{eq:StorageLevelBounds_weekly} \eqsepv 
       \eqref{eq:RecourseControlBounds_weekly}
            \eqfinp
            \label{eq:NodalCostToGoHD-8}  
    \end{align}
\end{subequations}
Note that the nodal price cost-to-go function~\eqref{eq:NodalCostToGoHD}
at week~$\week$ depends only on the values of the price decomposition process $\price_{\openWclosed{\week}}^{\node}$
from week~$\week$ to week~$\weeksup$.

If the sequence $\bp{\va{\Uncertain}_{\openWclosed{\binf{\week}}}, \dots,
\va{\Uncertain}_{\openWclosed{\week}}, \dots, \va{\Uncertain}_{\openWclosed{\bsup{\week}}}}$
of uncertainties is weekly independent as in Assumption~\ref{assumption:weekly_independence},
the weekly nodal dynamic programming equations
\begin{subequations}
    \label{eq:NodalPriceValueFunctionDynamicProgrammingHD}
    \begin{align}
   &
             \nBellmanf{\weeklast}{{\node}}\nc{\price^{\node}}\np{\state^{\node}} = \FinalCost^{\node}\np{\state^{\node}} \eqfinv
             \\
 \begin{split}
  &\nBellmanf{\week}{{\node}}\nc{\price^{\node}}\bp{\state^{\node}} =
        \EE  
         \Biggl[
            \min
                       {\InstantaneousCost}_{\week}^{\node}
                    \np{{\state^{\node}},
                       \va{\Uncertain}_{\openWclosed{\week}}^{\node},
                       {\recoursecontrol}_{\openWclosed{\week}}^{\node}}
                       +\price_{\openWclosed{\week}}^{\node} \flownode_{\openWclosed{\week}}^{\node}
                              +
                    \nBellmanf{\week\successor}{{\node}}
                    \nc{\price^{\node}}\bp{{\dynamics}_{\week}^{\node}\np{{\state^{\node}},
                                                   \va{\Uncertain}_{\openWclosed{\week}}^{\node},
                                                   {\recoursecontrol}_{\openWclosed{\week}}^{\node}}}
              \Biggr] \eqfinv
\end{split}
              \\
              & \qquad\ \qquad\qquad\quad \st \quad
              \eqref{eq:WeeklyNodalBalanceConstraint} \eqsepv 
            \eqref{eq:StorageLevelBounds_weekly} \eqsepv 
            \eqref{eq:RecourseControlBounds_weekly}  
             \eqfinv
    \end{align}
\end{subequations}
 provide an optimal solution for the problem~\eqref{eq:NodalPriceValueFunctionDetailHD}.
\sloppy We highlight that the hourly uncertainties within the week~\(\va{\Uncertain}_{\openWclosed{\week}}=
  \bp{\va{\Uncertain}_{\whinfsuccessor},
\dots, \va{\Uncertain}_{\wh\successor}, \dots, \va{\Uncertain}_{\wsuccessorhinf}}\)
need not be assumed independent. 

We get that, for all $\state^{\node}\in\STATE_{\week}^{\node}$,
\begin{align}
        \label{eq:NodalPriceValueFunctionAndCostToGoHD}
    \Value^{{\node}}\nc{{\price}^{\node}}\bp{\state^{\node}} = \nBellmanf{\binf{\week}}{{\node}}\nc{\price^{\node}}\bp{\state^{\node}}
    \eqfinv
\end{align}
that is, the nodal price value function~\eqref{eq:NodalPriceValueFunctionDetailHD}
can be computed recursively by solving dynamic programming equations~\eqref{eq:NodalPriceValueFunctionDynamicProgrammingHD}.

To summarize, the nodal price value functions~\eqref{eq:NodalPriceValueFunctionDetailHD}
are multistage stochastic optimisation problems.
We achieve temporal decomposition, for a given
deterministic price decomposition process~$\price$,
by solving the dynamic programming equations~\eqref{eq:NodalPriceValueFunctionDynamicProgrammingHD},
under the assumption of weekly independence of the uncertainties.

\subsection{Relation between nodal and global cost-to-go functions}
\label{subsection:Relation_between_nodal_and_global_cost_to_go_functions}
We first discuss the relation between the nodal price cost-to-go functions~\eqref{eq:NodalCostToGoHD}
and the global problem value function~\eqref{eq:GlobalProblem}.

\begin{proposition}\label{proposition:GlobalProblemBound}
    For any deterministic price decomposition process~$\price$, 
    if the sequence $\bp{\va{\Uncertain}_{\openWclosed{\binf{\week}}}, \dots,
\va{\Uncertain}_{\openWclosed{\week}}, \dots, \va{\Uncertain}_{\openWclosed{\bsup{\week}}}}$
of uncertainties is weekly independent as in Assumption~\ref{assumption:weekly_independence}, the
sum of the initial nodal price cost-to-go functions and
the arc price value function provides
a lower bound for the global problem~\eqref{eq:GlobalProblem}
    \begin{align}
        \label{eq:GlobalProblem_NodesBound}
        \sum_{\node\in\NODE} \nBellmanf{\binf{\week}}{{\node}}\nc{\price^{\node}}\bp{\stock_{\winfhinf}} + \Value^{\ARC}\nc{\price}
          \leq  \nBellmanf{\binf{\week}}{}\bp{\stock_{\winfhinf}}  
          \eqfinv
    \end{align}
and hence a lower bound for the initial global cost-to-go function (see Equation~\eqref{eq:GlobalProblemOptimalityCostToGo}).
\end{proposition}

\begin{proof}
     Equation~\eqref{eq:GlobalProblem_bound} holds for any deterministic price decomposition process~$\price$. 
    Under Assumption~\ref{assumption:weekly_independence}, for all $\node\in\NODE$,
    $
      \Value^{\node}\nc{\price^{\node}}\bigl(\stock_{\winfhinf}^{\node}\bigr)
      = \nBellmanf{\binf{\week}}{{\node}}\nc{\price^{\node}}\bigl(\stock_{\winfhinf}^{\node}\bigr),
     $   
    by Equation~\eqref{eq:NodalPriceValueFunctionAndCostToGoHD}, and
    $
      \GlobalProblem\bigl(\stock_{\winfhinf}\bigr)
      = \nBellmanf{\binf{\week}}{}\bigl(\stock_{\winfhinf}\bigr),
    $
    by Equation~\eqref{eq:GlobalProblemOptimalityCostToGo}.
    Substituting these equalities into Equation~\eqref{eq:GlobalProblem_bound} yields
    Equation~\eqref{eq:GlobalProblem_NodesBound}.
\end{proof}

We now study the relation between the nodal price cost-to-go functions~\eqref{eq:NodalCostToGoHD}
and the global cost-to-go functions~\eqref{eq:GlobalCostToGoHD}
for all $\week\in\WEEK$.
For this purpose, 
we first define the transport price cost-to-go functions as the constant functions given by 
\begin{subequations}\label{eq:ArcPriceCostToGoFunction} 
    \begin{align} 
        \nBellmanf{\week}{\ARC}\nc{\price} =
        \mathbb{E} 
        \vast[ \min & 
            \sum_{\week'=\week}^{\weeksup}  \left(\sum_{\arc\in\ARC} 
                        \InstantaneousCost_{\week'}^{\arc}\bp{
                            \va{\FlowArc}_{\openWclosed{\week'}}^{\arc}    
                        }
                - {\proscal{\NODEARCMATRIX\transp {\price_{\openWclosed{\week'}}}}{\va{\FlowArc}_{\openWclosed{\week'}}}}
                    \right)
                    \vast]\\
        \st & \quad  \forall \arc\in\ARC \eqsepv \forall \week'\in\ic{\week, \weeksup} \eqsepv
         \bsigmaf{\va{\FlowArc}_{\openWclosed{\week'}}^{\arc}}
        \subseteq
        \bsigmaf{ \va{\Uncertain}_{\openWclosed{{\week}}}, 
                                            \cdots, 
                                            \va{\Uncertain}_{\openWclosed{\week'}}}\eqfinp
    \end{align}
\end{subequations}
Since there is no temporal coupling in the transport problem,
it is easy to see that the initial transport cost-to-go function is 
equal to the arc price value function~\eqref{eq:ArcPriceValueFunction}:
\begin{align}
    \Value^{\ARC}\nc{\price} = \nBellmanf{\binf{\week}}{\ARC}\nc{\price} \eqfinp
\end{align}

Note that the 
 transport price cost-to-go function $\nBellmanf{\week}{\ARC}\nc{\price}$ 
 definition does not depend on the nodal price value functions information structures, and does not
 depend on any nodal storage level.
Using~\eqref{eq:ArcPriceCostToGoFunction}, we can now state the relation between the nodal price cost-to-go functions
and the global cost-to-go function for all $\week\in\WEEK$.

\begin{proposition} 
    \label{proposition:GlobalCostToGoBound}
    For any deterministic price decomposition process~$\price$, if the sequence
     $\bp{\va{\Uncertain}_{\openWclosed{\binf{\week}}}, \dots,
    \va{\Uncertain}_{\openWclosed{\week}}, \dots, \va{\Uncertain}_{\openWclosed{\bsup{\week}}}}$
    of uncertainties is weekly independent as in Assumption~\ref{assumption:weekly_independence},
    then, for all $\week\in\WEEK$,  the sum of the nodal price cost-to-go functions
    and of the transport price cost-to-go function provides a lower bound
    for the global cost-to-go function: 
    \begin{align}
        \sum_{\node\in\NODE} \nBellmanf{{\week}}{{\node}}\nc{\price^{\node}}\bp{\state^{\node}} 
        +\nBellmanf{\week}{\ARC}\nc{\price}
        \leq 
        \nBellmanf{{\week}}{}\bp{\state}\eqsepv \forall\state = \bseqp{\state_{\week}^{\node}}{\node\in\NODE}\in\STATE_{\week} \eqfinp
\end{align}
\end{proposition}

\begin{proof}
    The proof is similar to the one in~\cite[Proposition 7.3.10]{pacaud:tel-02134163}.
   
   The dynamic programming Equations~\eqref{eq:GlobalDP_HD}
   satisfied by global cost-to-go functions~\eqref{eq:GlobalCostToGoHD} can
   be rewritten, 
    for all $\week \in \WEEK$ and  
    for all $\state = \bseqp{\state_{\week}^{\node}}{\node\in\NODE}\in\STATE_{\week}$,
    as
\begin{subequations}
    \label{eq:DynProgGlobalHD}
    \begin{align}
       \nBellmanf{{\weeklast}}{}\bp{\state} & = \sum_{\node\in\NODE} \FinalCost^{\node}\np{\state^{\node}}
        \\
        \begin{split}
            \nBellmanf{{\week}}{}\bp{\state} & = 
              \EE  
            \vast[
                    \min
            \Biggl\{ 
                \sum_{\node\in\NODE} 
                {\InstantaneousCost}_{\week}^{\node}
                \np{{\state^{\node}},
                       \va{\Uncertain}_{\openWclosed{\week}}^{\node},
                       {\recoursecontrol}_{\openWclosed{\week}}^{\node}}
                +
                \sum_{\arc\in\ARC}   
                    \InstantaneousCost_{\week}^{\arc}\bp{
                        {\flowarc}_{\openWclosed{\week}}^{\arc}    
                    }
               \\
                &
                \qquad
                \qquad
                \quad
                \qquad
                \qquad
                \quad
                \qquad
                \quad
               + 
                \nBellmanf{{\week\successor}}{}\Bp{ \bseqp{
                                                    {\dynamics}_{\week}^{\node}\np{{\state^{\node}},
                                                   \va{\Uncertain}_{\openWclosed{\week}}^{\node},
                                                   {\recoursecontrol}_{\openWclosed{\week}}^{\node}}}{\node\in\NODE}  }
          \Biggr\}  \vast] \eqfinv
        \end{split}
        \\
        &\st \nonumber
         \\
        & \flownode_{\openWclosed{\week}} - \NODEARCMATRIX {\flowarc}_{\openWclosed{\week}}  =0 \label{eq:CouplingConstraintGlobalDynProgHD}
            \eqfinv 
         \\ 
         &
        \eqref{eq:WeeklyNodalBalanceConstraint} \eqsepv
        \eqref{eq:StorageLevelBounds_weekly} \eqsepv 
        \eqref{eq:RecourseControlBounds_weekly}
        \eqfinp
    \end{align}
\end{subequations}
For $\week\in\WEEK$, we define the global price cost-to-go function~$\nBellmanf{{\week}}{}\nc{\price}$ associated
with the global cost-to-go problem~\eqref{eq:GlobalCostToGoHD} 
by dualizing the coupling constraint~\eqref{eq:GlobalCostToGoHD-10}:
\begin{subequations}
    \label{eq:DualGlobalCostToGoHD}
    \begin{align}
    \begin{split}  
           \nBellmanf{\week}{}\nc{\price}\bp{\state} =\label{eq:DualGlobalCostToGoHD-1}
        &  
          \min
        \EE\vast[ \sum_{\week'=\week}^{\bsup{\week}}
      \Biggl(  \sum_{\node\in\NODE}       \left(
                    \InstantaneousCost_{\week'}^{\node}\bp{
                            \va{\State}_{\week'}^{\node}, 
                            \va{\Uncertain}_{\openWclosed{\week'}}^{\node},
                            \va{\RecourseControl}_{\openWclosed{\week'}}^{\node}}
                        +
                        \FinalCost^{\node}\bp{\va{\State}_{\weeklast}^{\node}}
                    \right)
               \\
               &  \qquad
                    \qquad
                    \qquad
                  +
                    \sum_{\arc\in\ARC}    
                    \InstantaneousCost_{\week'}^{\arc}\bp{
                        \va{\FlowArc}_{\openWclosed{\week'}}^{\arc}    
                    }
               + \proscal{\price_{\openWclosed{\week'}}}{\va{\FlowNode}_{\openWclosed{\week'}}  - \NODEARCMATRIX \va{\FlowArc }_{\openWclosed{\week'}}}
               \Biggr)
        \vast] 
        \end{split} 
        \\
       & \st\eqsepv \forall \week'\in \ic{\week,\bsup{\week}} \nonumber
          \\
    & 
    \va{\State}_{\week}            =  \label{eq:DualGlobalCostToGoHD-2}
            \state
                \eqfinv
    \\
    &
        \va{\State}_{{\week'}\successor}^{\node} 
            = 
            \dynamics_{\week'}^{\node}\Bp{
                \va{\State}_{\week'}^{\node},
                \va{\Uncertain}_{\openWclosed{\week'}}^{\node},
                \va{\RecourseControl}_{\openWclosed{\week'}}^{\node}
            }
            \eqsepv \forall \node\in\NODE \eqfinv  \label{eq:DualGlobalCostToGoHD-3}
    \\
     & 
     \nodeBalance_{\openWclosed{\week'}}^{\node}\bp{
                    \va{\Uncertain}_{\openWclosed{\week'}}^{\node},
                    \va{\RecourseControl}_{\openWclosed{\week'}}^{\node}
     }  = \va{\FlowNode}_{\openWclosed{\week'}}^{\node}
            \eqsepv \forall \node\in\NODE \label{eq:DualGlobalCostToGoHD-4}
            \eqfinv 
      \\
     & 
        \bsigmaf{\va{\RecourseControl}_{\openWclosed{\week'}}^{\node}}
        \subseteq 
      \bsigmaf{\va{\Uncertain}_{\openWclosed{{\week}}}, \cdots, 
                                 \va{\Uncertain}_{\openWclosed{{\week'}\predecessor}},  
                                 \va{\Uncertain}_{\openWclosed{\week'}}}
            \eqsepv \forall \node\in\NODE \eqfinv \label{eq:DualGlobalCostToGoHD-6}
     \\ 
        & 
        \bsigmaf{\va{\FlowNode}_{\openWclosed{\week'}}^{\node}}
        \subseteq 
        \bsigmaf{\va{\Uncertain}_{\openWclosed{{\week}}}, \cdots, 
                                 \va{\Uncertain}_{\openWclosed{{\week'}\predecessor}},  
                                 \va{\Uncertain}_{\openWclosed{\week'}}}
            \eqsepv \forall \node\in\NODE \eqfinv \label{eq:DualGlobalCostToGoHD-7}
                    \\
        & 
        \bsigmaf{\va{\FlowArc}_{\openWclosed{\week'}}^{\arc}}
        \subseteq  \bsigmaf{\va{\Uncertain}_{\openWclosed{{\week}}}, \cdots, 
                                 \va{\Uncertain}_{\openWclosed{{\week'}\predecessor}},  
                                 \va{\Uncertain}_{\openWclosed{\week'}}}
            \eqsepv \forall \arc\in\ARC \eqfinv \label{eq:DualGlobalCostToGoHD-8}
    \\
    &
      \eqref{eq:StorageLevelBounds_weekly} \eqsepv 
       \eqref{eq:RecourseControlBounds_weekly}
        \eqsepv \forall \hour\in\HOUR \eqsepv\forall \node \in \NODE \label{eq:DualGlobalCostToGoHD-9}
\eqfinp
        \end{align}
\end{subequations}
Under Assumption~\ref{assumption:weekly_independence} of weekly independence of the uncertainties,
the sequence of global price cost-to-go 
functions~$\bseqp{\nBellmanf{{\week}}{}\nc{\price}}{\week\in\cup\na{\WEEK\weeklast}}$
satisfies the following dynamic programming equations for all
$\state = \bseqp{\state_{\week}^{\node}}{\node\in\NODE}\in\STATE_{\week}$:
\begin{subequations}
    \label{eq:DynProgGlobalPriceHD}
    \begin{align}
        \label{eq:DynProgGlobalPriceHD-1}
       \nBellmanf{{\weeklast}}{}\nc{\price}\bp{\state} & = \sum_{\node\in\NODE} \FinalCost^{\node}\np{\state^{\node}}
        \\
        \begin{split}
            \label{eq:DynProgGlobalPriceHD-2}
            \nBellmanf{{\week}}{}\nc{\price}\bp{\state} & = 
             \EE  
            \vast[
                \min_{
                    {\recoursecontrol_{\openWclosed{\week}},
                     \flownode_{\openWclosed{\week}}},
                      \flowarc_{\openWclosed{\week}}
                      }
              \Biggl\{ 
                \sum_{\node\in\NODE} 
                {\InstantaneousCost}_{\week}^{\node}
                \np{{\state^{\node}},
                       \va{\Uncertain}_{\openWclosed{\week}}^{\node},
                       {\recoursecontrol}_{\openWclosed{\week}}^{\node}}
                +    \sum_{\arc\in\ARC} 
                    \InstantaneousCost_{\week}^{\arc}\bp{
                        {\flowarc}_{\openWclosed{\week}}^{\arc}    
                    }
                               \\
                & 
                \qquad\qquad
                +
                \proscal{\price_{\openWclosed{\week}}}{\flownode_{\openWclosed{\week}} - \NODEARCMATRIX {\flowarc}_{\openWclosed{\week}}} 
                +    \nBellmanf{{\week\successor}}{}\nc{\price}\Bp{\bseqp{
                    {\dynamics}_{\week}^{\node}\np{{\state^{\node}},
                                                   \va{\Uncertain}_{\openWclosed{\week}}^{\node},
                                                   {\recoursecontrol}_{\openWclosed{\week}}^{\node}}}{\node\in\NODE}    }
         \Biggr\}  \vast] \eqfinv
        \end{split}
        \\
        &\st \qquad
         \eqref{eq:WeeklyNodalBalanceConstraint} 
        \eqsepv
         \eqref{eq:StorageLevelBounds_weekly} 
         \eqsepv
            \eqref{eq:RecourseControlBounds_weekly}
        \eqfinp \label{eq:DynProgGlobalPriceHD-4}
    \end{align}
\end{subequations}
We know that, for any deterministic price decomposition process~$\price$,
\begin{align}
      \nBellmanf{{\week}}{}\nc{\price}\bp{\state} \leq   \nBellmanf{{\week}}{}\bp{\state}.
\end{align}
Now we prove by induction that,
$\forall\week\in\WEEK$, $\forall  \state = \bseqp{\state_{\week}^{\node}}{\node\in\NODE}\in\STATE_{\week}$, we have that 
\begin{align} 
    \nBellmanf{{\week}}{}\nc{\price}\bp{\state}  = \sum_{\node\in\NODE} \nBellmanf{{\week}}{{\node}}\nc{\price^{\node}}\bp{\state^{\node}} 
        + \nBellmanf{\week}{\ARC}\nc{\price}\eqfinp\label{eq:InductionHypothesisGlobalCostToGoHD}
\end{align}
The property~\eqref{eq:InductionHypothesisGlobalCostToGoHD} holds for $\week=\weeklast$ since~$\nBellmanf{\weeklast}{\ARC}\nc{\price}= 0$
and $\nBellmanf{\weeklast}{{\node}}\nc{\price}\np{\state} = \FinalCost^{\node}\np{\state^{\node}}$.
Let $\week\in\WEEK$ be such that~\eqref{eq:InductionHypothesisGlobalCostToGoHD} is true for $\week\successor$. 
Then, for all $\state = \bseqp{\state_{\week}^{\node}}{\node\in\NODE}\in\STATE_{\week}$,
the recursive equation in~\eqref{eq:DynProgGlobalPriceHD} can be rewritten as
\begin{subequations}
    \begin{align}
       \begin{split}
            \nBellmanf{{\week}}{}\nc{\price}\bp{\state} & = 
             \EE  
            \vast[
                \min
                _{\bseqp{
                \recoursecontrol_{\openWclosed{\week}}^{\node},
                             \flownode_{\openWclosed{\week}}^{\node}, 
                             \flowarc_{\openWclosed{\week}}^{\arc}
                             }{\node\in\NODE}}
              \Biggl\{ 
                \sum_{\node\in\NODE} 
                {\InstantaneousCost}_{\week}^{\node}
                \np{{\state^{\node}},
                       \va{\Uncertain}_{\openWclosed{\week}}^{\node},
                       {\recoursecontrol}_{\openWclosed{\week}}^{\node}}
                +    \sum_{\arc\in\ARC}   
                    \InstantaneousCost_{\week}^{\arc}\bp{
                        {\flowarc}_{\openWclosed{\week}}^{\arc}    
                    }
            \\
                & 
                \qquad\qquad
                     +\proscal{\price_{\openWclosed{\week}}}{\flownode_{\openWclosed{\week}} - \NODEARCMATRIX {\flowarc}_{\openWclosed{\week}}} 
                +      
            \sum_{\node\in\NODE} \nBellmanf{{\week\successor}}{{\node}}\nc{\price^{\node}}\bp{ {\dynamics}_{\week}^{\node}\np{{\state^{\node}},
                                                   \va{\Uncertain}_{\openWclosed{\week}}^{\node},
                                                   {\recoursecontrol}_{\openWclosed{\week}}^{\node}}} 
        +\nBellmanf{\week\successor}{\ARC}\nc{\price}
                                                   \Biggr\}  \vast] \eqfinv
        \end{split}
        \\
        &\st \qquad
            \eqref{eq:WeeklyNodalBalanceConstraint} 
        \eqsepv
         \eqref{eq:StorageLevelBounds_weekly} 
         \eqsepv
            \eqref{eq:RecourseControlBounds_weekly}
        \eqfinp 
    \end{align}
\end{subequations}
We can rearrange the terms in the expectation and the minimization,
separating the nodal terms from the transport terms, and thus get 
\begin{subequations}
    \begin{align}
            \nBellmanf{{\week}}{}\nc{\price}\bp{\state}  = 
          &  \nonumber
            \\
        \qquad  \qquad   \qquad 
        \begin{split}\label{eq:DemCostToGo-1}
            &
             \EE  
            \vast[
                \min_{\bseqp{
                            \recoursecontrol_{\openWclosed{\week}}^{\node},
                             \flownode_{\openWclosed{\week}}^{\node}
                             }{\node\in\NODE}} \Biggl\{ 
                \sum_{\node\in\NODE} 
             \biggl(   {\InstantaneousCost}_{\week}^{\node}
                \np{{\state^{\node}},
                       \va{\Uncertain}_{\openWclosed{\week}}^{\node},
                       {\recoursecontrol}_{\openWclosed{\week}}^{\node}}
                        +\price^{\node}_{\openWclosed{\week}}\flownode_{\openWclosed{\week}}^{\node}  
                       \\& 
                          \qquad\qquad\qquad\qquad\qquad
                           +
                         \nBellmanf{{\week\successor}}{{\node}}\nc{\price^{\node}}
                            \bp{ {\dynamics}_{\week}^{\node}\np{{\state^{\node}},
                                                   \va{\Uncertain}_{\openWclosed{\week}}^{\node},
                                                   {\recoursecontrol}_{\openWclosed{\week}}^{\node}}} 
                  \biggr)                  \Biggr\}  \vast]
       \end{split}
         \\
        & \qquad\qquad \st \qquad
         \eqref{eq:WeeklyNodalBalanceConstraint} 
        \eqsepv
         \eqref{eq:StorageLevelBounds_weekly} 
         \eqsepv
            \eqref{eq:RecourseControlBounds_weekly}        
        \eqfinp 
        \label{eq:DemCostToGo-4}
        \\
        \quad  
       \qquad
       \begin{split} \label{eq:DemCostToGo-2}
       &    + \EE  
            \vast[          
                 \min_{\flowarc_{\openWclosed{\week}}}
                \Biggl\{ 
                      \sum_{\arc\in\ARC}  
                    \InstantaneousCost_{\week}^{\arc}\bp{
                        {\flowarc}_{\openWclosed{\week}}^{\arc}    
                    }
               -\proscal{\NODEARCMATRIX\transp \price_{\openWclosed{\week}}}{{\flowarc}_{\openWclosed{\week}}} 
            + \nBellmanf{\week\successor}{\ARC}\nc{\price}
                                                   \Biggr\}  \vast] \eqfinp
        \end{split}
    \end{align}
\end{subequations}
Note that the constraints~\eqref{eq:DemCostToGo-4} concern only the
minimizations over the nodal decisions~$\recoursecontrol_{\openWclosed{\week}}^{\node}$, and $\flownode_{\openWclosed{\week}}^{\node}$.
We can exchange the sum over nodes and the minimizations and expectations in line~\eqref{eq:DemCostToGo-1} 
 to get
\begin{subequations}
    \begin{align}
            \nBellmanf{{\week}}{}\nc{\price}\bp{\state} & = \nonumber
            \\
            & \sum_{\node\in\NODE} 
             \vast(   
            \EE  
            \Biggl[
                \min_{{ 
                    \recoursecontrol_{\openWclosed{\week}}^{\node},
                     \flownode_{\openWclosed{\week}}^{\node}}}
               {\InstantaneousCost}_{\week}^{\node}
                \np{{\state^{\node}},
                       \va{\Uncertain}_{\openWclosed{\week}}^{\node},
                       {\recoursecontrol}_{\openWclosed{\week}}^{\node}}
                        +\price^{\node}_{\openWclosed{\week}}\flownode_{\openWclosed{\week}}^{\node}  
                           +
                         \nBellmanf{{\week\successor}}{{\node}}\nc{\price^{\node}}\bp{ {\dynamics}_{\week}^{\node}\np{{\state^{\node}},
                                                   \va{\Uncertain}_{\openWclosed{\week}}^{\node},
                                                   {\recoursecontrol}_{\openWclosed{\week}}^{\node}}} 
                                 \Biggr]\eqfinv \nonumber
                                 \\
                            &    \qquad\qquad\qquad \st \qquad
                                 \eqref{eq:WeeklyNodalBalanceConstraint} 
                                \eqsepv
                                \eqref{eq:StorageLevelBounds_weekly} \eqsepv
                                \eqref{eq:RecourseControlBounds_weekly}
                                  \vast)  
                                        \\
                & 
            + \EE  
            \vast[          \min_{\flowarc_{\openWclosed{\week}}}
                \Biggl\{ 
                      \sum_{\arc\in\ARC}  
                    \InstantaneousCost_{\week}^{\arc}\bp{
                        {\flowarc}_{\openWclosed{\week}}^{\arc}    
                    }
               -\proscal{\NODEARCMATRIX\transp \price_{\openWclosed{\week}}}{{\flowarc}_{\openWclosed{\week}}} 
            + \nBellmanf{\week\successor}{\ARC}\nc{\price}
                                                   \Biggr\}  \vast] \eqfinp
    \end{align}
\end{subequations}
Using the dynamic programming Equations~\eqref{eq:NodalPriceValueFunctionDynamicProgrammingHD}, we can see that
the terms inside the sum over the nodes are equal to the nodal price
 cost-to-go functions~$\nBellmanf{{\week}}{{\node}}\nc{\price^{\node}}\np{\state^{\node}}$
 defined in~\eqref{eq:NodalCostToGoHD}. 
Finally, from definition~\eqref{eq:ArcPriceCostToGoFunction}, we can see that the minimisation
over~$\flowarc_{\openWclosed{\week}}$ is equal to the transport price cost-to-go function~$\Value^{\ARC}_{\week}\nc{\price}$.
Therefore, we obtain, $\forall\week\in\WEEK$, $\forall \state = \bseqp{\state_{\week}^{\node}}{\node\in\NODE}\in\STATE_{\week}$, that
\begin{align}
        \nBellmanf{{\week}}{}\nc{\price}\bp{\state} = 
        \sum_{\node\in\NODE} \nBellmanf{{\week}}{{\node}}\nc{\price^{\node}}\np{\state^{\node}}
         + \nBellmanf{\week}{\ARC}\nc{\price} \eqfinp
\end{align}
This ends the proof.
\end{proof}

\subsection{Designing global policies from the nodal cost-to-go functions}
\label{subsection:Designing_global_policies_from_the_nodal_cost_to_go_functions}

We recall that, in this work, we do not intend to create a scheduling program for
the multinode energy system, but rather to compute usage values (through the computation
of Bellman or cost-to-go functions) to be used in simulations
while carrying out prospective studies over a set of uncertainty scenarios.
In this context, we present here how, from the results in~\S\ref{subsection:Relation_between_nodal_and_global_cost_to_go_functions},
we obtain proxies  for the global cost-to-go functions, hence strategies admissible for
the global optimisation problem~\eqref{eq:GlobalProblem}.

We use the result in Proposition~\ref{proposition:GlobalCostToGoBound} to
build proxies for global cost-to-go functions
from the nodal price cost-to-go functions. 
Note that the arc price cost-to-go functions~$\bseqp{\nBellmanf{\week}{\ARC}\nc{\price}}{\week\in\WEEK}$
 {in~\eqref{eq:ArcPriceCostToGoFunction}} 
do not depend on the nodal storage levels $\state^{\node}$, and are therefore
constant for a given deterministic price decomposition process~$\price$. As a consequence,
they do not influence the global policies obtained from the proxies.

Once the transport price cost-to-go functions~$\nBellmanf{\week}{\ARC}\nc{\price}$ 
are computed for all $\week\in\WEEK$ using definition~\eqref{eq:ArcPriceCostToGoFunction},
and the nodal price cost-to-go functions $\nBellmanf{{\week}}{{\node}}\nc{\price^{\node}}$
are computed,
for all $\week\in\WEEK$ and $\node\in\NODE$, using the nodal 
dynamic programming equations~\eqref{eq:NodalPriceValueFunctionDynamicProgrammingHD},
we design global policies for the global optimisation problem~\eqref{eq:GlobalProblem} as follows.
First, we take for all $\week\in\WEEK$ and $\state=\nseqp{\state^{\node}}{\node\in\NODE}$,
a proxy for the global cost-to-go function {in~\eqref{eq:GlobalCostToGoHD}} 
as 
 \begin{align}
    \label{eq:GlobalCostToGoProxy}    {\nBellmanf{\week}{}}\nc{\price}\bp{\state} = 
    \sum_{\node\in\NODE} \nBellmanf{{\week}}{{\node}}\nc{\price^{\node}}\np{\state^{\node}} + 
    \nBellmanf{\week}{\ARC}\nc{\price} 
    \eqfinp
\end{align}
Second, this proxy is used to substitute $\nBellmanf{\week}{}$  
while solving a global forward pass for~$\week \in \WEEK$
using the global dynamic programming equations~\eqref{eq:GlobalDP_HD}.
Third, we obtain global policies for the global optimisation problem~\eqref{eq:GlobalProblem}
using Algorithm~\ref{alg:SimulationHD_multinode} 
with
\begin{align}
    \textit{cost-to-go}_{\week}\np{\state}=\sum_{\node\in\NODE} \nBellmanf{{\week}}{{\node}}\nc{\price^{\node}}\np{\state^{\node}} +
    \nBellmanf{\week}{\ARC}\nc{\price} \eqfinv \forall \week\in\WEEK\cup\na{\weeklast}  \eqfinp
\end{align}

Up to this point, we have discussed the bounds on the global cost-to-go functions
and how to obtain global policies from the nodal price value functions
for a given deterministic price decomposition process~$\price$. However,
 we have not yet discussed how to choose this latter~$\price$.

\subsection{Challenges in improving the price decomposition process}
\label{subsection:Challenges_in_improving_the_price_decomposition_process}

In this~\S\ref{subsection:Challenges_in_improving_the_price_decomposition_process},
we discuss how to improve the price decomposition 
process~\eqref{eq:PriceDecompositionProcess}
to obtain a better lower bound~\eqref{eq:GlobalProblem_NodesBound} for the global
 optimisation problem~\eqref{eq:GlobalProblem}.
We want to get a deterministic price decomposition process~$\price$
to compute univariate nodal price cost-to-go functions and 
transport price cost-to-go functions that are good enough to
design global policies.

In~\S\ref{subsubsection:How_to_improve_the_price_decomposition_process}, we present
an algorithm to improve the price decomposition process.
In~\S\ref{subsubsection:Large_size_of_the_price_decomposition_process},
we discuss the large size of the price decomposition process and present 
some options for reducing it, while maintaining a good quality of the results.

\subsubsection{How to improve the price decomposition process?}
\label{subsubsection:How_to_improve_the_price_decomposition_process}

We explain hereafter how to improve the price decomposition process~\eqref{eq:PriceDecompositionProcess}
to obtain a better lower bound~\eqref{eq:GlobalProblem_NodesBound} 
for the global optimisation problem~\eqref{eq:GlobalProblem}.
For this, we employ a gradient-like method to obtain a more refined 
price decomposition process~$\price$, without necessarily aiming for 
convergence (indeed, we recall that we are dealing with a deterministic
 price decomposition process~$\price$, whereas the optimal multipliers, 
 if any, are in general, random).

Due to the large dimension of the price decomposition process, 
we chose to use the {L-BFGS}~\cite{liu1989limited} implementation 
in Ipopt.jl~\cite{OnTheImplementationOfAnInteriorPointFilter}
 and, for that,
we define an oracle function that computes 
the value of the lower bound~$\ValueInf\nc{\price}$ in~\eqref{eq:GlobalPriceValueFunction-3}
and its gradient 
with respect to~$\price$ for a given value of the price decomposition process~$\price$.
The decomposition prices are updated until a stopping criterion is reached, and then 
we obtain a price that is later used to compute a proxy 
for the global cost-to-go functions using equations~\eqref{eq:GlobalCostToGoProxy}.
This procedure is illustrated in Figure~\ref{fig:priceUpdateScheme} and detailed
below.
\begin{figure}[htbp]
    \centering
    \includegraphics[width=0.8\textwidth]{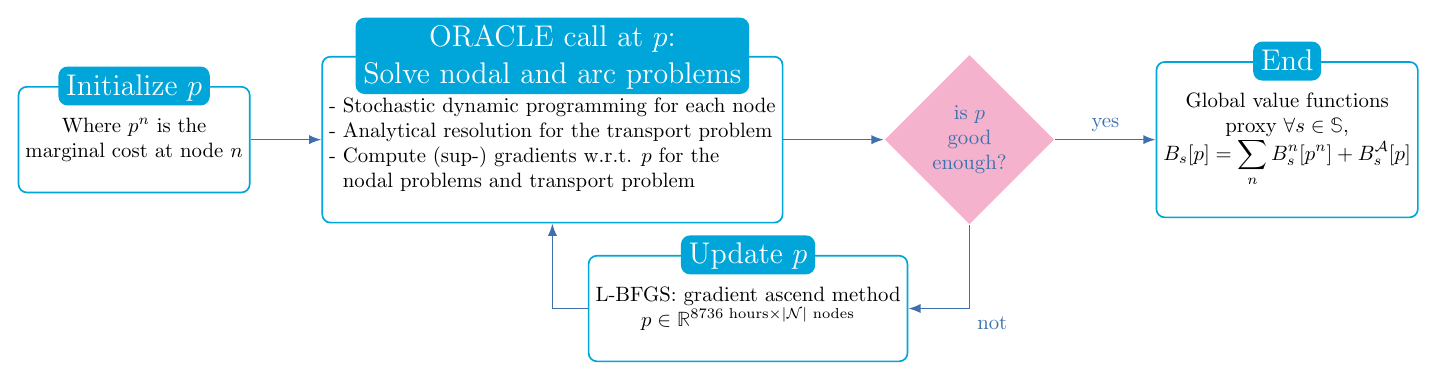}
  \caption{\small Price updates scheme}
    \label{fig:priceUpdateScheme}
\end{figure}
\subsubsubsection{Algorithm to improve the price decomposition process}
\label{subsubsection:Algorithm_to_improve_the_price_decomposition_process}
Our approach is inspired by the Uzawa algorithm~\cite{Carpentier-Cohen:2017_camila}, 
which has previously been applied in~\cite{pacaud2021distributed}. 
Since we consider a deterministic price decomposition process, 
we do not expect to obtain a zero duality gap, as the optimal multipliers are likely non-deterministic. 
Therefore, the objective is not to solve the dual problem to optimality, 
but rather to iteratively improve the deterministic price decomposition process. 
The algorithm is designed to improve the price decomposition process, without guaranteeing convergence.

\begin{remark} \label{remark:gap}
    While a lower bound for the global problem~\eqref{eq:GlobalProblem} is obtained via the Lagrangian 
    relaxation and spatial decomposition, an upper bound can be estimated through Monte Carlo simulations. 
    Specifically, we can use the proxy cost-to-go functions defined in~\eqref{eq:GlobalCostToGoProxy}—evaluated 
    at the resulting price decomposition process—to construct global policies. By simulating system operation 
    under these policies, we compute the expected operational cost, which serves as an upper bound for the 
    original problem. The gap between the lower bound~\eqref{eq:GlobalProblem_NodesBound} and this 
    simulated upper bound provides a practical measure of the quality of the decomposition
    and price decomposition process improvement. Note that  the lower bound is exact, 
    whereas the upper bound, based on Monte Carlo scenarios of noises, is statistical.
\end{remark} 

The main steps of the price decomposition process improvement algorithm are summarized below.
\begin{enumerate}
    \item \textbf{Initialization}:
        set the initial price decomposition process~$\price$.
    \item \textbf{ORACLE call at}~$\price$:        
            the oracle function
         \begin{subequations}\label{eq:OracleFunction}
            \begin{align}
                \oracle: \STOCK_{\winfhinf} \times \RR^{\cardinality{\HOUR}\times\cardinality{\WEEK}\times\cardinality{\NODE}}
                     \to \RR \times \RR^{\cardinality{\HOUR}\times\cardinality{\WEEK}\times\cardinality{\NODE}}
            \end{align}
            is given by 
             \begin{align}
                \oracle\bp{\stock_{\winfhinf}, \price} = 
                \left( \ValueInf\nc{\price}\bp{\stock_{\winfhinf}}, 
                \nabla_{\price}\ValueInf\nc{\price}\bp{\stock_{\winfhinf}} \right) \eqfinv
            \end{align}
            where 
            \begin{align} \label{eq:LagrangianDualValue_p}
              &  \ValueInf\nc{\price}\bp{\stock_{\winfhinf}} = 
                \sum_{\node\in\NODE} \nBellmanf{\binf{\week}}{{\node}}\nc{\price^{\node}}\bp{\stock^{\node}_{\winfhinf}} + 
                \Value^{\ARC}\nc{\price} \eqfinv
                \\
                  &  \nabla_{\price}\ValueInf\nc{\price}\bp{\stock_{\winfhinf}} = 
                    \sum_{\node\in\NODE} \nabla_{\price}\nBellmanf{\binf{\week}}{{\node}}\nc{\price^{\node}}\bp{\stock^{\node}_{\winfhinf}} + 
                    \nabla_{\price}\Value^{\ARC}\nc{\price} \eqfinp
                \end{align}
         \end{subequations}
            We now explain how we compute values and gradients. 
    \begin{enumerate}
        \item \textit{Nodal resolution of production problems (value and gradient)}:
             For each node~$\node\in\NODE$:
             \begin{itemize}
                \item   
                        The nodal value $\nBellmanf{\binf{\week}}{{\node}}\nc{\price^{\node}}$ by backward recursion 
                        using the dynamic programming equations~\eqref{eq:NodalPriceValueFunctionDynamicProgrammingHD}
                        . 
                \item   The nodal gradient $\nabla_{\price}\nBellmanf{\binf{\week}}{{\node}}\nc{\price^{\node}}$
                         by a backward recursion alongside the dynamic programming equations resolution
                      \begin{subequations} 
                        \label{eq:GradientNodalPriceValueFunction}
                        \begin{align}
                            & \nabla_{\price}\nBellmanf{{\weeklast}}{{\node}}\nc{\price^{\node}}\np{\state^{\node}} =  
                                            \nabla_{\price}\FinalCost^{\node}\np{\state^{\node}} \eqfinv
                            \\
                            &  \nabla_{\price}\nBellmanf{{\week}}{{\node}}\nc{\price^{\node}}\np{\state^{\node}} = 
                                \widehat{\espe} \Bigl[  \nabla_{\price^{\node}} \np{
                                            \price_{\week}^{\node} {\flownode_{\openWclosed{\week}}^{\node}}^* }
                                             \\
                                            &
                                    \qquad  \qquad \qquad \qquad 
                                          + \nabla_{\price^{\node}}\nBellmanf{{\week\successor}}{{\node}}\nc{\price^{\node}}\bp{\dynamics\np{\state^{\node},
                                            \va{\Uncertain}_{\openWclosed{\week}}^{\node},
                                        { {\recoursecontrol}_{\openWclosed{\week}}^{\node}}^*}}
                                                            \Bigr]
                                                            \eqfinv       \nonumber                 
                            \end{align}
                    \end{subequations}
                    where the optimal values of ${\flownode_{\openWclosed{\week}}^{\node}}^*$
                    and ${{\recoursecontrol}_{\openWclosed{\week}}^{\node}}^*$ 
                    for all $\week\in\WEEK$ are given by the solutions of the dynamic programming 
                    equations~\eqref{eq:NodalPriceValueFunctionDynamicProgrammingHD}
                    at state~$\state^{\node}$
                    and week~$\week$. The expectation $\widehat{\espe}$ 
                    is defined in Remark~\ref{remark:WeeklyProbability_multinode}
                    (Equation~\eqref{eq:weekly_probability_Expectation_multinode}).
            \end{itemize}           
          \item
            \textit{Transport problem resolution (value and gradient)}:
            \begin{itemize}
                \item $\Value^{\ARC}\nc{\price}$ is computed using
            Equation~\eqref{eq:ArcPriceValueFunction3}.
              \item $
                    \nabla_{\price}\Value_{\week}^{\ARC}\nc{\price} = 
                    -  \NODEARCMATRIX\transp \flowarc^*$, where~ $\flowarc^*$ is the optimal solution from Equation~\eqref{eq:ArcPriceValueFunction3}.
            \end{itemize}
            {The transport problem does not have temporal coupling and, therefore, the gradient
            can be computed directly from the optimal flow ${\FlowArc}$ from the resolution of 
            the arc price value function~\eqref{eq:ArcPriceValueFunction3} for a given
            price decomposition process~$\price$.}
    \end{enumerate}
  \item
    \textbf{Update} $\price$: {L-BFGS} algorithm implementation in Ipopt.
  \item
    \textbf{Stopping criterion}: the goal is not to solve the dual problem to optimality, 
    but to obtain a sufficiently good price decomposition process; 
    the algorithm terminates when the oracle value does not improve for two consecutive iterations.
  \item
    \textbf{Return}: the price decomposition process~$\price$ to compute 
    the nodal price cost-to-go functions and the transport price cost-to-go functions 
    for all $\week\in\WEEK$.
\end{enumerate}

\subsubsubsection{Comments on the price decomposition process improvement algorithm}
    The gradient for the nodal price value functions $\nabla_{\price}\nBellmanf{{\week}}{{\node}}\nc{\price^{\node}}$
    is computed using~\cite[Theorem 3.5]{franc2023differentiabilityregularizationparametricconvex}, where
    the authors show, under suitable regularity assumptions 
    (see details in~\cite[Assumptions 1, 2 and 3]{franc2023differentiabilityregularizationparametricconvex}), 
     that the gradient of a parametric convex value function
    can be computed, using the backward recursion in
    Equations~\eqref{eq:GradientNodalPriceValueFunction} alongside the
    dynamic programming equations resolution.
    We note that the gradient backward recursion in the referenced paper was 
    established for a decision-hazard information structure. However, since the 
    recursion is expressed in terms of the optimal control $\recoursecontrol^*$, the same 
    formulation can be used for the hazard-decision case, as long as $\recoursecontrol^*$ is computed
    according to the respective information structure.
    Although we cannot be sure to satisfy the assumptions recalled above, we however use the 
    ``gradient" computation as in~\cite{franc2023differentiabilityregularizationparametricconvex}.

    This approach using Equations~\eqref{eq:GradientNodalPriceValueFunction}
    potentially improves the performance with respect to what was done in the
    PhD thesis~\cite{pacaud:tel-02134163} where the gradient was 
    estimated from the expression
    \( 
        \nabla_{\price}\ValueInf\nc{\price}\bp{\stock_{\winfhinf}} = 
        \EE\left[ \va{\FlowNode^*} -\NODEARCMATRIX \flowarc^*\right] 
\),     using Monte Carlo simulations that can be computationally expensive.
    The values of~$\va{\FlowNode^*}$ are the optimal values 
    from Algorithm~\ref{alg:SimulationHD_multinode}.

    \subsubsubsection{Comments on the nodal problems resolution with SDP}
    We compute the nodal price cost-to-go functions $\nBellmanf{{\week}}{{\node}}\nc{\price^{\node}}$
    for all $\week\in\WEEK$ and $\node\in\NODE$,
    using a Stochastic Dynamic Programming-like approach obtaining the gradient
    directly in the same state discretization grid without need of simulations. 
    Since the nodal cost-to-go functions are
    computed once the global problem is decomposed into
    the nodal problems, the state dimension is small enough to use {SDP}. 
     Moreover, this framework allows for computing the nodal cost-to-go
     functions and their gradients in parallel for all nodes $\node\in\NODE$. 

    Unlike the standard numerical {SDP} approach, we do not discretize the control space.
    Instead, we formulate each weekly problem in the dynamic programming equations as a linear problem, 
    modelling the cost-to-go functions as piecewise linear functions. For each week $\week \in \WEEK$ and each
    state $\state^{\node}$ in the nodal state grid, we add a cut to the nodal cost-to-go function, in the same
    way as in the {SDDP} algorithm. This process yields a lower piecewise linear approximation 
    $\nBellmanfHat{{\week}}{{\node}}\nc{\price^{\node}}$ of the nodal cost-to-go functions 
    $\nBellmanf{{\week}}{{\node}}\nc{\price^{\node}}$, ensuring that we always have a lower bound
    for the global problem~\eqref{eq:GlobalProblem}.

\begin{remark}\label{remark:nodalResolution_SDDP}
    The nodal price cost-to-go function $\nBellmanf{{\week}}{{\node}}\nc{\price^{\node}}$
    can be computed using {SDDP}, but, in that case, the gradient should be estimated
    using Monte Carlo simulations as in~\cite{pacaud:tel-02134163} and the quality
    of the gradient would depend on the number of scenarios used in the simulations. 
    The use of expressions~\eqref{eq:GradientNodalPriceValueFunction} with a {SDDP} resolution
    of the nodal price value functions would require enlarging the {SDDP} state dimension
    and to include the 
    price decomposition process~$\price$ as in the numerical application
     in~\cite{franc2023differentiabilityregularizationparametricconvex}. 
    However, in this case, the nodal price cost-to-go functions are concave with 
    respect to~$\price$ (and convex with respect to the storage levels), 
    making it not possible to directly apply {SDDP}.      
\end{remark}

\subsubsubsection{Comments on the choice of the L-BFGS algorithm}

We chose the {L-BFGS}~\cite{liu1989limited} algorithm for updating the price decomposition process due to 
its efficiency in handling large-scale optimization problems with many variables. 
{L-BFGS} is a quasi-Newton method that approximates the inverse Hessian matrix using limited memory,
making it well-suited for the high-dimensional price vector encountered in our spatial decomposition approach. 
This choice allows us to iteratively improve the price decomposition process without incurring prohibitive
computational or memory costs. We remark that the {L-BFGS} algorithm is designed to handle 
unconstrained and smooth optimization problems. As we model the nodal price cost-to-go functions as 
piecewise linear functions, we do not expect the objective function to be smooth.
However, when considering a hazard-decision information structure,
the nodal price cost-to-go functions are computed as the expectation
of piecewise linear functions, which induces a smoothing effect in the objective function~\cite{wets1974stochastic}.
An alternative approach to overcoming this difficulty is to utilize a non-smooth optimization algorithm 
for updating the price decomposition process. However, finding an algorithm that can efficiently handle 
the non-smoothness of the objective function in the high-dimensional price decomposition process remains 
a significant challenge, which is beyond the scope of this work. For this reason, even if we are not under 
the hypothesis required to obtain convergence from {L-BFGS}, we still use it in practice, showing its 
effectiveness in improving the price decomposition process.

\bigskip
Summing up, 
the improvement of the price decomposition process is performed
iteratively using a gradient-like method, which relies on the
oracle~$\oracle\np{\price}=\bp{  \ValueInf\nc{\price}\bp{\stock_{\winfhinf}} ,
\nabla_{\price}\ValueInf\nc{\price}\bp{\stock_{\winfhinf}} } $.
However, in practical numerical implementations,
the dimensionality of the price decomposition process can become a significant challenge.
As its size increases, even {L-BFGS} may struggle to handle the optimization efficiently.
In~\S\ref{subsubsection:Large_size_of_the_price_decomposition_process},
we discuss an approach to address this challenge.

\subsubsection{Large size of the price decomposition process}
\label{subsubsection:Large_size_of_the_price_decomposition_process}
We discuss in this~\S\ref{subsubsection:Large_size_of_the_price_decomposition_process} 
the challenges posed by the large size of the price decomposition process
 and potential strategies to address them.

 Due to the hourly granularity of the coupling constraint~\eqref{eq:CouplingConstraint},
  the size of the price decomposition process~\eqref{eq:PriceDecompositionProcess}
 is proportional to the number~$\cardinality{\HOUR}\times\cardinality{\WEEK}=8736$
  of hours in the year, and this 
 becomes a challenge even for a deterministic price decomposition process~$\price$. 
 Recalling also that  the coupling constraint~\eqref{eq:CouplingConstraint} is an energy
  balance constraint at  each node in the system (Kirchhoff's law between nodes and arc flows),
  the size of the price decomposition process  is proportional to the number of nodes in the system. 
  Therefore, when increasing the system size, the size of the price decomposition process grows linearly 
  with the number of nodes. This could be a setback to the {DADP} method, which handles well the oracle 
  step of the price decomposition process improvement (the nodal computations being done in parallel), 
  but that could be limited by the size of the price decomposition process due to memory issues. 

To address this, we consider a strategy that involves aggregating the spatial 
coupling constraints~\eqref{eq:CouplingConstraint} into time blocks that 
are smaller than or equal to one week.
We examine two cases. 
\begin{itemize}
    \item \textbf{Prices each $\mathbb{\hour}$ hours:} we aggregate the coupling constraints
    in blocks of  $\mathbb{\hour}$ hours and, therefore, the size of the price decomposition process
    is proportional to the number~$\mathbb{N_{\hour}}$ of $\mathbb{\hour}$-hours blocks in the week,
     the number of weeks in the year and the number of nodes in the system.
    \item \textbf{Weekly prices:} we aggregate the coupling constraints
    in weekly blocks and, therefore, the size of the price decomposition process
    is proportional to the number (52) of weeks in the year and the number of nodes in the system.
\end{itemize}
Note that the second case is a particular case of the first one, where $\mathbb{\hour} = 168$ hours.
It is important to highlight that the number (168) of hours in the week must be divisible by
the value chosen for $\mathbb{\hour}$, that is, it should divide the week into
$\mathbb{N_{\hour}} =\cardinality{\HOUR} / \mathbb{\hour} $ blocks
of length $\mathbb{\hour}$ hours. 
Each block is denoted by $\mathbb{\hour}_i$ for $i\in\ic{1,\mathbb{N_{\hour}}}$, and the 
set of hours in the week $\HOUR$ is composed of the concatenation
    $\HOUR = \bc{\mathbb{\hour}_1, \mathbb{\hour}_2, \ldots, \mathbb{\hour}_{\mathbb{N_{\hour}}}}$
of these blocks.

From an energy system perspective, this means that the nodal marginal cost in the system is 
stationary during the $\mathbb{\hour}$-hour blocks.  A similar approach was already used by 
Hawkins in his PhD thesis~\cite{Hawkins:2003} but, in that case, the multiplier was static for 
the entire optimization timespan.

Note that, when considering this coupling constraints aggregation, 
we are actually considering different optimisation problems that 
are relaxations of the original one presented in Equations~\eqref{eq:GlobalProblem},
where the constraint~\eqref{eq:globalcoupling} is substituted by
    $
          \sum_{\hour\in\mathbb{\hour}_i}
  \bp{\va{\FlowNode}_{\wh\successor}
  -
  \NODEARCMATRIX \va{\FlowArc }_{\wh\successor}} = 0
  \eqsepv \forall i\in\ic{1,\mathbb{N_{\hour}}} \eqsepv \forall \week\in\WEEK
$.
We remark that we do not modify the temporal structure of the problem, 
where the decisions, uncertainties, and most constraints (except for the spatial coupling)
are still defined in the hourly timescale, but we only aggregate the spatial coupling constraints
in blocks of $\mathbb{\hour}$ hours and, therefore, the price decomposition 
process~${\price}$, introduced when the spatial coupling constraints are relaxed, 
has a size reduced from $\bp{\cardinality{\HOUR}\times\cardinality{\WEEK}\times\cardinality{\NODE}}$
to $\bp{\mathbb{N_{\hour}}\times\cardinality{\WEEK}\times\cardinality{\NODE}}$.
 As a consequence of these aggregations, when solving the node and arc problems,
several hours of the week have an identical price.

In the numerical studies for a multinode system in Section~\ref{section:Modelling_and_numerical_study_in_multinode_systems},
we consider aggregations with
$\mathbb{\hour}=8$ and $\mathbb{\hour}=168$ hours (weekly prices), and we compare the results
obtained with the case where the coupling constraints are not aggregated
(i.e., $\mathbb{\hour}=1$ hour).

\section{Modelling and numerical study in multinode systems}
\label{section:Modelling_and_numerical_study_in_multinode_systems}

We aim to evaluate the performance of the {DADP} method,
as developed in Section~\ref{section:Mixing_spatial_and_temporal_decomposition},
in a large and complex system,
focusing on its scalability and effectiveness in handling increased dimensionality.

This thirty-node system introduces a high complexity due to high number of nodes and interconnections.
This setting also features a large heterogeneity, with some nodes lacking storage units,
and substantial variations in installed capacities and demand profiles across the network.
Due to the high dimensionality of the state space in this thirty-node system,
it is not possible to perform an exhaustive pointwise comparison of the usage values derived
from different methods. However, we focus on aspects that are meaningful at a large scale:
computational effort, the quality of the obtained lower bounds,
and statistical upper bounds. 
We therefore do not attempt a pointwise multivariate usage-value comparison 
for every node; instead, we report simulation-based policy performance (operational cost and energy not served (ENS)) 
and present storage trajectories only for a selected set of representative nodes.  
This focused evaluation compares the scalability of {DADP} (with hourly,  
8‑hour, and weekly price aggregation schemes), against {SDDP}.

In~\S\ref{subsection:Objectives_and_methodological_considerations},
    we outline the objectives and methodological considerations
    made for this numerical study.
    In~\S\ref{subsection:Thirty_node_system_description}, we describe the thirty-node system,
    detailing the market zones, interconnections, and a summary of the system components at each node.
    In~\S\ref{subsection:Numerical_results_for_the_thirty_node_system}, we present the numerical results obtained
    using the {DADP} method with hourly, 8-hour, and weekly price aggregation,
    comparing them with the {SDDP} method in terms of computation times, lower bounds,
    statistical upper bounds, and the empirical indicators
    in simulation.

\subsection{Objectives and methodological considerations}
\label{subsection:Objectives_and_methodological_considerations}

In~\S\ref{subsubsection:Choice_of_software}, we detail the software and implementations used
for the numerical experiments.
In~\S\ref{subsubsection:Uncertainties_modelling}, we describe the models used for the uncertainties
in the thirty-node system.
In~\S\ref{subsubsection:Lower_bounds_computing_times_and_statistical_upper_bounds}, 
we present the performance indicators
used to compare the four methods studied during this numerical experiment.
Finally, in~\S\ref{subsubsection:Comparison_between_policies}, 
we explain how we compare the policies obtained with the different methods
in simulation using the RTE reference uncertainties chronicles.

\subsubsection{Choice of software and implementations}
\label{subsubsection:Choice_of_software}
All implementations are carried out in Julia~\cite{bezanson2017julia}, 
using the JuMP.jl modelling language~\cite{dunning2017jump} with Gurobi as the solver.

\phantomsection
\subsubsubsection{Stochastic Dynamic Programming ({SDP})} 
Custom implementations are used for {SDP} backward recursions to compute the univariate 
cost-to-go functions in the nodal resolution of the {DADP} oracle
  (Equations~\eqref{eq:NodalPriceValueFunctionDynamicProgrammingHD}).

    In this implementation, we formulate each weekly problem in the dynamic programming equations 
    as a linear problem, where the cost-to-go functions are modelled as
    piecewise linear functions. For each week $\week\in\WEEK$ 
    and each state~$\state$ in the state grid, 
    we add a cut to the cost-to-go function, in the same way as in {SDDP} algorithm.
    By doing this, we obtain a lower piecewise linear approximation of the  cost-to-go functions.

\subsubsubsection{Improvement of the price decomposition process}
The improvement of the price decomposition process presented
in~\S\ref{subsubsection:Algorithm_to_improve_the_price_decomposition_process} 
is implemented using the limited memory algorithm (L-BFGS) 
from Ipopt.jl package~\cite{OnTheImplementationOfAnInteriorPointFilter}.
We choose to use a stopping criterion
based on the improvement 
of the lower bound: the algorithm terminates after two consecutive 
iterations with an absolute improvement of less than 100~\euro\ 
(approximately $5 \times 10^{-7}\%$ relative to the lower bound). 

\subsubsubsection{Parallelization strategy}
The nodal backward recursions {in~\eqref{eq:NodalPriceValueFunctionDynamicProgrammingHD}
 and~\eqref{eq:GradientNodalPriceValueFunction}}
(value and gradient, respectively) for the DADP method were computed independently
for different nodes in parallel, using 16~threads.
Nodes without storage, as well as smaller nodes with storage but limited computational 
requirements (e.g., with only few thermal clusters of dispatchable units), were grouped for sequential processing.
The price decomposition process size grows proportionally to the number of nodes,
regardless of whether a node includes storage or not. 
As a result, we encountered 
numerical challenges when implementing 
the {DADP} method with hourly prices, as the price decomposition process became 
too large to be handled efficiently, and the oracle computations became too time-consuming.
This limitation was primarily due to the parallelization strategy used in our implementation, 
which did not scale well with the expanded size of the price decomposition process. Nevertheless,
we were able to run the {DADP} method with hourly prices for a limited number of iterations, 
but did not achieve satisfactory results compared with the {DADP} 
approaches with 8-hour and weekly price aggregation.

\subsubsubsection{Stochastic Dual Dynamic Programming ({SDDP})}
For the global multivariate {SDDP} resolution, we used the {SDDP}.jl package~\cite{dowson_2021_sddp.jl}, 
with as a stopping criterion the improvement of the lower bound with the same tolerance as in Ipopt for {DADP} 
(less than 100~\euro, approximately $5 \times 10^{-7}\%$ relative to the lower bound).

\subsubsubsection{Simulation of policies}
Custom implementations are used for the simulation Algorithm~\ref{alg:SimulationHD_multinode}.
These implementations can handle the multivariate cost-to-go functions obtained with
{SDDP} as well as the univariate cost-to-go functions obtained with {DADP}.

\subsubsection{Models for the uncertainties} 
\label{subsubsection:Uncertainties_modelling} 

The uncertainties at each node {in~\eqref{eq:uncertainty_process_multinode_hour}} 
are given by the net demand (nodal demand load minus 
    the non-dispatchable production), the availability 
    of dispatchable production (expressed in percentage of maximum power) 
    and the inflows into storages (a single aggregated dam per node).
    To model the uncertainties, we have taken a finite number of uncertainty chronicles
    (see definition in Remark~\ref{remark:empirical_probability_multinode})
    provided by {RTE}, with uniform probability. 
    Each uncertainty chronicle is, for each $\node\in\NODE$, a three-dimensional time series
    (net demand, dispatchable production availability, and storage inflows) consisting of a sequence
    of 52 weekly uncertainty vectors,
    each containing 168 hourly observations.

Because we want, on the one hand, to compute cost-to-go functions and, on the other hand, to use them in simulation, 
we consider two subsets of these chronicles as follows.
Note that the intersection of the two subsets is not necessarily empty.

\phantomsection
\subsubsubsection{Set of uncertainty chronicles for cost-to-go functions computing: $\MCHRONICLE$}

We consider, for the computing of the cost-to-go functions, uncertainty chronicles
    $\np{\uncertain}^{\mchronicle} = \bp{ \np{\uncertain_{\openWclosed{\weekinf}}}^{\mchronicle},
                                  \dots,
                                        \np{\uncertain_{\openWclosed{\week}}}^{\mchronicle}, 
                                  \dots,
                                        \np{\uncertain_{\openWclosed{\weeksup}}}^{\mchronicle}}$,
    \begin{subequations}
      with $\np{\uncertain_{\openWclosed{\week}}}^{\mchronicle}=
     \bp{\nseqp{\uncertain_{\openWclosed{\week}}^{\node}}{\node\in\NODE}}^{\mchronicle}$
     for all $\week \in \WEEK$, and $\mchronicle\in\MCHRONICLE$.
     The set $\MCHRONICLE$ is chosen according to a study carried out by {RTE} data scientists, 
who selected a subset of representative scenarios from the full set of uncertainty chronicles. 

      We define the empirical product probability~$\widehat{\prbt}$ (see remark~\ref{remark:WeeklyProbability_multinode}) 
        \begin{align}  \label{eq:weekly_probability_Prbt_multinode_Mchronicles}
          \widehat{\prbt} = \bigotimes_{\week \in \WEEK} \frac{1}{\cardinality{\MCHRONICLE}} 
            \sum_{\mchronicle \in \MCHRONICLE} 
          \delta_{\np{\uncertain_{\openWclosed{\week}}}^{\mchronicle}}
          \eqfinp
        \end{align}
      The set of weekly uncertainty chronicles 
      $\Bseqp{\bseqp{\np{\uncertain_{\openWclosed{\week}}}^{\mchronicle}}{\mchronicle\in\MCHRONICLE} }{\week\in\WEEK}$
      and the empirical product probability $\widehat{\prbt}$ are used to compute the cost-to-go functions and
      their associated nonanticipative policies.
      The expectations in the dynamic programming equations, that depend on the weekly uncertainties, 
      are computed as finite sums
      \begin{equation}
          \widehat{\espe}
          \nc{\zeta(\va{\Uncertain}_{\openWclosed{\week}})} = 
                \frac{1}{\cardinality{\MCHRONICLE}} 
            \sum_{\mchronicle \in \MCHRONICLE} 
            \zeta\bp{\np{\uncertain_{\openWclosed{\week}}}^{\mchronicle}}
            \eqfinv 
          \end{equation}
      where $\np{\uncertain_{\openWclosed{\week}}}^{\mchronicle}$ is a weekly uncertainty 
      chronicle and $\frac{1}{\cardinality{\MCHRONICLE}}$ its corresponding probability.
\end{subequations}

\subsubsubsection{Set of uncertainty chronicles for evaluating policies: $\PCHRONICLE$}

        We consider, for the evaluation of the policies derived from the cost-to-go functions, 
        the uncertainty chronicles~$\np{\uncertain}^{\pchronicle} = \bp{\np{\uncertain_{\openWclosed{\weekinf}}}^{\pchronicle},
                                    \dots,\np{\uncertain_{\openWclosed{\week}}}^{\pchronicle}, 
                                    \dots,\np{\uncertain_{\openWclosed{\weeksup}}}^{\pchronicle}}$,
        with $\np{\uncertain_{\openWclosed{\week}}}^{\pchronicle}=
        \bp{\nseqp{\uncertain_{\openWclosed{\week}}^{\node}}{\node\in\NODE}}^{\pchronicle}$
        for all $\week \in \WEEK$ , and $\pchronicle\in\PCHRONICLE$.

        We define the empirical probability~$\widetilde{\prbt}$ (see remark~\ref{remark:empirical_probability_multinode})
        \begin{align}\label{eq:year_probability_Prbt_multinode_Pchronicles}
            \widetilde{\prbt} = \frac{1}{\cardinality{\PCHRONICLE}} 
            \sum_{\pchronicle \in \PCHRONICLE} 
                \delta_{\bp{\np{\uncertain_{\openWclosed{\weekinf}}}^{\pchronicle},
                 \dots,
                 \np{\uncertain_{\openWclosed{\week}}}^{\pchronicle}, 
                 \dots,
                 \np{\uncertain_{\openWclosed{\weeksup}}}^{\pchronicle}}} 
                 \eqfinp
        \end{align}

    The set of weekly uncertainty chronicles 
    $\Bseqp{\bseqp{\np{\uncertain_{\openWclosed{\week}}}^{\pchronicle}}{\week\in\WEEK}}{\pchronicle\in\PCHRONICLE}$
    is used to simulate the system operation
    when applying the policies derived from the cost-to-go functions
    for the four methods studied. 

    We refer to the set $\PCHRONICLE$ as the \emph{RTE reference uncertainty chronicles}
    and they are used for the comparison between policies.
    The expected values for the empirical indicators
    are computed using the empirical probability $\widetilde{\prbt}$.

\begin{remark}
    \label{remark:yearly_timespan}
    It is common practice in prospective studies to set the year horizon so that 
    the winter period is not split across two years.
    Therefore, we consider a year with 52 weeks, starting in the first week of July 
    and ending in the last week of June of the following year.
    Note that, as the year is composed of 52~weeks, the studied timespan is 364~days
    (instead of 365 or 366~days in a leap year).
\end{remark}

\subsubsection{Algorithms assessment}
\label{subsubsection:Lower_bounds_computing_times_and_statistical_upper_bounds}

For each studied resolution method, we compare the following indicators computed using 
 the empirical product probability~$\widehat{\prbt}$ {in~\eqref{eq:weekly_probability_Prbt_multinode_Mchronicles}}:
\begin{itemize}
    \item Lower bound obtained from the cost-to-go calculation phase,
    \item Computation time for the cost-to-go calculation phase,
    \item Statistical upper bound,
    \item Statistical gap between the lower and the statistical upper bounds.
\end{itemize}

All four methods yield a lower bound. In the {SDDP} cases, the bound is the value of the initial cost-to-go function
at the initial state $\state_0$, and in the {DADP} case it is the value of the Lagrangian dual function at the resulting
price vector~$\price$ after applying the improvement algorithm {in~\eqref{eq:LagrangianDualValue_p}}.

A statistical upper bound is obtained for the four methods by simulating the policies derived 
from the cost-to-go functions. To remain consistent with the use of the empirical product 
probability~$\widehat{\prbt}$ {in~\eqref{eq:weekly_probability_Prbt_multinode_Mchronicles}}
 in computing the cost-to-go functions, the statistical upper
bound is computed using a sample drawn from~$\widehat{\prbt}$. In practice, drawing from the 
empirical product probability~$\widehat{\prbt}$ means obtaining the sample by concatenating 
weekly uncertainty chronicles, 
$\bseqp{{\np{\uncertain_{\openWclosed{\week}}}^{\mchronicle}}}{\week\in\WEEK}$, 
sampled i.i.d. according to $\widehat{\prbt}$. We report the expected operational cost,
the 95\% confidence interval, and the statistical gap across 1000 samples. 
Note that a truly robust statistical assessment would require more samples, since the number of possible 
sequences of weekly chronicles equals $\cardinality{\MCHRONICLE}^{\cardinality{\WEEK}} = 10^{52}$.

Comparing the lower and statistical upper bounds provides an indication of how the different 
algorithms performed on the data used to design them, as {SDDP} and {DADP} are terminated using 
custom stopping criteria that may not be optimal. In particular, for {DADP}, the price-improvement 
algorithm is unlikely to converge to an optimal price vector, since the price decomposition 
process is deterministic. In an ideal world, the lower and upper bounds should be equal.
The statistical gap is computed as the relative difference between the statistical upper bound 
and the lower bound.

\subsubsection{Comparison between policies}
\label{subsubsection:Comparison_between_policies}

We evaluate how the proxy cost-to-go functions obtained with 
the {DADP} method behave when used to design global policies for prospective studies  
for each information structure. To this end, we consider 1000 {RTE} reference uncertainty chronicles 
(see~\S\ref{subsubsection:Uncertainties_modelling}), which account for uncertainties introduced by 
net demand, the availability of dispatchable production, and inflows into storages. Then, we carry 
out simulations over the {RTE} reference uncertainty chronicles using the policies derived from 
 {SDDP}  and {DADP} (hourly, 8-hour, and weekly prices). 
For each method, we use Algorithm~\ref{alg:SimulationHD_multinode},
with the corresponding cost-to-go functions.

The indicators compared in simulation are the \emph{empirical expected operational cost},
the \emph{empirical expected energy not supplied} ({ENS}),
and the \emph{storage trajectories} for France to illustrate some key differences.
The {ENS} is a crucial indicator for prospective studies, as it offers insights 
into the reliability of the energy supply and the potential need for additional resources or
infrastructure investments. Moreover, due to the high penalty set for {ENS} (3000~\euro/MWh),  
{ENS} is a good indicator for assessing the effectiveness of a storage policy,  
as usage values should reflect the value of stored energy in preventing {ENS}.

\subsection{Thirty-node system description}
\label{subsection:Thirty_node_system_description}

In Figure~\ref{fig:marketZonesMap}, we illustrate the market zones
considered in this study, 
along with the representative interconnections between them.
Nodes equipped with long-term storage (modelled as a single aggregated storage) are highlighted in blue.
\begin{figure}[h!]
    \centering
    \includegraphics[width=0.45\textwidth]{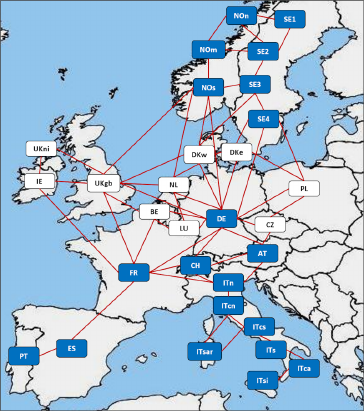}
  \caption{\small Market zones and their interconnections. In blue, the nodes with long-term storage.}
    \label{fig:marketZonesMap}
\end{figure}
The interconnections between nodes are modelled as transport links with specified capacities. 
Note that there are nodes with many connections, as well as nodes with few connections, 
resulting in heterogeneity in the network structure.
In Table~\ref{Table:30_system_summary}, we present a 
summary of the system components aggregated by regions,
including the storage capacity (if any),
the number of clusters of dispatchable units and the total dispatchable capacity installed
(details per node are not disclosed due to confidentiality).
\begin{table}[h!]
    \centering
    \begin{tabular}{lccc}
    \hline
    Region & \begin{tabular}{c} Storage \\ capacity \\ (GWh) \end{tabular}& \begin{tabular}{c}Number of \\ dispatchable \\ clusters\end{tabular} & \begin{tabular}{c} Dispatchable \\ capacity  \\(GW) \end{tabular} \\
    \hline
        Southern Europe	&	34~946	&	66	&	177~240 \\ 
        Northern Europe	&	11~311	&	82	&	154~248 \\
        Scandinavia	    &	121~508	&	44	&	15~209 \\
        British Isles 	&	 -- 	&	38	&	58~355 \\
                                \hline
    \end{tabular}
  \caption{\small  Summary of the 30-node system components aggregated by region. 
    Nodes in Southern Europe: ES, PT, FR, ITca, ITcn, ITcs, ITn, ITs, ITsar and ITsic.
    Nodes in Northern Europe: AT, BE, CH, CZ, DE, LU, NL and PL.
    Nodes in Scandinavia: DKe, DKw, NOm, NOn, NOs, SE1, SE2, SE3, and SE4.
    Nodes in British Isles: IE, UKgb and UKni.
    Details per node are not reported for reasons of data confidentiality.}
    \label{Table:30_system_summary}
\end{table}

In addition to the presence of nodes without storage,
the system exhibits significant heterogeneity in terms of the number of clusters of dispatchable units,
 installed capacities,
 and storage sizes across nodes. 
 The nodes with significantly  large long-term storage facilities are those in Norway, Sweden, 
 France, Spain, and Switzerland.
  This diversity reflects realistic differences in market zones and impacts both
 the complexity and the coordination required in the optimization process.

\begin{remark} \label{remark:Final_storage_level_constraint_30nodes}
    It is common practice in prospective studies to impose a final storage level constraint at the end of 
    the studied timespan,
    typically implemented as a penalty in each nodal optimization problem ($ \FinalCost^{\node}$ 
    in Equations~\eqref{eq:GSFWeeklyHD-1}). 
    This penalty is activated when the storage at 
    the end of the horizon falls below a specified threshold which, in this case, 
    is set to the initial condition~$\stock_{\winfhinf}^{\node}$. 
    Since exceeding the desired threshold does not provide any economic benefit, we do not penalize storage 
    levels above the target.
    This practice is followed in this study, during both the calculation of usage value
    and in the simulation for all four methods studied.

    {The penalty is modeled as piecewise linear function with two segments:}
    \begin{align}
       \FinalCost^{\node}\bp{\state^{\node}}=
        \begin{cases}
            150~\text{\euro/MWh} \times \bp{\stock_{\winfhinf}^{\node} - \state^{\node}}
            & \text{if } \state^{\node} < \stock_{\winfhinf}^{\node} \eqfinv
            \\
              0
            & \text{if } \state^{\node} \geq\stock_{\winfhinf}^{\node} \eqfinp                       
        \end{cases}
    \end{align}
    The numerical value of 150~\euro/MWh for the penalty rate was determined
    so that it is preferable to use thermal generation rather than depleting storage below the initial level
    at the end of the year, 
    but to still keep the penalty moderate compared to the high cost of energy not served (3000~\euro/MWh).
\end{remark} 

\subsection{Numerical results for the thirty-node system}
\label{subsection:Numerical_results_for_the_thirty_node_system}

In this~\S\ref{subsection:Numerical_results_for_the_thirty_node_system}, 
we present the numerical results obtained for the 30-node system.
We compare the {DADP} method with the three price aggregation schemes
(hourly, 8-hour, and weekly prices) against the {SDDP} method. 

The size of the price decomposition process varies depending on the price aggregation scheme. 
Specifically, the dimensions of the price decomposition process are 262~020 for hourly prices,
32~760 for 8-hour prices, and 1~560 for weekly prices.

\subsubsection{Algorithms assessment}
\label{subsubsection:Computing_times_and_lower_bounds_30nodesHD}
We first present the assessment of the different algorithms
in terms of lower bounds, computation times, statistical upper bounds, and gaps.

\phantomsection
\subsubsubsection{Lower bounds and computing times}
In Table~\ref{table:30nodesHD_computingTimes_and_Iterations}, we report the computation times and lower bounds
for each method.
\begin{table}[h!]
    \centering
    \begin{tabular}{lcccc}
    \hline
                Method  & Iterations & Lin. search & Time (h) & Lower Bound (B\euro) \\
    \hline
\hline
                {SDDP}\footnotemark[1]                 & 250 &   --        & 24.0                &    45.175 \\
                \hline
                {DADP} hourly prices\footnotemark[1] & 72  &    122      &  24.0               &  44.696 \\
                {DADP} 8-hour prices   & 128  &    141     &  11.7               &    44.924 \\
                {DADP} weekly prices    & 112  &    278     &  15.0               &    44.810 \\
        \hline 
    \end{tabular} 
  \caption{\small Number of iterations, number of linear searches, computation times and lower
    bound for different solution methods in the 30-node system.
        }
        \footnotetext[1]{{SDDP} and {DADP} hourly prices were both terminated due to the time limit.}
    \label{table:30nodesHD_computingTimes_and_Iterations}
\end{table} 
Note that the multivariate {SDDP} method and the univariate {DADP}
method with hourly prices were terminated due to the 24-hour time limit, 
rather than because they met the stopping criterion.
However, the lower bound achieved by {SDDP} after 24 hours is the highest among all methods.
The {DADP} method with 8-hour prices is the fastest,  completing in 11.7 hours,
while the {DADP} method with weekly prices takes 15 hours. 
Both approaches produce lower bounds slightly below {SDDP}'s, differing by 0.55\% and 0.83\%, respectively.

When comparing the {DADP} methods in Table~\ref{table:30nodesHD_computingTimes_and_Iterations},
we observe that the hourly prices scheme takes the longest,
with the lowest number of iterations (72) and linear searches (122), and, 
as a consequence, it does not 
reach the stopping criterion. 
Consequently, the quality of the lower bound is reduced.
This limitation could be attributed to the thread-based parallelization currently implemented in our code,
which becomes inefficient when the price decomposition process has a large dimension.
To address this issue in the hourly price scheme, a distributed computing framework would be more suitable.
Since the nodal problems are independent, they can be solved in parallel on separate processors.
A message-passing interface (MPI) can be used to broadcast the hourly prices to each node and to gather the
associated nodal costs and gradients.

\subsubsubsection{Statistical upper bounds and gaps}
In Table~\ref{table:30nodesHD_statisticalUpperBounds} we present the statistical upper bound
(see Remark~\ref{remark:gap}) obtained by simulation (see Algorithm~\ref{alg:SimulationHD_multinode})
across 1000 Monte Carlo scenarios sampled
according to the empirical product probability in~\eqref{eq:weekly_probability_Prbt_multinode_Mchronicles}. 
The results use the policies derived from each method and include the associated 95\% confidence intervals.
For completeness, we also report each method's lower bound from Table~\ref{table:30nodesHD_computingTimes_and_Iterations}.
\begin{table}[h!]
    \centering
    \begin{tabular}{lccc}
        \hline
                    Method ({HD}) 
                    & \begin{tabular}{c}Lower Bound\\(B\euro)\end{tabular}  
                    & \begin{tabular}{c}Statistical \\ upper bound\\(B\euro)\end{tabular}
                    &\begin{tabular}{c}Statistical\\gap\end{tabular} \\
        \hline
        \hline
                {SDP}                                    & N/A & N/A   &  N/A \\
                {SDDP}\footnotemark[1]                  & 45.175 & { 45.919 ± 0.057}  & 1.65\% \\
                \hline
                {DADP} hourly  prices\footnotemark[1]   & 44.696 & 45.891 ± 0.087   & 2.67\% \\
                {DADP} 8-hour prices                    & 44.924 & 45.477 ± 0.060   & 1.23\% \\
                {DADP} weekly prices                    & 44.810 & 45.567 ± 0.062   & 1.69\% \\
        \hline
    \end{tabular} 
  \caption{\small Lower bound, statistical upper bound (with confidence intervals) and statistical gap 
    (see Remark~\ref{remark:gap}) for different resolution methods.
     {SDDP} and {DADP} hourly prices were both terminated due to the time limit. 
    }
    \label{table:30nodesHD_statisticalUpperBounds}
  \footnotetext[1]{{SDDP} and {DADP} hourly prices were both terminated due to the time limit.}
\end{table}

The statistical upper bounds obtained with the four methods are relatively close, 
with a maximum difference of 0.442 B\euro, which corresponds to less than 1\% of the best upper bound.

The {DADP} method with hourly prices, which did not reach the stopping criterion, 
yields the highest statistical upper bound among the {DADP} methods, 
with a gap of 2.67\% relative to its lower bound.
The gap obtained with {SDDP} is smaller than that of {DADP} with hourly prices, 
despite {SDDP} also being terminated due to the time limit.
The best statistical gap is achieved by the {DADP} method with 8-hour prices, 
with a gap of 1.23\% relative to its lower bound.
It is followed by {SDDP} and {DADP} with weekly prices,
showing gaps of 1.65\% and 1.69\%, respectively.
Although the {DADP} method with hourly prices achieves a 2.67\% gap,
it is 
roughly twice as large as the gap of the {DADP} method with 8-hour prices.
This suggests that the hourly price scheme struggles to handle the increased complexity of the larger system, 
likely due to limitations in the current parallelization strategy.

Overall, the gaps reported for the 30-node system in Table~\ref{table:30nodesHD_statisticalUpperBounds} 
are tight. The increased complexity and diversified cost structure 
of the large system provide a smooth optimization landscape and facilitate the approximation
of the cost-to-go functions.
Nevertheless, to draw more definitive conclusions, a larger number of Monte Carlo scenarios would be required, 
as the current analysis is based on only $10^{3}$ scenarios out of the $10^{52}$ possible ones.

\subsubsection{Comparison between policies across the {RTE} reference uncertainty chronicles}
\label{subsection:Policies_comparison_30nodesHD}
It is important to emphasize that the quality of usage values cannot be fully 
assessed in isolation. Moreover, in a complex system with 20 storage units,
assessing the impact of individual usage values on the overall system performance is not straightforward.
Therefore, we focus on assessing the performance of policies constructed from 
these usage values by simulating system operation across 1000 {RTE} reference uncertainty chronicles
 and comparing the resulting 
operational costs and performance indicators with those obtained with policies derived from {SDDP}.

In Table~\ref{table:30nodesHD_opCostENS}, we summarize the results of these simulations,
reporting the empirical expected operational cost (in B€) and the expected energy not supplied ({ENS} in GWh).
These empirical expectations are computed using the empirical probability~$\widetilde{\prbt}$
 in Equation~\eqref{eq:year_probability_Prbt_multinode_Pchronicles}.
\begin{table}[h!]
        \centering
        \begin{tabular}{lcc}
                \hline
            Method                     & Operational cost (B\euro) &   {ENS} (GWh) \\
            \hline
\hline
            {SDDP} \footnotemark[1]                             & 45.628                   & 99.59  \\
            \hline
            {DADP} hourly prices\footnotemark[1]            & 46.530                      &  489.54\\
            {DADP} 8-hour prices             &  45.489                         &  124.93  \\
            {DADP} weekly prices              &  45.645                      &   160.44  \\
            \hline
        \end{tabular}
      \caption{\small {Expected operational costs, and expected energy not supplied ({ENS})
            for four different resolution methods. 
            {SDDP} and {DADP} hourly prices were both terminated due to the time limit. 
            }
        }\label{table:30nodesHD_opCostENS}
        \footnotetext[1]{{SDDP} and {DADP} hourly prices were both terminated due to the time limit.}
    \end{table}
The {DADP} method with hourly prices that did not reach the stopping criterion
yields the highest operational cost.
{DADP} with 8-hour prices achieves an empirical expected
operational cost that is slightly lower than that of {SDDP} (45.489 B\euro vs. 45.628 B\euro),
with usage values computed in less than half the time (11.7 hours vs. 24 hours).
The mean {ENS} for {DADP} with 8-hour prices is 25\% higher than that of {SDDP} (124.93 GWh vs. 99.59 GWh),
but it does not affect the mean operational cost.

The operational costs reported in Table~\ref{table:30nodesHD_opCostENS} consist of two components:
the thermal generation cost and the energy not supplied ({ENS}) cost.
Table~\ref{table:Thermal_ENS_costs_30HD} reports the expected thermal generation cost, the expected 
{ENS} cost, and the expected operational cost for the four methods studied.
\begin{table}[h!]
        \centering
        \begin{tabular}{lccc}
                \hline
            Method                     & Thermal cost (B\euro) & {ENS} cost (B\euro)    & Operational cost (B\euro) \\
            \hline
\hline
            {SDDP}                              & 45.329            & 0.299       & 45.628  \\
            \hline
            {DADP} hourly prices               &  45.061            & 1.469       &  46.530\\
            {DADP} 8-hour prices               &  45.114            & 0.375       &  45.489  \\
            {DADP} weekly prices               &  45.164            & 0.481       &  45.645  \\
            \hline
        \end{tabular}
      \caption{\small {Thermal, {ENS} and operational cost, for four different resolution methods.
        }
        }\label{table:Thermal_ENS_costs_30HD}
    \end{table}
    When inspecting the thermal generation and {ENS} costs reported in Table~\ref{table:Thermal_ENS_costs_30HD},
    we observe that {DADP} with 8-hour prices yields a lower thermal cost than {SDDP} but a higher {ENS} cost.
    This results in a slightly lower total operational cost for the {DADP} 8-hour scheme, indicating a different trade-off
    between thermal generation and {ENS} in the resulting policies.

\subsubsubsection{Costs}
To further analyze the performance of these methods, Figure~\ref{Figure:costDistribution30HD}
presents the operational cost distributions
resulting from the simulation across the {RTE} reference uncertainty chronicles.
\begin{figure}[h!]
    \centering
    \begin{subfigure}{0.4\textwidth}
           \centering
                \includegraphics[width=\textwidth]{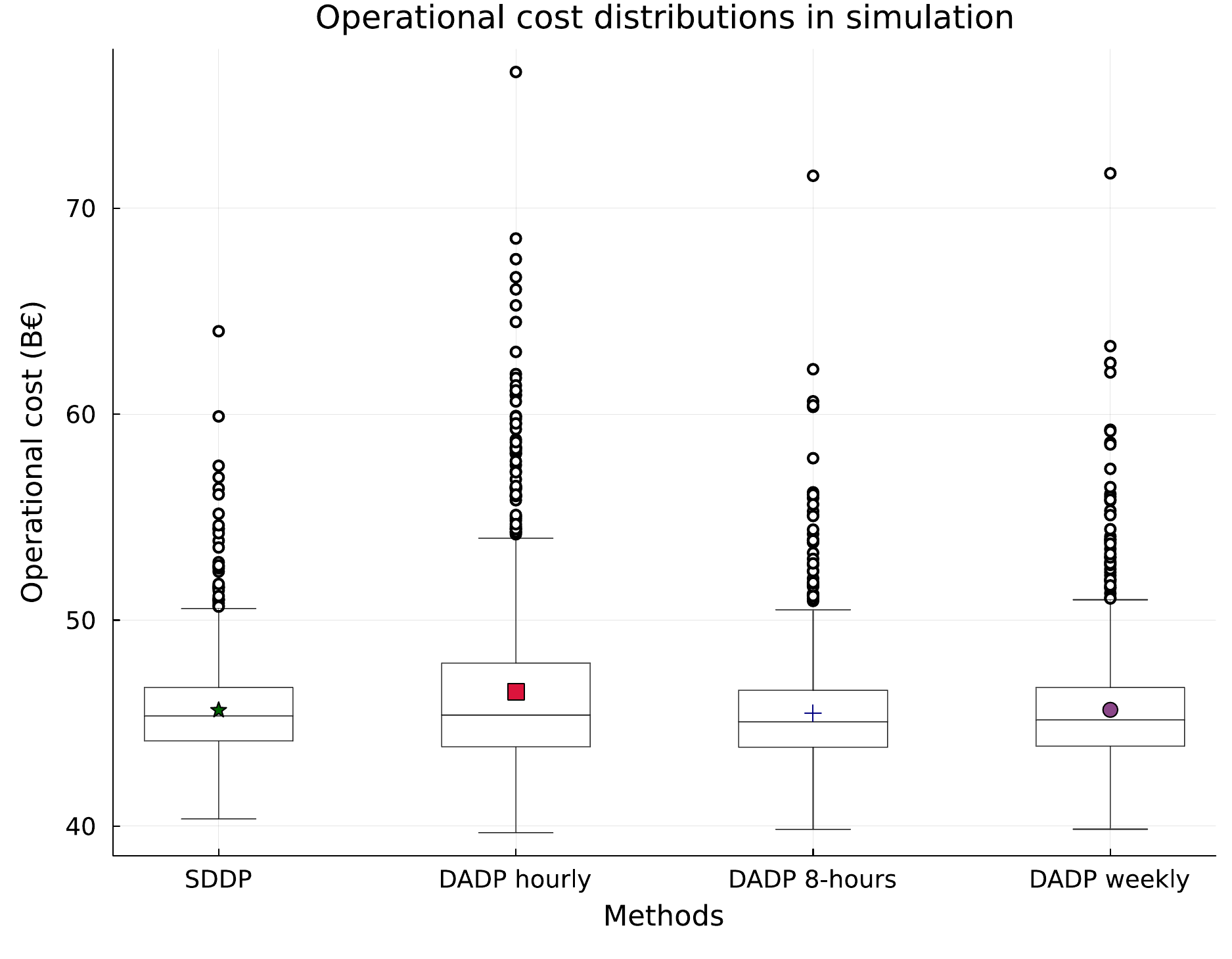}
              \caption{\small \footnotesize{Cost distributions with outliers}
                \label{fig:costDistributionOutliers30HD}}
        \end{subfigure}
        \begin{subfigure}{0.4\textwidth}
            \centering
            \includegraphics[width=\textwidth]{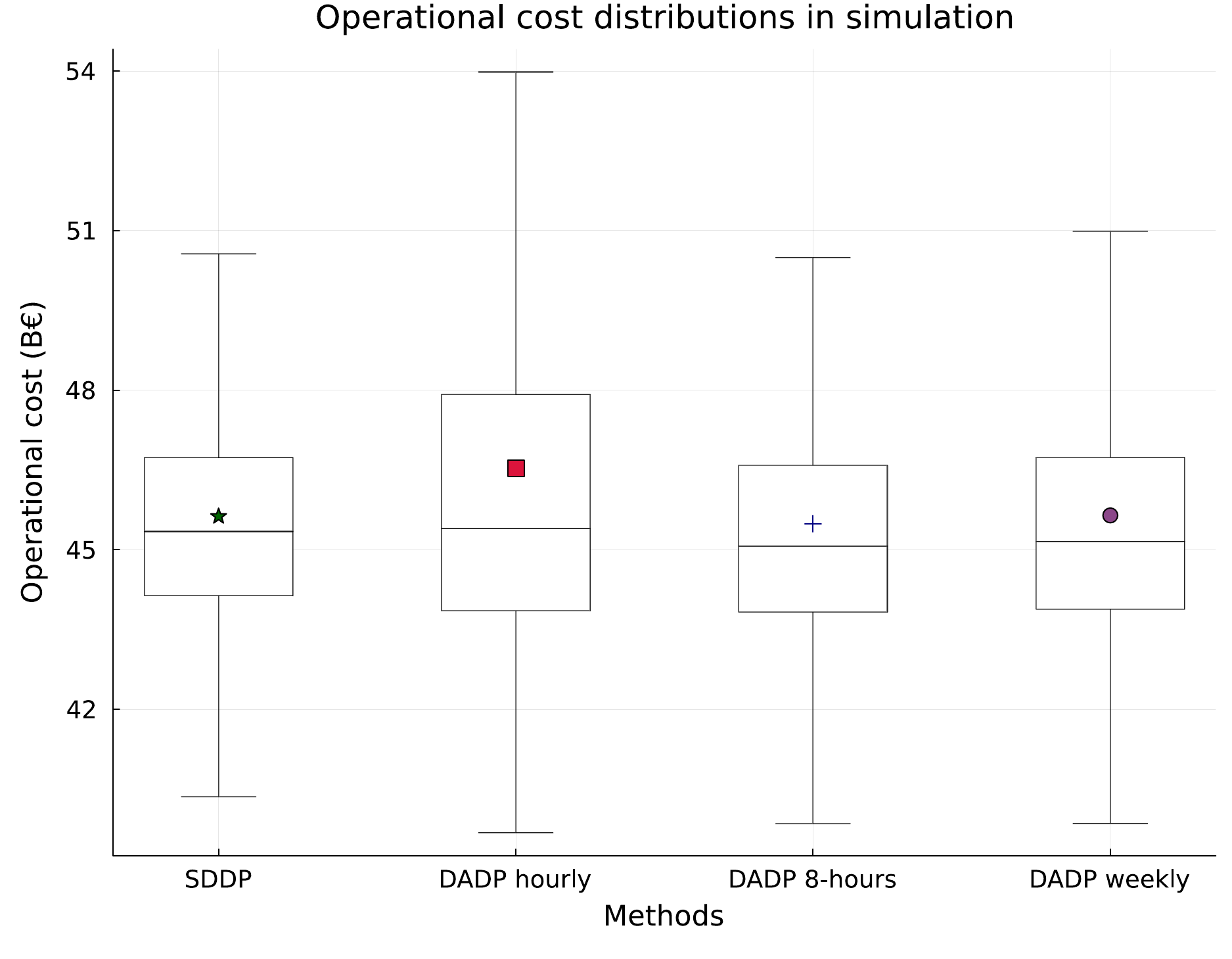}
          \caption{\small \footnotesize{Cost distributions without outliers}
            \label{fig:costDistributionNoOutliers30HD}}
        \end{subfigure}
        \includegraphics[width=0.8\textwidth]{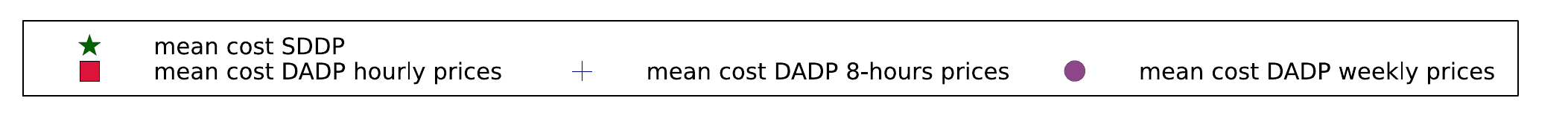}
      \caption{\small Operational cost distributions
        for the 30-node system under the weekly hazard-decision information structure, including the mean cost value for each method.
        The distributions are reported with and without outliers (\ref{fig:costDistributionOutliers30HD} and~\ref{fig:costDistributionNoOutliers30HD} respectively).}
        \label{Figure:costDistribution30HD} 
    \end{figure} 

\begin{remark}
  A boxplot (also known as a box-and-whisker plot) is a graphical tool used to summarize
  the distribution of a dataset. It displays the median, the first and third quartiles
  (Q1 and Q3), and potential outliers. The box represents the interquartile range 
  (IQR), which contains the middle 50\% of the data. 
  The line inside the box indicates the median. 
  The whiskers extend to the smallest and largest values within 1.5 
  times the IQR from Q1 and Q3, respectively. 
  Data points outside this range are considered outliers and are plotted individually.
  Boxplots provide a concise visual summary of the central tendency, spread,
  and presence of outliers in the data.
\end{remark}

For the four methods studied, Figures~\ref{fig:costDistributionOutliers30HD} and~\ref{fig:costDistributionNoOutliers30HD}
show the operational cost distributions with and without outliers, respectively, for the 30-node system.
The distributions include the mean cost value (markers) and the lower bound (dashed lines) for each method.
As expected, the {DADP} method with hourly prices exhibits a wider distribution of operational costs,
with a few scenarios leading to very high costs, as reflected in the outliers.
This outcome is consistent with expectations, given that this method did not reach the stopping criterion.

Despite this behaviour for {DADP} with hourly prices, the other two {DADP} methods (8-hour and weekly prices)
exhibit operational cost distributions that are quite similar to that of the {SDDP} method,
with lower median values and only a few outliers leading to higher costs.
Regarding the number of outliers observed in the operational cost distributions, 
we obtain 32~outliers for {SDDP}, 38~outliers for {DADP} with 8-hour prices, and 43~outliers for {DADP} with weekly prices.
Notably, the operational cost distributions for {SDDP} and {DADP} with 8-hour prices are very similar.

\subsubsubsection{Energy not supplied ({ENS})}
In Figure~\ref{Figure:ensDistribution30HD}, we present the {ENS} distributions
resulting from the simulation across the {RTE} reference uncertainty chronicles, for the four methods studied.

\begin{figure}[h!]       
    \centering
    \begin{subfigure}{0.4\textwidth}
           \centering
                \includegraphics[width=\textwidth]{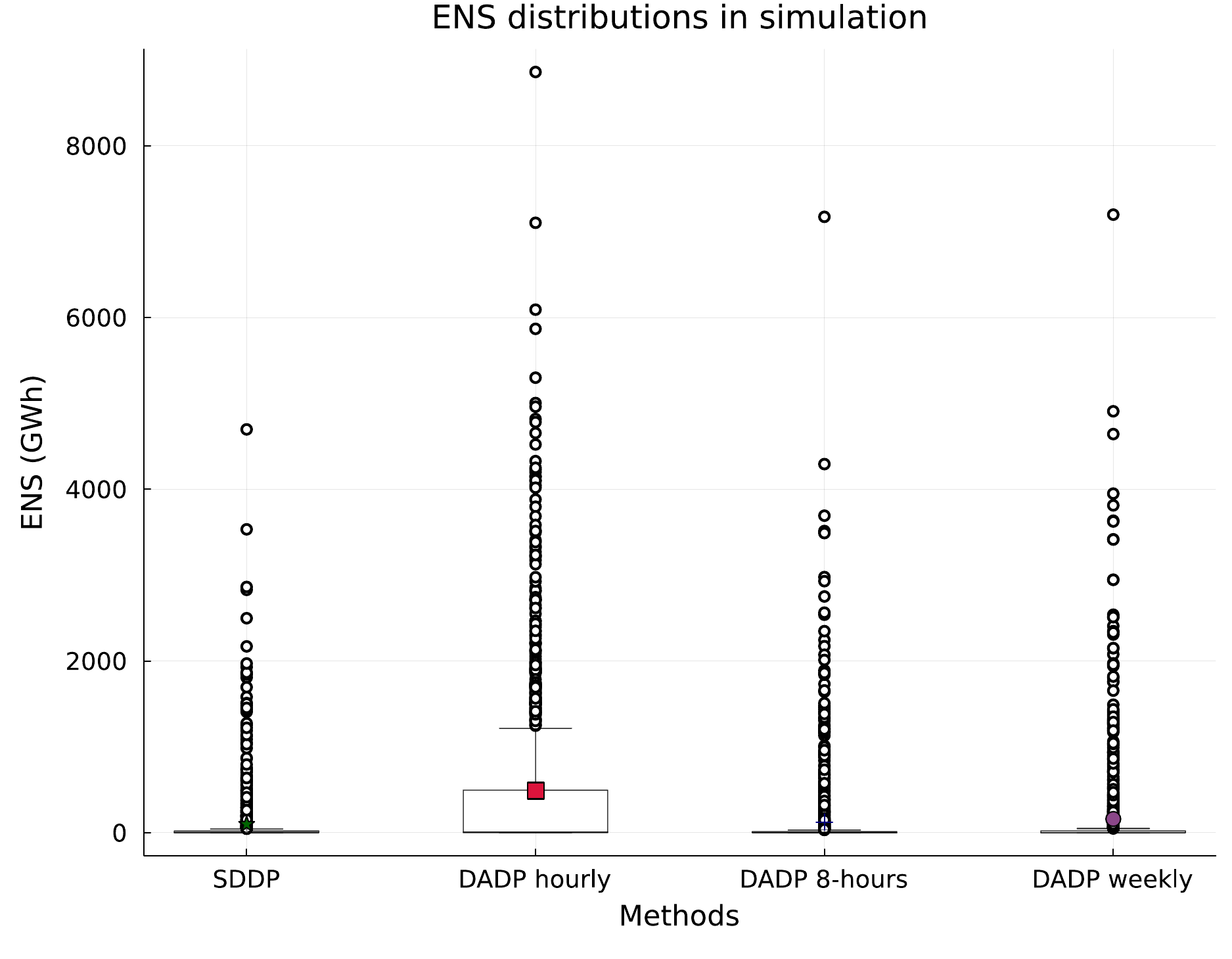}
              \caption{\small \footnotesize{{ENS} distributions with outliers}
                \label{fig:ensDistributionOutliers30HD}}
        \end{subfigure}
        \begin{subfigure}{0.4\textwidth}
            \centering
            \includegraphics[width=\textwidth]{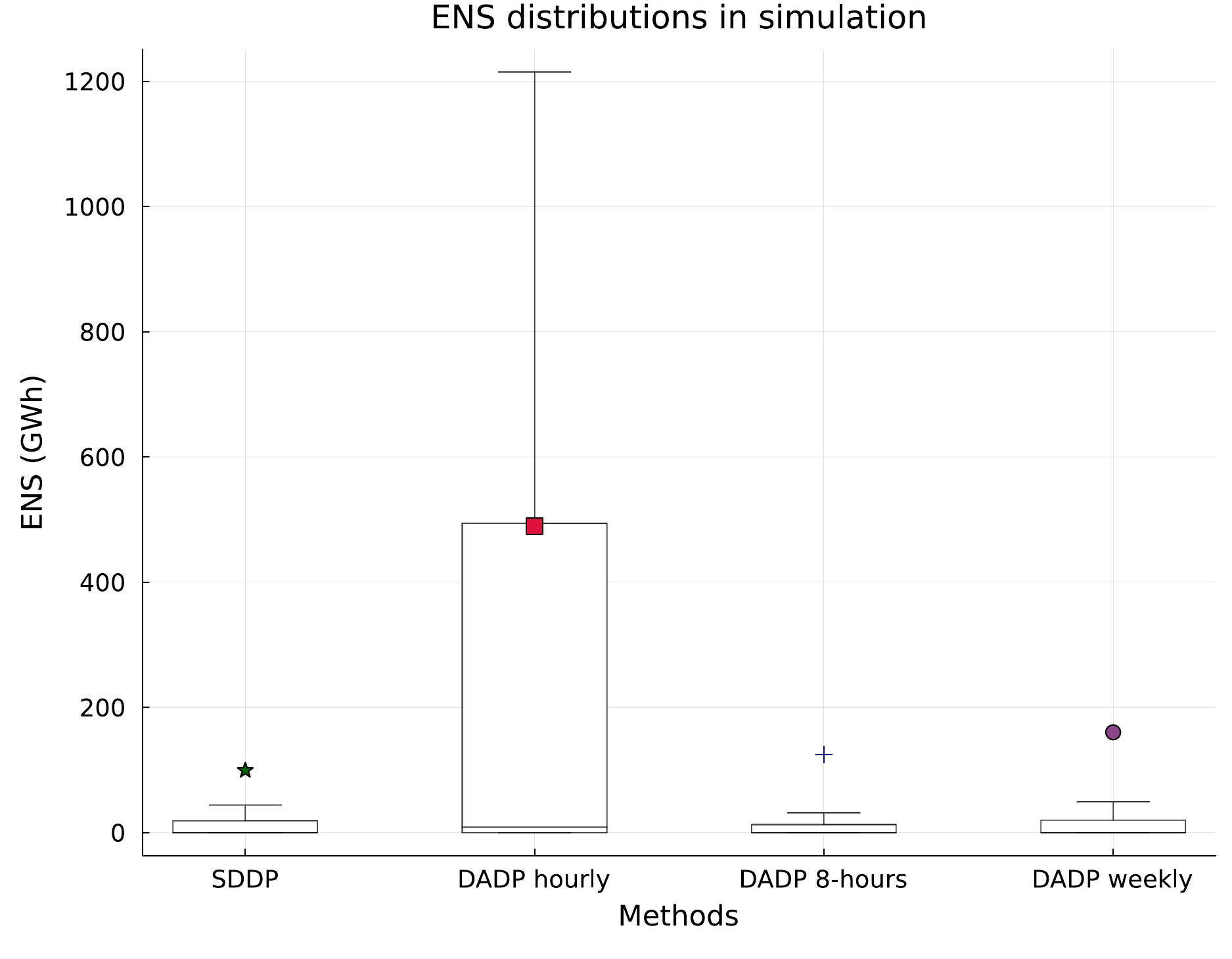}
          \caption{\small \footnotesize{{ENS} distributions without outliers}
            \label{fig:ensDistributionNoOutliers30HD}}
        \end{subfigure}
        \includegraphics[width=0.8\textwidth]{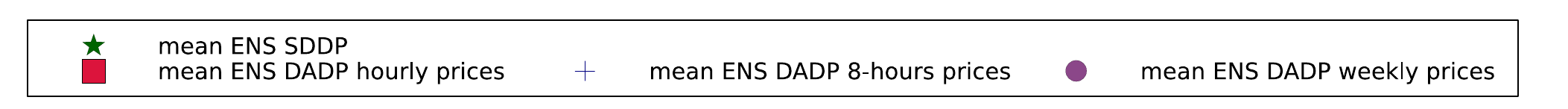}
      \caption{\small {ENS} distributions
        for the 30-node system,
        including the mean {ENS} value for each method.
         The distributions are reported with and without outliers (\ref{fig:ensDistributionOutliers30HD} and~\ref{fig:ensDistributionNoOutliers30HD} respectively).}
        \label{Figure:ensDistribution30HD}
    \end{figure}
Figures~\ref{fig:ensDistributionOutliers30HD} and~\ref{fig:ensDistributionNoOutliers30HD}
show the {ENS} distributions with and without outliers, respectively, for the 30-node system.
The distributions include the mean {ENS} value (markers) for each method.
Again, the {DADP} method with hourly prices exhibits a wider distribution of {ENS},
with a few scenarios leading to very high {ENS} values, as reflected in the outliers.
Even excluding outliers, the {DADP} method with hourly prices still shows a wider {ENS} distribution
compared with the other methods.

For the number of outliers observed in the {ENS} distributions, we obtain 197~outliers for {SDDP}, 
149~outliers for {DADP} with hourly prices, 186~outliers for {DADP} with 8-hour prices, and 168~outliers for {DADP} 
with weekly prices. As {DADP} with hourly prices exhibits a wider {ENS} distribution, it is not 
surprising that it has the lowest number of outliers. In terms of operational cost distributions, 
the {DADP} method with 8-hour prices exhibits an {ENS} distribution comparable to that of {SDDP}, 
with fewer outliers; however, these outliers are more widely dispersed than those observed for {SDDP}.

Due to the large dispersion of the {ENS} values for the {DADP} method with hourly prices,
 it is hard to compare the distributions 
of the other three methods (SDDP, {DADP} with 8-hour and weekly prices). Therefore, we present in 
Figure~\ref{Figure:ensDistribution30HD_noHourly} the {ENS} distributions without outliers, 
excluding the {DADP} method with hourly prices.
\begin{figure}[h!]       
    \centering
            \centering
            \includegraphics[width=0.5\textwidth]{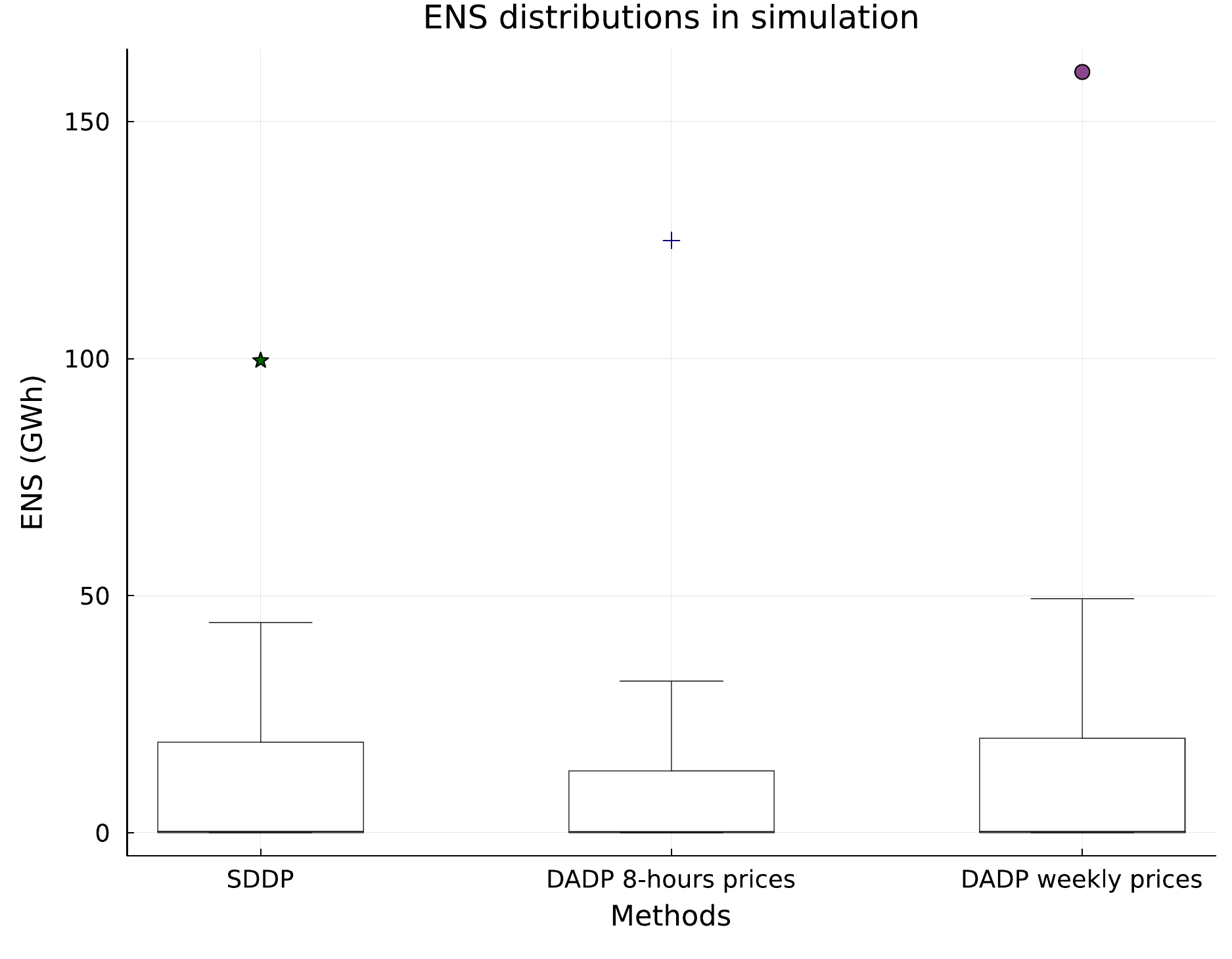}
        \includegraphics[width=0.5\textwidth]{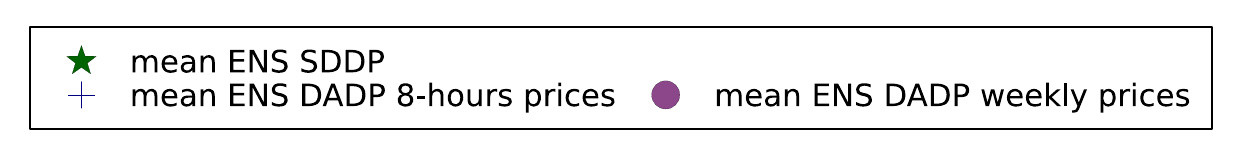}
      \caption{\small {ENS} distributions without outliers
        for the 30-node system under the weekly hazard-decision information structure,
        including the mean {ENS} value for {SDDP}, {DADP} with 8-hour prices, and {DADP} with weekly prices.}
        \label{Figure:ensDistribution30HD_noHourly}
    \end{figure}

We observe that the {ENS} distributions for the {SDDP} method and the {DADP} methods with 8-hour and weekly prices are quite 
similar, with the 8-hour prices scheme being the one with less dispersion of {ENS} values.
However, when observing the mean {ENS} values (markers in Figure~\ref{Figure:ensDistribution30HD_noHourly}),
as reported in Table~\ref{table:30nodesHD_opCostENS}, we see that the {DADP} method with 8-hour prices
leads to a mean {ENS} of 124.93 GWh, which is higher than the mean {ENS} of 99.59 GWh obtained with the {SDDP} method.
We observe in Figure~\ref{fig:ensDistributionNoOutliers30HD} that the three distributions are close to the zero value,
 indicating that in most scenarios the {ENS} is low.

\subsubsubsection{Storage trajectories}

Finally, we present the annual
storage trajectories with 52 values of storage levels for France 
to illustrate some of the differences in the storage operation induced by the different policies
(see Algorithm~\ref{alg:SimulationHD_multinode} for details on how trajectories are computed).

In Figure~\ref{Figure:StorageTrajectories_30HD1}, we present the storage trajectories 
 across the {RTE} reference uncertainty chronicles,  using the policies derived from
 {SDDP} (green), {DADP} with hourly prices (red), {DADP} with 8-hour prices (blue), and {DADP} with weekly prices
    (purple).
\begin{figure}[h!]
    \centering
    \begin{subfigure}[c]{0.45\textwidth}
        \includegraphics[width=\columnwidth]{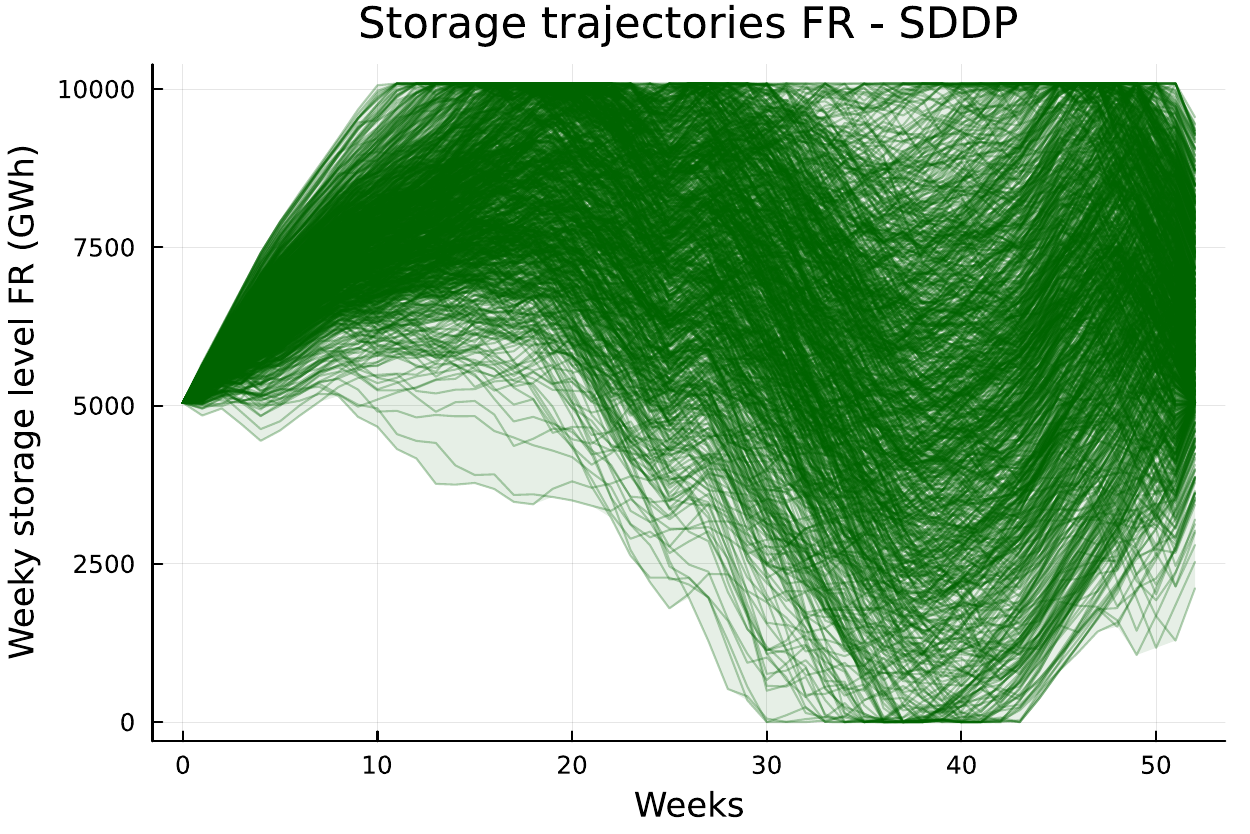}
      \caption{\small \footnotesize{Storage trajectories for the {SDDP} method}}
        \label{fig:StorageTrajectories_SDDP_30HD1}
    \end{subfigure}
    \hspace{0.02\textwidth}
    \begin{subfigure}[c]{0.45\textwidth}
        \includegraphics[width=\columnwidth]{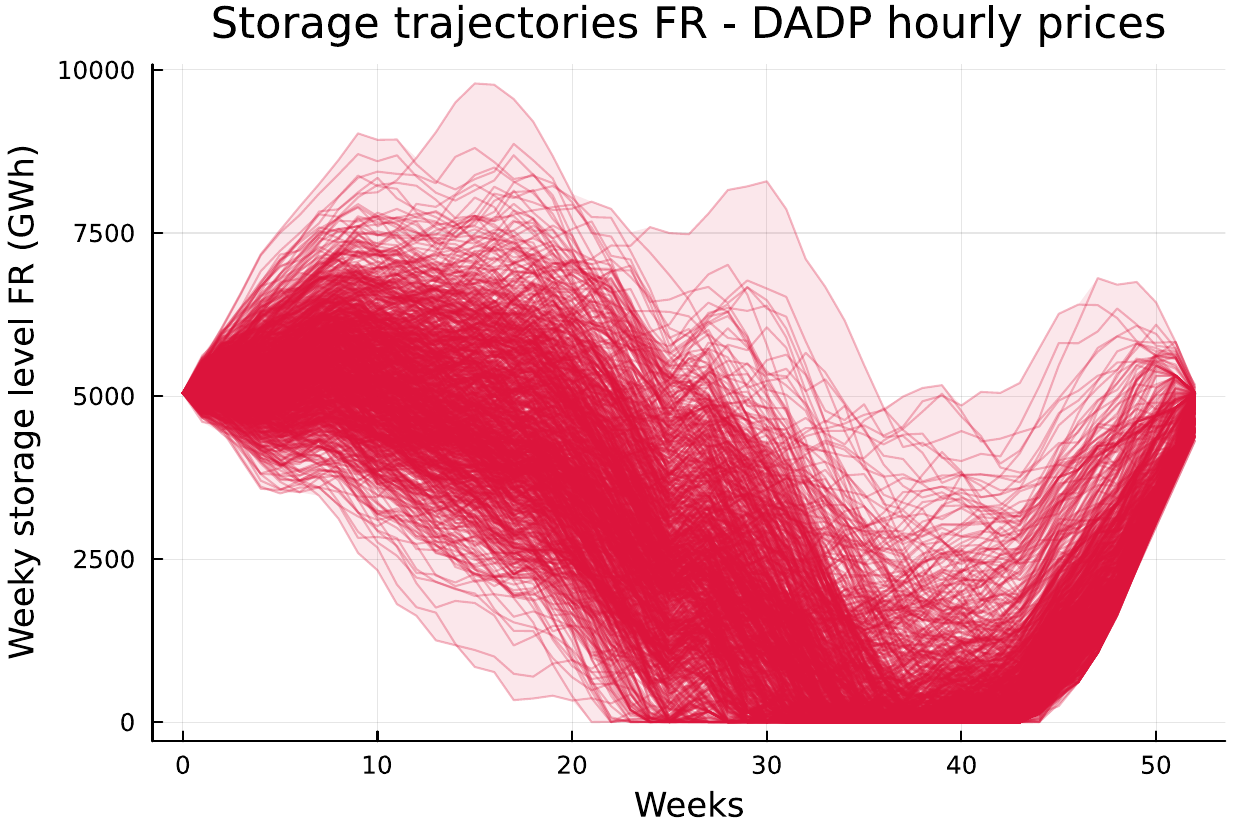}
      \caption{\small \footnotesize{Storage trajectories for the {DADP} method with hourly prices}}
        \label{fig:StorageTrajectories_DADPPhourly_30HD1}
    \end{subfigure}
    \begin{subfigure}[c]{0.45\textwidth}
        \includegraphics[width=\columnwidth]{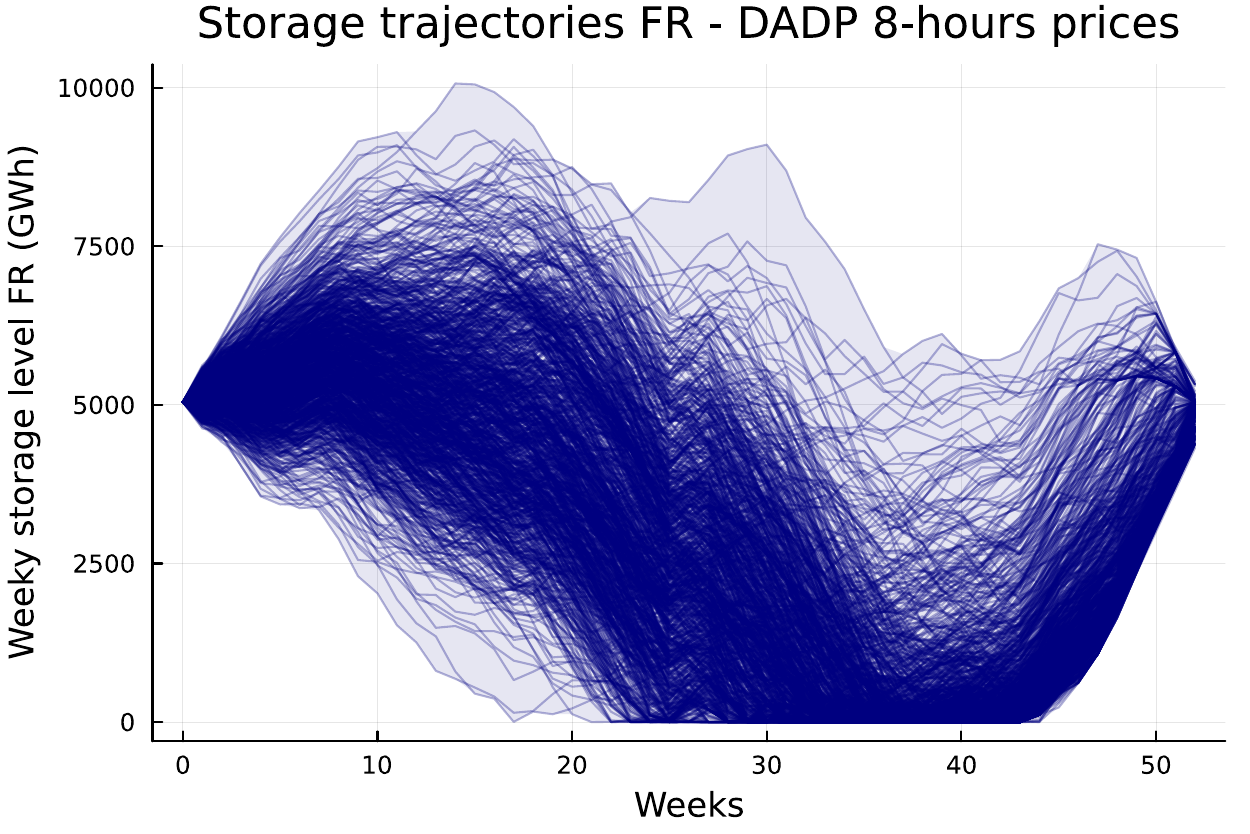}
      \caption{\small \footnotesize{Storage trajectories for the {DADP} method with 8-hour prices}}
        \label{fig:StorageTrajectories_DADP8Hours_30HD1}
    \end{subfigure}
    \hspace{0.02\textwidth}
    \begin{subfigure}[c]{0.5\textwidth}
        \includegraphics[width=\columnwidth]{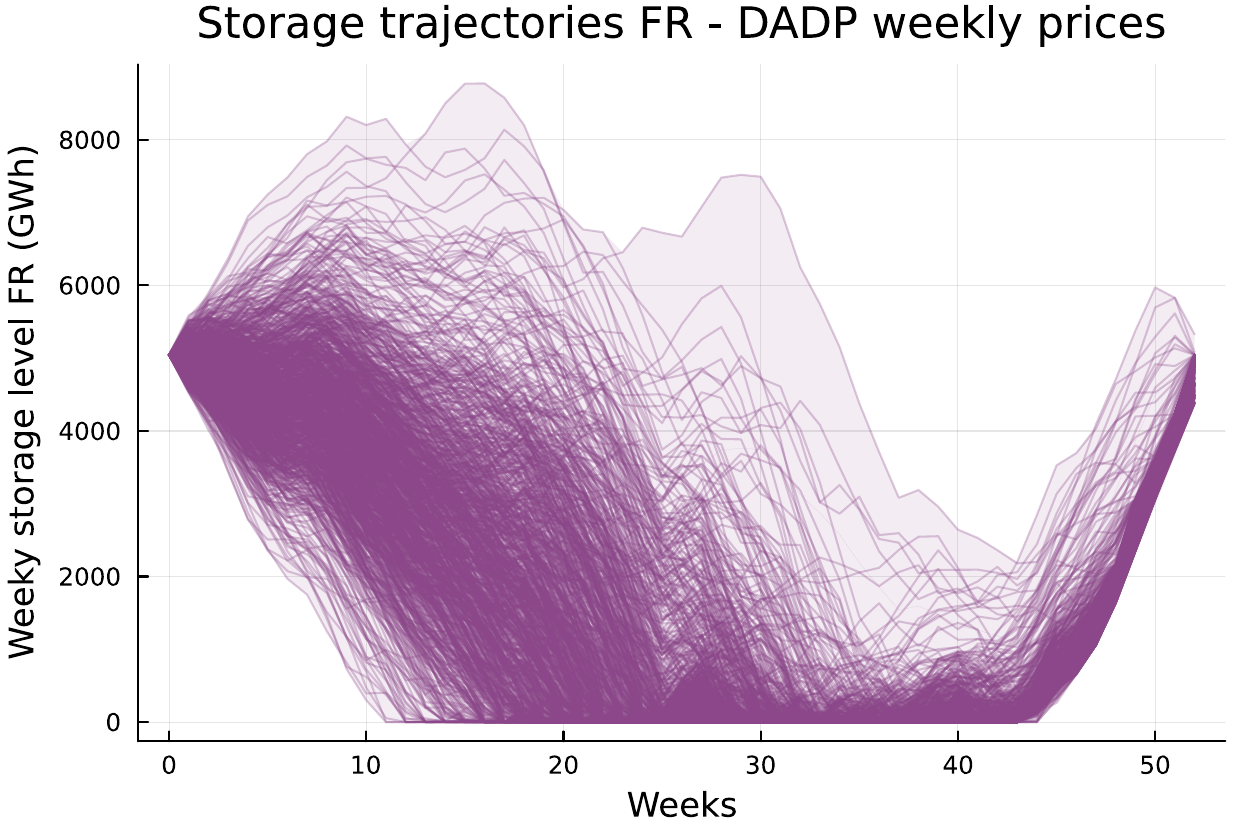}
      \caption{\small \footnotesize{Storage trajectories for the {DADP} method with weekly prices}}
        \label{fig:StorageTrajectories_DADPWeekly_30HD1}
    \end{subfigure}
  \caption{\small Storage trajectories for France,
     across {RTE} reference  scenarios of uncertainty chronicles,
     using the policies derived from
     {SDDP}, {DADP} with hourly prices, {DADP} with 8-hour prices, and {DADP} with weekly prices (top to bottom)}
     \label{Figure:StorageTrajectories_30HD1}
\end{figure}
Each subplot in Figure~\ref{Figure:StorageTrajectories_30HD1} reports 1000~trajectories of storage (in GWh).
The shaded region illustrates the range of storage levels observed across all simulated scenarios,
representing the minimum and maximum values attained during the simulation period. 
This provides a visual summary of the variability in storage trajectories resulting from different policies. 
Individual lines within the region correspond to specific storage level trajectories for each {RTE} reference
 uncertainty chronicle.

As presented in Remark~\ref{remark:Final_storage_level_constraint_30nodes},
a final penalty (150\euro/MWh) 
in the event of a final storage level below the target (50\% of the storage capacity) was incorporated both
during the usage value calculation and in 
the simulation for all four methods studied.
We observe in Figure~\ref{Figure:StorageTrajectories_30HD1}
that the impact of this penalty varies significantly across the different methods,
with the number of scenarios failing to meet the target differing from one approach to another. 
The reasons why a simulation scenario fails to meet the final storage target can be diverse, 
including inaccurate cost-to-go functions in the last weeks of the year, 
which makes it impossible to achieve 
the target at the end due to previous decisions. In other cases, the inflows are higher or lower than 
expected, and even with accurate usage values, the target cannot be attained. 

We observe a mayor dispersion in the final storage levels for the {SDDP} method
(Figure~\ref{fig:StorageTrajectories_SDDP_30HD1}) compared with the three {DADP} methods
(Figures~\ref{fig:StorageTrajectories_DADPPhourly_30HD1},~\ref{fig:StorageTrajectories_DADP8Hours_30HD1},
and~\ref{fig:StorageTrajectories_DADPWeekly_30HD1}).
As {SDDP} only explores a subset of relevant points in the three-dimensional state space, it may
 miss combinations of storage levels that are critical to correctly incorporating the final penalty 
 into usage values, leading to higher dispersion in the final storage levels in simulations.  
 By contrast, {DADP} methods do not explore the three-dimensional state space; however, 
 they perform an exhaustive exploration of the local one-dimensional state space, 
 allowing them to capture the effects of the penalty locally. This results in a lower dispersion 
 of final storage levels compared with {SDDP}.
This effect is also observed for other large storage nodes in the system, 
such as Norway and Sweden (not shown here).

In general, {SDDP} storage trajectories are higher than those of {DADP},
which is consistent with the lower {ENS} observed for {SDDP} in Table~\ref{table:30nodesHD_opCostENS},
as energy stored helps to avoid {ENS}. However, this does not translate into lower operational costs,
as the mean operational cost for {DADP} with 8-hour prices is slightly lower than that of {SDDP}.
This observation is consistent with the results reported in Table~\ref{table:Thermal_ENS_costs_30HD},
which show that the methods manage thermal generation, storage usage, and {ENS} differently while
producing comparable total costs. This suggests the existence of multiple solutions that, 
although significantly different in terms of trajectory, 
are all nearly optimal for managing the system.

\section{Conclusion and perspectives}
\label{section:Conclusion_and_perspectives}
 We have presented a numerical study of the {DADP} method on a 30-node
 electrical energy system, benchmarking it against the {SDDP} approach under the weekly hazard-decision
 information structure. The results have shown that {DADP} and SDDP are comparable
 in terms of lower bounds and simulation performance, with {DADP} being faster.

 We have observed that the {DADP} method with hourly prices did not meet the stopping criterion
 within the given time limit, due to the inefficiency of the thread-based parallelization currently 
 implemented in our code, highlighting the need to explore distributed computing frameworks in future work. 
 Therefore, the numerical results for this method have been presented for completeness; however, it is not 
 possible to draw definitive conclusions about its performance.

The focus has been placed on evaluating the performance of the policies constructed from
these usage values by simulating system operation across {RTE} reference scenarios of
uncertainty chronicles, and comparing the resulting
operational costs and performance indicators with those achieved using policies derived from {SDDP}.

Overall, {DADP} proves to be a promising approach for computing usage values in large-scale 
multinode energy systems,
as it avoids the curse of dimensionality by allowing for parallel resolution of nodal problems.
Aggregating prices over 8-hour blocks offers a favourable trade-off between 
solution quality and computational efficiency, even yielding a better mean operational cost than {SDDP}.
This makes it a practical choice when memory or 
computational resources limit the feasibility of hourly price aggregation.

The main advantage of univariate {DADP} methods over multivariate {SDDP} or {SDP} 
(when computationally feasible)
is that the former make the simulation algorithm faster, as the cost-to-go functions are univariate and more 
manageable when modelled as piecewise linear functions. 
In the multivariate case, we consider multivariate piecewise linear approximations. 
In contrast, the univariate case involves the sum of univariate piecewise linear approximations, 
which require fewer constraints for modelling.

An important observation made during the numerical studies
for DADP,
is that the algorithm to improve the price decomposition process
(see~\S\ref{subsubsection:How_to_improve_the_price_decomposition_process})
behaves differently in small and large systems. Its performance depends
not only on the number of nodes but also on the system cost structure:
the more continuous the cost functions are, the better the algorithm performs.
The behaviour also changes with the information structure. In deterministic
tests (not reported here), the algorithm performed poorly; but it improved
in the stochastic case where randomness produces a smoothing effect. 

Finally, in the large multi-node system
(with realistic cost functions),
the algorithm performed satisfactorily, producing good lower bounds in
reasonable computational times.

\newcommand{\noopsort}[1]{} \ifx\undefined\allcaps\def\allcaps#1{#1}\fi

\end{document}